\documentclass[11pt]{article}
\usepackage{verbatim} 
\usepackage{graphicx}
\usepackage[utf8]{inputenc}
\usepackage{amsmath,amsfonts,amssymb,amscd,amsthm,txfonts,hyperref,float,tikz}
\usepackage{xcolor}
\newtheorem{theorem}{Theorem}[section]
\newtheorem{proposition}[theorem]{Proposition}
\newtheorem{lemma}[theorem]{Lemma}
\newtheorem{corollary}[theorem]{Corollary}
\newtheorem{remark}{Remark}[section]

\theoremstyle{definition}
\newtheorem{definition}[theorem]{Definition}
\newtheorem{notation}[theorem]{Notation}

\newcommand{\R}{\mathbb{R}}

\newcommand{\N}{\mathbb{N}}
\newcommand{\Z}{\mathbb{Z}}
\newcommand{\T}{\mathbb{T}}

\newcommand{\E}{\mathbb E}

\newcommand{\an}[1]{\langle #1 \rangle}

\newcommand{\tildeA}{\mathcal{\tilde A}}

\begin{document}

\title{Discrete wave turbulence for the Benjamin-Bona-Mahony equation, Part I: oscillations for the correlations between the solutions and its initial datum}
\author{Anne-Sophie de Suzzoni\footnote{CMLS, \'Ecole Polytechnique, CNRS, Universit\'e Paris-Saclay, 91128 PALAISEAU Cedex, France, \texttt{anne-sophie.de-suzzoni@polytechnique.edu}}}

\maketitle

\begin{abstract}
We investigate different problems regarding wave turbulence for the Benjamin-Bona-Mahony (BBM) equation in the context of discrete turbulence regime. In the part I, we investigate the behaviour of the correlations between the solution to the BBM equation at latter times with the initial datum.
\end{abstract}

\section{Introduction}

The object of weak turbulence is to study the dynamics of the statistics of solutions to equations coming from fluid mechanics or quantum mechanics when the initial datum is chosen random. Namely, one considers a nonlinear Cauchy problem with an initial datum whose Fourier coefficients are chosen Gaussian and independent (see \cite{Naz} for an extension to random phase approximation). Here, we consider the Benjamin-Bona-Mahony equation : 
\begin{equation}\label{BBMnu1}
\left \lbrace{\begin{array}{cc}
\partial_t u_{L,\varepsilon} + \partial_t \partial_x^2 u_{L,\varepsilon} + \partial_x u_{L,\varepsilon} + \frac\varepsilon{2} \partial_x (u_{L,\varepsilon}^2) =0 & \textrm{on } L\T\\
u(t=0) = a_L = \sum_{\xi \in \frac1{L}\Z^*} a(\xi) \frac{e^{i\xi x}}{(2\pi L)^{d/2}} g_{L \xi}& 
\end{array}} \right. 
\end{equation}
where $a$ is some even, real-valued bounded function and the $(g_k)_{k\in \N^*}$ are independent complex Gaussian variables, and $g_{-k} = \bar g_k$. Here, we are interested in the evolution of the correlations
\[
n(t,\xi,\xi') = \lim_{L\rightarrow \infty}\lim_{\varepsilon \rightarrow 0}\E(\hat u_{L,\varepsilon}(0,\xi) \hat u_{L,\varepsilon}(\varepsilon^{-2}t,\xi')) 
\]
where $\E$ is the expectation with regards to the underlying probability space defining the initial datum, where $\hat u_{L,\varepsilon}(t, \xi)$ is the Fourier coefficient of the solution $u_{L,\varepsilon}$ at time $t$ for the wavenumber $\xi$. 

For symmetry reasons, namely the law of the initial datum and the equation are invariant under the action of space translations, when $\xi\neq \xi'$ the above quantity is null at all times. Therefore, we consider only
\[
n(t,\xi) = \lim_{L\rightarrow \infty}\lim_{\varepsilon \rightarrow 0}\E(\hat u_{L,\varepsilon}(0,\xi) \hat u_{L,\varepsilon}(\varepsilon^{-2}t,\xi)). 
\]

Note that we have two asymptotic regimes. One is $L\rightarrow \infty$ which corresponds to taking the size of the torus to $\infty$. The second one consists in taking the size of the nonlinearity $\varepsilon$ to $0$. This means that for the correlations to display some nontrivial dynamics, one should consider the solution at very large time in terms of the size of the nonlinearity, $\varepsilon^{-2}$, in order to witness some nonlinear effects. 

We mention that wave turbulence has been introduced by Peierls in the late 1920s, in \cite{Peierls1} for cristals, a work that was followed by works by Brout and Prigogine, \cite{Brout-Prigo}. In the 1960s, wave turbulence was developed for plasma physics \cite{Vedenov1967,ZS67}, and in the context of water waves and primitive equations, for example the works by Benney and co-authors \cite{Benney}, by Hasselman \cite{Hass1,Hass2}. Finally, it was popularised by Zakharov in the 60s. Zakharov introduced in particular the now-called Kolmogorov-Zakharov spectra, \cite{KZspectra}, that models a specific momentum transport. This is a very active field in physics. 

In recent years, an interest has been developped for weak turbulence in the mathematics community. For instance,  Lukkarinnen and Spohn, \cite{LukSpo} described some mixing effect of the dynamics for discrete Schrödinger equations. Nikolay Tzvetkov and the author studied the small nonlinearity limit in \cite{dSTont,dS15} for fluid model equations (Kortweg-de Vries, Benjamin-Bona-Mahony, etc). It was then followed by series of works on Schrödinger equations which culminated on the one hand with the derivation by Deng and Hani, \cite{denghani2021,denghani2023} of the kinetic equation (the equation describing the evolution of correlations in the kinetic limit) for cubic Schrödinger equations 
when the size of the torus and the nonlinearity relate in ways called mesoscopic and kinetic in the physics literature; 
and one the other hand by this same type of derivation for some model fluid equations by Staffilani and Tran, \cite{staffilanitran} 
in the kinetic regime (first large box limit, then small nonlinearity).  We also mention \cite{DHPropag,CoG19,denghani19,BGHS,DyKuk1,DyKuk2,Dykuk3,PropagChaos}. The kinetic equation per se has been studied in \cite{collot2022stability,menegaki2022l2stability}.

Another topic is the study of finite box effects, see \cite{KARTASHOVA1,KARTASHOVA2} in the physics literature. Kinetic equations are derived in the kinetic regime (first $L \rightarrow \infty$, then $\varepsilon\rightarrow0$). However, one may consider the opposite regime that is taking first the nonlinearity to $0$, and then the large box limit. Practically speaking, this means that one is conducting an experiment in a large tank of water, with a small initial datum, such that the water tank's size is very small in front of the inverse of the size of the initial datum. 
This regime induces finite box effects (a chaotic behavior) that have been observed experimentally, see \cite{expFiniteBox}. This is a phenomenon observed in \cite{SchroQuint}, where the author proved that for the Schrödinger equation, the derivative of correlations  had an infinite number of singularities in the finite box regime. This regime was also studied for the Schrödinger equation with a forcing term in \cite{DymKukDiscrete}.

Here, we consider this discrete wave turbulence regime for BBM. From the physics literature, it is believed that what is driving the dynamics are resonances of order $1$. But BBM does not display such resonances. Here, we prove that the first nonlinear effect we observe is a phase dynamics. We observe that the correlations between the solution at time $t\varepsilon^{-2}$ and its initial datum differs from its initial value by a phase. We can compute it explicitely when the initial datum has for law an invariant measure for the BBM equation. In the work by Lukkarinnen and Spohn \cite{LukSpo}, the initial datum is an invariant measure for the Schrödinger equation in a discrete setting. The kinetic regime was considered and the result was of a substantially different nature.

\subsection{Framework and result}

Writing $W$ the operator defined as 
\[
W = (1-\partial_x^2)^{-1}\partial_x
\]
which is a Fourier multiplier by $i\omega(\xi)$ with
\[
\omega(\xi) = \frac{\xi}{1+\xi^2}
\]
we can rewrite Equation \eqref{BBMnu1} as
\begin{equation}\label{BBMnu}
\partial_t u_{\varepsilon,L} + W u_{\varepsilon,L} +\varepsilon  W (u^2_{\varepsilon,L}) = 0 .
\end{equation}
We write its flow $\Psi_{\varepsilon,L}$.

Let $(g_n)_{n\geq 1}$ be a sequence of independent complex Gaussian variables on a probability space $(\Omega, \mathcal F, \mathbb P)$ and set $g_{-n} = \bar g_n$. Let $\varphi \in \an{\xi}^{-1} L^2(\R)$ even and bounded. We assume that the derivative of $\varphi$ is a bounded function.

Alternatively, we assume that $\varphi(\xi) = \frac1{\an{\xi}}$.

We define
\[
\varphi_L (x)= \sum_{n\in \Z^*} g_n \varphi(\frac{n}{L}) \frac{e^{inx/L}}{\sqrt{2\pi L}}.
\]
This sum converges in $L^2(\Omega, H^s(L\T))$ for all $s<\frac12$.

We write its law $\mu_L$. Namely, for all $A$ in the Borel sigma-algebra of $\cap_{s<\frac12} H^s$, we have
\[
\mu_L(A) = \mathbb P( \varphi_L^{-1}(A)).
\] 

In the case that $\varphi(\xi) = \an{\xi}^{-1}$, we write the corresponding measure $\nu_L$.

In \cite{moimeme}, we proved the invariance of $\nu_L$ under the flow of BBM. For $\nu_L$ almost all $u_0 \in \cap_{s<1/2} H^s$, the flow $\Psi_{L,\varepsilon}$ of \eqref{BBMnu} is globally well-posed and we have for all $t\in \R$, all $A$ in the Borel sigma-algebra of $\cap_{s<\frac12} H^s$,
\[
\nu_L(\Psi_{\varepsilon,L}(t)^{-1}(A)) = \nu_L(A).
\]

Of course, the proof in \cite{moimeme} requires some adaptation to $L$ and $\varepsilon$, but it is based on the conservation of the $H^1$ norm under the flow of BBM on the torus, which is still valid for \eqref{BBMnu}.

We are interested in the behaviour of the correlations of the initial datum $\varphi_L$ and the solution at ulterior times $t\varepsilon^{-2}$, $\Psi_{\varepsilon,L}(t\varepsilon^{-2})(\varphi_L)$ in the discrete weak turbulence regime. The linear flow induces rapid oscillations, we remove it to observe the next order. Set $S(t) = \Psi_{0,L}(t)$. 

\begin{theorem}\label{th:main} There exists $C$ such that for all $L\geq 1$ such that $L^2$ is not algebraic of degree less than $2$, for all $t\in \R$, there exists $f_{L,t} : \R_+^*\rightarrow \R_+$ that goes to $0$ at $0$ such that for all $\xi \in \frac1{L}\Z^*$, all $\varepsilon>0$, we have 
\[
\Big|\E( \overline{\widehat {S(t)(\varphi_L)}(\xi)}\hat \Psi_{\varepsilon,L}(t\varepsilon^{-2})(\varphi_L)(\xi)) - e^{it \xi \Phi(\xi)}|\varphi(\xi)|^2 \Big|  \leq \frac{e^{C(1+|t|)}}{L} |\varphi(\xi)|^2 + f_{L,t}(\varepsilon) |\varphi(\xi)|.
\]
where 
\[
\Phi(\xi) = \frac2{\pi}
\int d\eta \frac{|\varphi(\eta)|^2(1+\eta^2)}{(3+\frac{\xi^2+ \eta^2+(\xi-\eta)^2}{2})(3+ \frac{\xi^2+\eta^2 + (\xi+\eta)^2}{2})}.
\]
In the case that $\varphi(\xi) = \an{\xi}^{-1}$, we have 
\[
\Phi(\xi)=
\frac{1}{(3+\xi^2) \sqrt{3(1+\xi^2/4)}} .
\]
\end{theorem}

We deduce in particular the following corollary.

\begin{corollary} Let $(L_n)_n$ be a sequence of number such that above a certain rank $n_0$, all $L_n^2$ are not algebraic of degree less that $2$ and such that $L_n\rightarrow \infty$, let $(\xi_n)_n$ be a sequence of real numbers such that for all $n\in \N$, $\xi_n \in \frac1{L_n}\Z^*$ and such that $(\xi_n)_n$ converges to some $\xi \in \R$. We have
\[
\lim_{n\rightarrow \infty} \lim_{\varepsilon \rightarrow 0} \E( \overline{\widehat {S(t)(\varphi_L)}(\xi)}\hat \Psi_{\varepsilon,L}(t\varepsilon^{-2})(\varphi_L)(\xi)) = e^{it\xi \Phi(\xi) }|\varphi(\xi)|^2.
\]
\end{corollary}

\begin{remark} The reason why we do not allow $L^2$ to be algebraic of degree less than $2$ is to avoid non trivial resonances of degree $2$. The proof is relatively robust in the sense that it could apply to many models displaying the following properties : no resonances of order $1$ and only trivial resonances of order $2$. We believe that the result would probably still hold with non-trivial resonances of order $2$, but the result would be implied by more involved probabilistic cancellations. 

Opposite to standard results on discrete wave turbulence, because we have no resonances of order $1$, they do not play an important role in the description of the dynamics. This is why we move to second order resonances.
\end{remark}

\begin{remark} We compare our result with the one in \cite{LukSpo}. The nature of the equations are very different because in \cite{LukSpo}, the equations are discrete and cubic while we have a quadratic and continuous equation. The main difference comes however from the fact that Lukarinnen and Spohn consider the kinetic regime (first $L\rightarrow \infty$, then $\varepsilon\rightarrow 0$). We do the opposite, which is referred to as discrete wave turbulence. This is why Lukarinnen and Spohn observe some mixing effect and we observe phase oscillations.

What is more, the $\liminf$ in $\Lambda$ in \cite{LukSpo} is comparable to the fact that we avoid certain sizes of the torus.

Finally, our result is global in time, whereas the result in \cite{LukSpo} is in finite time.
\end{remark}

The strategy of the proof is the following. First, we adapt normal forms to suit our purpose and get local well-posedness up to times of order $\varepsilon^{-2}$ (but badly behaving in $L$), see Section \ref{sec:NF}. Then, we describe the solution thanks to Feynmann diagrams, here binary trees with an order on nodes that respect the parenthood order, see Subsection \ref{subsec:FD}. Then, we describe the leading order term in each tree and estimate the remainder, see Subsection \ref{subsec:estimates}. In particular, we use the invariance of the measure or of the $H^1$ norm under the flow to propagate the estimate of the remainder up to any time in Subsection \ref{subsec:PropagEst}.  We exhibit a first cancellation between the trees and realise that there are only a few trees participating to the leading order in Subsection \ref{subsec:fstcancel} and more precisely in Proposition \ref{prop:fstcancel}. The key argument in this paper is that because the order two resonances are trivial, we can exhibit a second cancellation between trees, which is that only combs participate to the leading order, see Subsection \ref{subsec:combs}, and more precisely Proposition \ref{prop:rewriteK2} and Remark \ref{rm:combs}. Finally, we estimate the measure of the set of initial data we consider in order to propagate our estimates on the remainder and conclude.

The paper is organised as follows : in Section \ref{sec:NF}, we perform a normal form argument to expand the solution as Dyson series up to times of order $\varepsilon^{-2}$; in Section \ref{sec:trees}, we expand the solution into trees and identify the leading order in our analysis; in Section \ref{sec:EstRemInvMea} we estimate the set of initial data that are appropriate to our problem, and in Section \ref{sec:proof}, we give the final lines of the proof of the result.

\subsection{Notations}

\begin{notation}[Fourier coefficients] We take the following convention : for $u\in L^2(L\T)$, and $\xi \in \frac1{L} \Z^*$, we set
\[
\hat u(\xi) = \frac1{\sqrt{2\pi L}} \int_{L\T} u(x) e^{-ix\xi} dx.
\]
We often write the Fourier transform directly on flows or on operators. This should be understood as describing flows or operators in Fourier mode. 
\end{notation}

\begin{notation}[Sobolev spaces] Let $s\geq 0$, and $L>0$, we say that $u\in L^2(L\T)$ belongs to $H^s(L\T)$ if $\hat u(0) = 0$ and 
\[
\|u\|_{H^s(L\T)}^2 := \sum_{\xi \in \frac1{L}\Z^*} \an{\xi}^{2s} |\hat u(\xi)|^2 < \infty .
\]
In the following, we often omit the dependence in $L$ of the Sobolev space $H^s(L\T)$ and write $H^s = H^s(L\T)$.
\end{notation}

\begin{notation} For any $x,y \in \R$, we set $[|x,y|] = [x,y]\cap \Z$.
\end{notation}

Constants may depend on the regularity $s$ at which we work and on the function $\varphi$ but not on any other data of the problem unless it is specified.

\subsection{Acknowledgements}

The author is supported by ANR grant ESSED ANR-18-CE40-0028.

\section{Normal forms}\label{sec:NF}

We consider the Cauchy problem
\begin{equation}\label{cauchyproblem} 
\left \lbrace{\begin{array}{cc}
\partial_t u_{\varepsilon,L} + W u_{\varepsilon,L} +  W (u^2_{\varepsilon,L}) = 0 \\
u_{\varepsilon,L}(t=0) = a, & \in H^s(L\T), \quad s>\frac14.
\end{array}} \right.
\end{equation}
We prove that it is locally well-posed up to times of order $\varepsilon^{-2}$ and that the flow associated can be written in the form of a convergent series using normal forms. 

We remark that the quantity
\[
\int_{L\T} u_{\varepsilon,L}(t,x) dx
\]
does not depend on time. We assume implicitely in all the rest of the paper that $\int_{L\T} a(x) dx = 0$, that is $\hat a(0) =0$ and thus that $\hat u_{\varepsilon,L}(t) (0)=0$. Hence we consider only non zero frequencies.

We set $u_n(t) = \Psi_{\varepsilon,L,n}(t)(a)$ the following sequence. Writing $S(t)$ the semi-group $S(t) = e^{-tW}$, we set
\[
u_0 = S(t) a,
\]
and for $n\in \N$,
\begin{equation}\label{defun}
u_{n+1} = \varepsilon \int_{0}^t  W S(t-\tau) \Big[\sum_{n_1+n_2 = n} u_{n_1}(\tau)u_{n_2}(\tau)\Big]d\tau.
\end{equation}
If the series converges, we have (formally) setting $u = \sum_n u_n$,
\[
u = S(t) a + \varepsilon \int_{0}^t WS(t-\tau) (u^2(\tau))d\tau
\]
and thus $u$ is indeed the solution to \eqref{cauchyproblem}.

\begin{definition} Let $s> \frac14$. We define the bilinear map
\[
N_\varepsilon : H^s\times H^s \rightarrow H^s
\]
as for $u,v \in H^s$ and in Fourier mode
\[
\hat N_\varepsilon (u,v)(\xi) = \frac1{\sqrt{2\pi L}}{\bf 1}_{|\xi|\leq \varepsilon^{-\alpha(s)}} \sum_{\xi_1+\xi_2 = \xi} \frac{\omega(\xi)}{\Delta^{\xi}_{\xi_1,\xi_2}}\hat u(\xi_1) \hat v(\xi_2)
\]
where
\[
\Delta^\xi_{\xi_1,\xi_2} = \omega(\xi) - \omega(\xi_1) - \omega(\xi_2) \neq 0
\]
and where
\[
\alpha(s) = \left \lbrace{\begin{array}{cc}
\frac1{1-2s} & \textrm{ if } s\leq \frac13 \\
3 & \textrm{ otherwise.} 
\end{array}} \right.
\]
\end{definition}

\begin{proposition} For all $\xi_1,\xi_2,\xi \in \R^*$ such that $\xi_1 + \xi_2 = \xi$, we have 
\[
\Delta^\xi_{\xi_1,\xi_2} = -\omega(\xi)\omega(\xi_1)\omega(\xi_2) \Big(3 + \frac{\xi^2 + \xi_1^2 + \xi_2^2}{2} \Big).
\]
\end{proposition}

\begin{proof} Straightforward computation.
\end{proof}

\begin{proposition}\label{prop:estNeps} Let $s>\frac14$. There exists $C$ such that for all $L\geq 1$ and all $\varepsilon < 1$, for all $u,v\in H^s(L\T)$,
\[
\|N_\varepsilon(u,v)\|_{H^s}\leq C L^{2} \varepsilon^{-1/2} \|u\|_{H^s}\|v\|_{H^s}.
\]
\end{proposition}

\begin{proof}
We have for $\xi \in \frac1{L}\Z^*$,
\[
\hat N_\varepsilon (u,v)(\xi) = \frac1{\sqrt{2\pi L}}{\bf 1}_{|\xi|\leq \varepsilon^{-\alpha(s)}} \sum_{\xi_1+\xi_2 = \xi} \frac{\omega(\xi)}{\Delta^{\xi}_{\xi_1,\xi_2}}\hat u(\xi_1) \hat v(\xi_2).
\]
Set $s' = \min(1,s)$. Since we have for all $\xi_1,\xi_2$,
\[
\an{\xi}^{s-s'} \lesssim_s \an{\xi_1}^{s-s'} + \an{\xi_2}^{s-s'},
\]
we deduce 
\[
\an{\xi}^s \lesssim_s \an{\xi}^{s'}\an{\xi_1}^{-s'}\an{\xi_2}^{-s'} \an{\xi_1}^s\an{\xi_2}^s.
\] 
By triangle inequality, we deduce
\[
\an{\xi}^s |\hat N_\varepsilon(u,v)(\xi)| \leq \frac1{\sqrt{2\pi L}} {\bf 1}_{|\xi|\leq \varepsilon^{-\alpha(s)}}  \sum_{\xi_1+\xi_2 = \xi} \an{\xi}^{s'}\an{\xi_1}^{-s'}\an{\xi_2}^{-s'}\frac{|\omega(\xi)|}{|\Delta^{\xi}_{\xi_1,\xi_2}|}\an{\xi_1}^s|\hat u(\xi_1)| \an{\xi_2}^s |\hat v(\xi_2)| .
\]
We have 
\[
\frac{|\omega(\xi)|}{|\Delta^{\xi}_{\xi_1,\xi_2}|} = \frac1{|\omega(\xi_1)\omega(\xi_2)| \Big(3 + \frac{\xi^2 + \xi_1^2 + \xi_2^2}{2} \Big)}
\]
and, since $|\xi_1|,|\xi_2|\geq \frac1{L}$,
\[
|\omega(\xi_j)| \geq  \frac1{\an{\xi_j}L}
\]
and thus
\[
 \frac{|\omega(\xi)|}{|\Delta^{\xi}_{\xi_1,\xi_2}|} \leq L^2\frac{\an{\xi_1} \an{\xi_2}}{\Big(3 + \frac{\xi^2 + \xi_1^2 + \xi_2^2}{2} \Big) }.
\]
We deduce
\[
\an{\xi}^{s'}\an{\xi_1}^{-s'}\an{\xi_2}^{-s'}\frac{|\omega(\xi)|}{|\Delta^{\xi}_{\xi_1,\xi_2}|} \leq L^2 \frac{\an{\xi_1}^{1-s'}\an{\xi}^{s'} \an{\xi_2}^{1-s'}}{\Big(3 + \frac{\xi^2 + \xi_1^2 + \xi_2^2}{2} \Big)^{1/2} }.
\]
Because $s'+ 1-s' + 1-s' = 2 - s'$, we deduce that
\[
\an{\xi}^{s'}\an{\xi_1}^{-s'}\an{\xi_2}^{-s'}\frac{|\omega(\xi)|}{|\Delta^{\xi}_{\xi_1,\xi_2}|} \leq CL^2 \an{\xi}^{-s'}.
\]
Therefore, by Cauchy-Schwarz, we have 
\[
\an{\xi}^s |\hat N_\varepsilon(u,v)(\xi)| \leq CL^{3/2} \an{\xi}^{-s'} {\bf 1}_{|\xi|\leq \varepsilon^{-\alpha(s)}} \|u\|_{H^s(L\T)} \|v\|_{H^s(L\T)}.
\]
Since we have
\[
\sum_{\xi \in \frac1{L}\Z}\an{\xi}^{-2s'} {\bf 1}_{|\xi|\leq \varepsilon^{-\alpha(s)}} \leq C_s L \int_1^{\varepsilon^{-\alpha(s)}} |x|^{-2s'}dx \leq C_s \left \lbrace{\begin{array}{cc}
L\varepsilon^{-\alpha(s)(1/2-s)} & \textrm{ if } s\leq \frac12,\\
-\alpha(s) L \ln \varepsilon & \textrm{ if } s=\frac12,\\
L & \textrm{ otherwise.}\end{array}}\right.
\]
We deduce 
\[
\|N_\varepsilon (u,v)\|_{H^s(L\T)} \leq C_s \left \lbrace{\begin{array}{cc}
L\varepsilon^{-\alpha(s)(1/2-s)} & \textrm{ if } s< \frac12,\\
-3L \ln \varepsilon & \textrm{ if } s=\frac12,\\ 
L & \textrm{ otherwise.}\end{array}}\right.
\]
If $s\leq \frac13$, we have $\alpha(s)(1/2 - s) = 1/2$, if $s\in [1/3,\frac12)$, we have $\alpha(s)(\frac12 - s)=3(1/2-s) \leq 1/2$, which concludes the proof.
\end{proof}

\begin{notation} We set $v_0 = u_0$ and for $n\in \N$,
\[
v_{n+1} = u_{n+1} + \varepsilon \sum_{n_1+n_2 = n} N_\varepsilon(u_{n_1},u_{n_2}).
\]
\end{notation}

\begin{proposition}\label{prop:eqv} Assume $a$ is a trigonometric polynomial. Then all the $u_n$ are smooth functions of time with values in trigonometric polynomials. So are the $v_n$. We deduce that $v_n$ solves the equation.
\[
\partial_t v_n + W v_n = \varepsilon \sum_{n_1+n_2 = n-1}Q_\varepsilon(u_{n_1},u_{n_2}) + \varepsilon^2 \sum_{n_1+n_2+n_3 = n-2} P_\varepsilon (u_{n_1},u_{n_2},u_{n_3})
\]
where $Q_\varepsilon$ and $P_\varepsilon$ are respectively bilinear and trilinear maps defined in Fourier mode by
\[
\hat Q_\varepsilon (u,v)(\xi) = -i \frac{\omega(\xi)}{\sqrt{2\pi L}}{\bf 1}_{|\xi|> \varepsilon^{-\alpha(s)}} \sum_{\xi_1+\xi_2 = \xi} \hat u(\xi_1)\hat u(\xi_2)
\]
and
\[
\hat P_\varepsilon(u,v,w) (\xi) = \frac{-i}{\pi L}  {\bf 1}_{|\xi|\leq \varepsilon^{-\alpha(s)}} \sum_{\xi_1+\xi_2+\xi_3} \frac{\omega(\xi) \omega(\xi_2+\xi_3)}{\Delta^\xi_{\xi_1,\xi_2+\xi_3}} \hat u(\xi_1)\hat v(\xi_2) \hat w(\xi_3).
\] 

\end{proposition}

\begin{proof}
Because $v_n$ is a smooth function of time with values in trigonometric polynomials, we can derive it explicitely :
\[
\partial_t v_n = \partial_t u_n + \varepsilon \sum_{n_1+n_2= n-1} \partial_t N_\varepsilon(u_{n_1},u_{n_2}).
\]
Since $N_\varepsilon$ is bilinear, and the sum is symmetric in $n_1,n_2$ we get
\[
\partial_t v_n  = \partial_t u_n + 2\varepsilon \sum_{n_1+n_2 = n-1} N_\varepsilon(u_{n_1},\partial_t u_{n_2}). 
\]
We now recall that $u_n$ satisfies \eqref{defun} and get
\begin{multline*}
\partial_t v_n  = -W u_{n} -  W \varepsilon \sum_{n_1+n_2 = n-1} u_{n_1}u_{n_2} \\
- 2\varepsilon \sum_{n_1+n_2 = n-1} N_\varepsilon( u_{n_1}, W u_{n_2})) - 2\varepsilon^2  \sum_{n_1+n_2+n_3 = n-2}N_\varepsilon (u_{n_1}, W (u_{n_2}u_{n_3})) . 
\end{multline*}

We set 
\[
P_\varepsilon (u,v,w) = -2 N_\varepsilon(u,  W(vw)).
\]
Let us verify that it corresponds to the definition in Fourier mode of the proposition. For $\xi \in \frac1{L}\Z$, we have 
\[
\hat P_\varepsilon (u,v,w)(\xi) = -2 \hat N_\varepsilon (u, W(vw)) (\xi).
\]
Using the definition of $N_\varepsilon$ we get
\[
\hat P_\varepsilon (u,v,w)(\xi) = - \frac2{\sqrt{2\pi L}}{\bf 1}_{|\xi|\leq \varepsilon^{-\alpha(s)}} \sum_{\xi_1+\xi_2 = \xi} \frac{\omega(\xi)}{\Delta^{\xi}_{\xi_1,\xi_2}}\hat u(\xi_1)  i\omega(\xi_2) \widehat {vw}(\xi_2).
\]
We use that
\[
\widehat{vw}(\xi_2) = \frac1{\sqrt{2\pi L}}\sum_{\xi_2' + \xi_3 = \xi_2} \hat v(\xi_2')\hat w(\xi_3)
\]
to get
\[
\hat P_\varepsilon (u,v,w)(\xi) = - i \frac1{\pi L} {\bf 1}_{|\xi|\leq \varepsilon^{-\alpha(s)}} \sum_{\xi_1 + \xi_2 + \xi_3 = \xi}  \frac{\omega(\xi)\omega(\xi_2 + \xi_3)}{\Delta^\xi_{\xi_1,\xi_2 + \xi_3}} \hat u(\xi_1) \hat v(\xi_2) \hat w (\xi_3).
\]

We now compute 
\[
I = 2\varepsilon \sum_{n_1+n_2 = n-1} N_\varepsilon( u_{n_1}, W u_{n_2}).
\]
By symmetry in the sum in $n_1,n_2$, we have
\[
I = \varepsilon \sum_{n_1+n_2 = n-1} (N_\varepsilon( u_{n_1}, W u_{n_2}) + N_\varepsilon (Wu_{n_1},u_{n_2})).
\]
We use the definition of $N_\varepsilon$ to get, in Fourier mode, for $\xi \in \frac1{L}\Z$,
\[
\hat I(\xi) = \varepsilon \omega(\xi) \frac1{\sqrt{2\pi L}}{\bf 1}_{|\xi|\leq \varepsilon^{-\alpha(s)}} \sum_{n_1+ n_2 = n-1} \sum_{\xi_1 + \xi_2 = \xi} \frac{i\omega(\xi_1) + i\omega(\xi_2) }{\Delta^\xi_{\xi_1,\xi_2}}  \hat u_{n_1}(\xi_1) \hat u_{n_2}(\xi_2).
\]
Using the definition of $\Delta^\xi_{\xi_1,\xi_2}$, we get that
\[
\frac{i\omega(\xi_1) + i\omega(\xi_2) }{\Delta^\xi_{\xi_1,\xi_2}}  = -i + i\frac{\omega(\xi)}{\Delta^\xi_{\xi_1,\xi_2}}.
\]
We deduce
\begin{multline*}
\hat I(\xi) = -i \varepsilon \omega(\xi) \frac1{\sqrt{2\pi L}}{\bf 1}_{|\xi|\leq \varepsilon^{-\alpha(s)}} \sum_{n_1+ n_2 = n-1} \sum_{\xi_1 + \xi_2 = \xi}  \hat u_{n_1}(\xi_1) \hat u_{n_2}(\xi_2) \\
+ i \omega(\xi) \frac1{\sqrt{2\pi L}}{\bf 1}_{|\xi|\leq \varepsilon^{-\alpha(s)}} \sum_{n_1+ n_2 = n-1} \sum_{\xi_1 + \xi_2 = \xi} \frac{\omega(\xi)}{\Delta^\xi_{\xi_1,\xi_2}}\hat u_{n_1}(\xi_1) \hat u_{n_2}(\xi_2) .
\end{multline*}
We recognize
\[
I  = -\varepsilon W  \chi_\varepsilon \Big(\sum_{n_1+n_2 = n-1 }u_{n_1}u_{n_2}\Big) + \varepsilon W \sum_{n_1+n_2 = n-1 }N_\varepsilon (u_{n_1},u_{n_2})
\]
where $\chi_\varepsilon$ is the orthogonal projection over Fourier modes lower than $\varepsilon^{-\alpha(s)}$.

We deduce that
\begin{multline*}
\partial_t v_n = -W u_n - W \varepsilon\sum_{n_1+ n_2 = n-1} N_\varepsilon(u_{n_1},u_{n_1}) - \varepsilon (1-\chi_{\varepsilon}) W \sum_{n_1+n_2 =n-1} u_{n_1}u_{n_2}\\
 + \varepsilon^2 \sum_{n_1+n_2+n_3 = n-2}P_\varepsilon (u_{n_1},u_{n_2},u_{n_3}).
\end{multline*}
We have that
\[
Q_\varepsilon (u,v) = - (1- \chi_\varepsilon ) W (uv)
\]
which concludes the proof.
\end{proof}

\begin{proposition}\label{prop:estonQeps} Let $s >\frac14$. There exists $C_s$ such that for all $\varepsilon \in (0,1)$ and all $L\geq 1$, and for all $u,v \in H^s(L\T)$, we have 
\[
\|Q_\varepsilon(u,v)\|_{H^s} \leq C_s \varepsilon \|u\|_{H^s}\|v\|_{H^s}.
\]
\end{proposition}

\begin{proof} Let $\xi\in \frac1{L} \Z$. Since for all $\xi_1,\xi_2 \in \frac1{L}\Z$, we have 
\[
\an{\xi_1+\xi_2}^s \leq C_s (\an{\xi_1}^s + \an{\xi_2}^s), 
\]
we have 
\[
\an{\xi}^s |\hat Q_\varepsilon(u,v)(\xi)| \leq \frac{C_s}{\sqrt L} {\bf 1}_{|\xi|\geq \varepsilon^{-\alpha(s)}}|\omega(\xi)| \sum_{\xi_1+\xi_2 = \xi} (\an{\xi_1}^s |\hat u (\xi_1)|\, |\hat v(\xi_2)| + |\hat u(\xi_1) |\an{\xi_2}^s|\hat v(\xi_2)|).
\]
By Cauchy-Schwarz, we get
\[
\an{\xi}^s |\hat Q_\varepsilon(u,v)(\xi)| \leq \frac{C_s}{\sqrt L} {\bf 1}_{|\xi|\geq \varepsilon^{-\alpha(s)}}\an{\xi}^{-1} \|u\|_{H^s}\|v\|_{H^s}.
\]
We have that
\[
{\bf 1}_{|\xi|\geq \varepsilon^{-\alpha(s)}}\an{\xi}^{-1} \leq {\bf 1}_{|\xi|\geq \varepsilon^{-\alpha(s)}}\an{\xi}^{-1/\alpha(s)} \an{\xi}^{-1 + 1/\alpha(s)}\leq \varepsilon \an{\xi}^{-1+1/\alpha(s)}.
\]
For $s\in (\frac14,\frac13)$, we have $-1+ \frac1{\alpha(s)} = -2s < -\frac12$ and for $s\geq \frac13$, we have $-1+ \frac1{\alpha(s)} = -\frac23 <-\frac12.$

It remains to sum on $\xi$ to get the result.
\end{proof}

\begin{proposition}\label{prop:estonTeps} Let $s>0$. There exists $C_s$ such that for all $\varepsilon \in (0,1)$ and all $L\geq 1$, and for all $u,v,w \in H^s(L\T)$, we have 
\[
\|P_\varepsilon(u,v,w)\|_{H^s} \leq C_s L\|u\|_{H^s}\|v\|_{H^s}\|w\|_{H^s}.
\]
\end{proposition}

\begin{proof} Let $\xi\in \frac1{L}\Z$. By definition of $P_\varepsilon$, we have 
\begin{multline*}
\an{\xi}^s|\hat P_\varepsilon(u,v,w) (\xi)| \leq \\
 \frac1{\pi L} \sum_{\xi_1+\xi_2+\xi_3 = \xi} \an{\xi}^s\an{\xi_1}^{-s'}\an{\xi_2+\xi_3}^{-s'} \frac{|\omega(\xi)\omega(\xi_2+\xi_3)|}{|\Delta^\xi_{\xi_1,\xi_2+\xi_3}|}\an{\xi_1}^{s'}|\hat u(\xi_1)|\an{\xi_2+\xi_3}^{s'}|\hat v(\xi_2)\hat w(\xi_3)|
\end{multline*}
where $s' = min(s,1)$.

Since
\[
\an{\xi}^{s-s'} = \an{\xi_1+\xi_2 + \xi_3}^{s-s'}\leq C_s (\an{\xi_1}^{s-s'} + \an{\xi_2}^{s-s'} + \an{\xi_3}^{s-s'}),\quad 
\an{\xi_2+\xi_3}^s \leq C_s (\an{\xi_2}^s + \an{\xi_3}^s), 
\]
using Cauchy-Schwarz, we get
\[
\an{\xi}^s|\hat P_\varepsilon(u,v,w) (\xi)| \leq C_s \frac1{\pi L} \sum_{\xi_1+\eta = \xi} \an{\xi}^{s'}\an{\xi_1}^{-s'}\an{\eta}^{-s'} \frac{|\omega(\xi)\omega(\eta)|}{|\Delta^\xi_{\xi_1,\eta}|}\an{\xi_1}^s|\hat u(\xi_1)\|v\|_{H^s}\|w\|_{H^s}.
\]
We have that for all $\xi_1,\eta,\xi$ such that $\xi_1+\eta = \xi$,
\[
\frac{|\omega(\xi)\omega(\eta)|}{|\Delta^\xi_{\xi_1,\eta}|} \leq 2L \frac{\an{\xi_1}}{6 + \xi_1^2 + \eta^2 + \xi^2}.
\]
Therefore, we have that
\[
 \an{\xi}^{s'}\an{\xi_1}^{-s'}\an{\eta}^{-s'} \frac{|\omega(\xi)\omega(\eta)|}{|\Delta^\xi_{\xi_1,\eta}|}\leq 2L \an{\xi}^{-(1+s')/2} \an{\eta}^{-(1+s')/2} \frac{\an{\xi_1}^{1-s'}\an{\xi}^{1/2+ 3s'/2} \an{\eta}^{(1-s')/2}}{6 + \xi^2 +\eta^2+\xi_1^2}.
\]
Since $1-s' +\frac12 + \frac{3s'}2 + \frac{1-s'}{2} = 2$, we deduce
\[
 \an{\xi}^{s'}\an{\xi_1}^{-s'}\an{\eta}^{-s'} \frac{|\omega(\xi)\omega(\eta)|}{|\Delta^\xi_{\xi_1,\eta}|}\leq 2L\an{\xi}^{-(1+s')/2} \an{\eta}^{-(1+s')/2}.
\]
By Cauchy-Schwarz, we get
\[
\an{\xi}^s|\hat P_\varepsilon(u,v,w) (\xi)| \leq C_s \sqrt L \an{\xi}^{-(1+s')/2}\|u\|_{H^s}\|v\|_{H^s}\|w\|_{H^s}.
\]
It remains to sum on $\xi$ to get the result.
\end{proof}

\begin{definition} We define $\Psi_{\varepsilon,L,n}(t)$ the flow corresponding to $u_n$, that is for $a\in H^s$,
\[
\Psi_{\varepsilon,L,0}(t)(a) = S(t)(a),
\]
and for $n\geq 0$,
\[
\Psi_{n+1}(t)(a) = \varepsilon \int_{0}^t W S(t-\tau) \Big[\sum_{n_1+n_2 = n} \Psi_{\varepsilon,L,n_1}(\tau)(a) \Psi_{\varepsilon,L,n_2}(\tau)(a)\Big] d\tau.
\]
We often write for short $\Psi_n = \Psi_{\varepsilon,L,n}$.
\end{definition}

\begin{proposition}\label{prop:fstestonPsin}Let $s>\frac14$. There exists $C_s$ such that for all $L\geq 1$, $D\geq 1$, all $T\leq T_0(L,D)= \frac1{C_s DL^2}$ and for all $\varepsilon \leq \epsilon_0(L,D) =\frac1{C_s D^2 L^5 }$, for all $a\in H^s$ such that $\|a\|_{H^s}\leq D\sqrt L$, then for all $n\in \N$, all $t\in [-T\varepsilon^{-2}, T\varepsilon^{-2}]$, we have
\[
\|\Psi_n (t) (a) \|_{H^s} \leq 8^{-n} c_n D\sqrt L
\]
where $(c_n)_n$ are the Catalan numbers.

What is more, for all $b\in H^s$ such that $\|b\|_{H^s}\leq D\sqrt L$, we have
\[
\|\Psi_n(t)(a) - \Psi_n(t)(b)\|_{H^s}\leq 8^{-n}c_n \|a-b\|_{H^s}.
\]
\end{proposition}

\begin{proof}We work by induction on $n$. For $n=0$, we have that $\Psi_0(t) = S(t)$, which preserves the $H^s$ norm and thus the inequalities are verified. 

Assume the inequalities are true for all $n'\leq n$. We prove them for $n+1$.

We first deal with the cases $a,b$ trigonometric polynomials. We set $u_n(t) = \Psi_n(t)(a)$ and 
\[
v_n = u_n + \varepsilon \sum_{n_1+n_2=n-1} N_\varepsilon(u_{n_1},u_{n_2})
\]
as before.

We have that 
\[
\Psi_{n+1}(t) (a) = u_{n+1}(t) = v_{n+1}(t) - \varepsilon\sum_{n_1+n_2 = n}N_\varepsilon(u_{n_1}(t),u_{n_2}(t)).
\]
We use that $v_{n+1}$ is a solution to the equation described in Proposition \ref{prop:eqv} to get
\[
v_{n+1}(t) = \int_{0}^t S(t-\tau)\Big[\sum_{n_1+n_2 = n}\varepsilon Q_\varepsilon(u_{n_1},u_{n_2}) + \varepsilon^2\sum_{n_1+n_2+n_3 = n-1} P_\varepsilon(u_{n_1},u_{n_2},u_{n_3})\Big] d\tau.
\]
We get the induction formula
\begin{multline*}
\Psi_{n+1}(t)(a) = \int_{0}^t S(t-\tau)\Big[\varepsilon\sum_{n_1+n_2 = n}Q_\varepsilon(\Psi_{n_1}(\tau)(a),\Psi_{n_2}(\tau)(a)) \Big] \\
+ \int_{0}^t S(t-\tau)\Big[ \varepsilon^2\sum_{n_1+n_2+n_3 = n-1} P_\varepsilon(\Psi_{n_1}(\tau)(a),\Psi_{n_2}(\tau)(a),\Psi_{n_3}(\tau)(a))\Big]\\
-\varepsilon\sum_{n_1+n_2 = n}N_\varepsilon(\Psi_{n_1}(t)(a),\Psi_{n_2}(t)(a)).
\end{multline*}

On the one hand, we have for all $t$, thanks to Proposition \ref{prop:estNeps},
\[
\|\varepsilon\sum_{n_1+n_2 = n}N_\varepsilon(\Psi_{n_1}(t)(a),\Psi_{n_2}(t)(a)) \|_{H^s} \leq C_s\varepsilon^{1/2}L^2\sum_{n_1+n_2=n} \|\Psi_{n_1}(t)(a)\|_{H^s}\|\Psi_{n_2}(t)(a)\|_{H^s}. 
\]
We also have
\begin{multline*}
\|\varepsilon\sum_{n_1+n_2 = n}N_\varepsilon(\Psi_{n_1}(t)(a),\Psi_{n_2}(t)(a)) - \varepsilon\sum_{n_1+n_2 = n}N_\varepsilon(\Psi_{n_1}(t)(b),\Psi_{n_2}(t)(b)) \|_{H^s} \\
\leq C_s\varepsilon^{1/2}L^2\sum_{n_1+n_2=n} (\|\Psi_{n_1}(t)(a)\|_{H^s} + \|\Psi_{n_1}(t)(b)\|_{H^s})\|\Psi_{n_2}(t)(a) - \Psi_{n_2}(t)(b)\|_{H^s}. 
\end{multline*}

Using the induction hypothesis for appropriate times, we get
\[
\|\varepsilon \sum_{n_1+n_2 = n}N_\varepsilon(\Psi_{n_1}(t)(a),\Psi_{n_2}(t)(a)) \|_{H^s} \leq C_s 8^{-n} D L^{5/2} \varepsilon^{1/2}\Big(\sum_{n_1+n_2=n} c_{n_1}c_{n_2}\Big) D\sqrt L. 
\]
Taking $\varepsilon $ small enough such that $DL^{5/2}\varepsilon^{1/2}  C_s \leq \frac1{24}$ and by definition of the Catalan numbers, we get
\[
\|\varepsilon \sum_{n_1+n_2 = n}N_\varepsilon(\Psi_{n_1}(t)(a),\Psi_{n_2}(t)(a)) \|_{H^s} \leq \frac1{3}8^{-(n+1)}c_{n+1} D \sqrt L . 
\]
For the same reasons, we have
\[
\|\varepsilon\sum_{n_1+n_2 = n}N_\varepsilon(\Psi_{n_1}(t)(a),\Psi_{n_2}(t)(a)) - \varepsilon\sum_{n_1+n_2 = n}N_\varepsilon(\Psi_{n_1}(t)(b),\Psi_{n_2}(t)(b)) \|_{H^s} \leq \frac1{3}8^{-(n+1)} c_{n+1} \|a-b\|_{H^s}.
\]

Set 
\[
A(t) (a) = A_1(t)(a) + A_2(t)(a)
\]
with
\[
A_1(t) (a) =  \varepsilon\int_{0}^t S(t-\tau)\Big[\sum_{n_1+n_2 = n}Q_\varepsilon(\Psi_{n_1}(\tau)(a),\Psi_{n_2}(\tau)(a))\Big]d\tau
\]
and
\[
A_2(t)(b) = \varepsilon^2 \int_{0}^t S(t-\tau) \Big[\sum_{n_1+n_2+n_3 = n-1} P_\varepsilon(\Psi_{n_1}(\tau)(a),\Psi_{n_2}(\tau)(a),\Psi_{n_3}(\tau)(a))\Big].
\]

For $t \in [-T\varepsilon^{-2},T\varepsilon^{-2}]$, we have
\[
\|A_1(t)(a)\|_{H^s} \leq \varepsilon^{-1} T \sum_{n_1+n_2=n}\sup_{\tau \in [-T\varepsilon^{-2},T\varepsilon^{-2}]} \|Q_\varepsilon(\Psi_{n_1}(\tau)(a),\Psi_{n_2}(\tau)(a))\|_{H^s}.
\]
Using Proposition \ref{prop:estonQeps}, we get
\[
\|A_1(t)(a)\|_{H^s} \leq C_s T \sum_{n_1+n_2=n}\sup_{\tau \in [-T\varepsilon^{-2},T\varepsilon^{-2}]} \|\Psi_{n_1}(\tau)(a)\|_{H^s}\|\Psi_{n_2}(\tau)(a))\|_{H^s}.
\]
We use the induction hypothesis and get
\[
\|A_1(t)(a)\|_{H^s} \leq C_s T \sum_{n_1+n_2=n}c_{n_1}c_{n_2}8^{-n} D^2L.
\]
Choosing $T$ small enough such that $24 C_sT D\sqrt L \leq 1$, we get
\[
\|A_1(t)(a)\|_{H^s} \leq \frac13 c_{n+1}8^{-(n+1)} D\sqrt L.
\]
Similarly, using the bilinearity of $Q_\varepsilon$ and the bilinear estimates on $Q_\varepsilon$, we get
\[
\|A_1(t)(a) - A_1(t)(b)\|_{H^s} \leq \frac13 c_{n+1}8^{-(n+1)}\|a-b\|_{H^s}.
\]

For $t \in [-T\varepsilon^{-2},T\varepsilon^{-2}]$, we have
\[
\|A_2(t)(a)\|_{H^s} \leq T \sum_{n_1+n_2+n_3 = n-1} \sup_{\tau \in [-T\varepsilon^{-2},T\varepsilon^{-2}] }\|P_\varepsilon(\Psi_{n_1}(\tau)(a),\Psi_{n_2}(\tau)(a),\Psi_{n_3}(\tau)(a))\|_{H^s}.
\]
Using Proposition \ref{prop:estonTeps}, we get
\[
\|A_2(t)(a)\|_{H^s} \leq C_s L T \sum_{n_1+n_2+n_3 = n-1} \sup_{\tau \in [-T\varepsilon^{-2},T\varepsilon^{-2}] }\|\Psi_{n_1}(\tau)(a)\|_{H^s}\|\Psi_{n_2}(\tau)(a)\|_{H^s}\|\Psi_{n_3}(\tau)(a))\|_{H^s}.
\]
Using the induction hypothesis, we get
\[
\|A_2(t)(a)\|_{H^s} \leq C_s L T \sum_{n_1+n_2+n_3 = n-1} 8^{-n+1} c_{n_1}c_{n_2}c_{n_3}D^3L^{3/2}.
\]
Taking $T$ small enough such that $192 C_s T D^2 L^2 \leq 1$, we get
\[
\|A_2(t)(a)\|_{H^s} \leq c_{n+1} 8^{-n-1}DL^{1/2}.
\]
Similarly, using the trilinearity of $P_\varepsilon$, we get
\[
\|A_2(t)(a)-A_2(t)(b)\|_{H^s} \leq c_{n+1} 8^{-n-1}\|a-b\|_{H^s}.
\]

Assume $a$ is not a trigonometric polynomial. There exists a sequence $(a_m)_m$, such that $a_m \rightarrow a$ in $H^s$, and such that for all $m$
\[
\|a_m\|_{H^s} \leq D\sqrt L.
\]
By definition, we have 
\[
\Psi_{n+1}(t)(a) = -\varepsilon \sum_{n_1+n_2 = n} \int_{0}^t S(t-\tau)W[ \Psi_{n_1}(\tau)(a) \Psi_{n_2}(\tau)(a)] d\tau.
\]
Therefore, we get that 
\begin{multline*}
\Psi_{n+1}(t)(a) - \Psi_{n+1}(t)(a_m) 
= -\varepsilon \sum_{n_1+n_2 = n} \int_{0}^t S(t-\tau)W[ (\Psi_{n_1}(\tau)(a)-\Psi_{n_1}(\tau)(a_m)) \Psi_{n_2}(\tau)(a) \Big]\\
 -\varepsilon \sum_{n_1+n_2 = n} \int_{0}^t S(t-\tau)W[ \Psi_{n_1}(\tau) (a_m)(\Psi_{n_2}(\tau)(a) - \Psi_{n_2}(\tau)(a_m)) ] d\tau.
\end{multline*}
We deduce that for  $t \in [-T\varepsilon^{-2},T\varepsilon^{-2}]$, we have 
\begin{multline*}
\|\Psi_{n+1}(t)(a) - \Psi_{n+1}(t)(a_m)\|_{H^s} \\
\lesssim T\varepsilon^{-1} \sum_{n_1+n_2 = n} \sup_{\tau \in  [-T\varepsilon^{-2},T\varepsilon^{-2}]}[\|\Psi_{n_1}(\tau)(a)-\Psi_{n_1}(\tau)(a_m)\|_{H^s} \|\Psi_{n_2}(\tau)(a)\|_{H^s} \\
+ T\varepsilon^{-1} \sum_{n_1+n_2 = n} \sup_{\tau \in  [-T\varepsilon^{-2},T\varepsilon^{-2}]}\|\Psi_{n_1}(\tau) (a_m)\|_{H^s}\|\Psi_{n_2}(\tau)(a) - \Psi_{n_2}(\tau)(a_m)\|_{H^s}].
\end{multline*}
By the induction hypothesis, we get
\[
\|\Psi_{n+1}(t)(a) - \Psi_{n+1}(t)(a_m)\|_{H^s} \lesssim T\varepsilon^{-1} 8^{-n}D\sqrt L \sum_{n_1+n_2 = n} c_{n_1}c_{n_2}\|a-a_m\|_{H^s} .
\]
We deduce the continuity of $\Psi_{n+1}(t)$ and thus that the induction hypothesis is satisfied for any $a,b\in H^s$ and for $n+1$.
\end{proof}

\begin{corollary}\label{cor:LWP} The Cauchy problem \eqref{cauchyproblem} is well-posed in $\mathcal C([-\varepsilon^{-2}T_0(L,D),\varepsilon^{-2}T_0(L,D)], H^s)$ for initial datum $a \in H^s$ such that $\|a\|_{H^s}\leq D\sqrt L$, its flow $\Psi_{\varepsilon,L}$ can be written as the series
\[
\Psi_{\varepsilon,L} = \sum_{n} \Psi_{\varepsilon,L,n}
\]
and we have 
\[
\sup_{t\in [-\varepsilon^{-2}T_0(L,D),\varepsilon^{-2}T_0(L,D)]}\|\Psi_{\varepsilon,L,n} (t)(a)\|_{H^s} \leq 2^{-n} D\sqrt L.
\]
\end{corollary}

\begin{proof} Existence : Catalan numbers satisfy the estimate
\[
c_n \leq 4^n.
\]
We deduce that
\[
\|\Psi_{\varepsilon,L,n} (a)\|_{H^s} \leq 2^{-n} D\sqrt L
\]
and thus the series $\sum_n \Psi_{L,\varepsilon,n}$ converges, which gives a solution up to times of order $\varepsilon^{-2}$. 

Uniqueness : the uniqueness can be deduced by solving the fixed point problem
\[
u (t) = A_{\varepsilon,L,a,t_0}(u) (t) := e^{-(t-t_0)W} a - \varepsilon W \int_{t_0}^{t} e^{-(t-\tau)W} [u^2(\tau)]d\tau
\]
in $L^2(L\T)$ for times of order $|t-t_0| \lesssim \varepsilon^{-1}$, which is sufficient to get uniqueness by standard connectedness arguments.

Continuity in the initial datum : from the estimates on the Catalan numbers, we get that
\[
\|\Psi_{\varepsilon,L,n}(t) (a) - \Psi_{\varepsilon,L,n}(t) (b)\|_{H^s} \leq 2^{-n}\|a-b\|_{H^s}.
\]
Summing on $n$, we get
\[
\|\Psi_{\varepsilon,L}(t) (a) - \Psi_{\varepsilon,L}(t) (b)\|_{H^s} \leq 2\|a-b\|_{H^s}.
\]
\end{proof}

\section{Trees}\label{sec:trees}

In this section, but in this section only we use the letters $A$ and $\mathcal A$ to denote trees and sets of trees, the letter $T$ being already taken by time. 

\subsection{\texorpdfstring{Trees describing $\Psi_{\varepsilon,L,n}$}{Trees describing PsiEpsLn}}\label{subsec:FD}

We recall that the sequence $(\Psi_n(t)(a))_n$ for $a\in H^s$ (we omit the dependence in $(\varepsilon, L)$) is defined as follow:
\[
\Psi_0(t)(a) = S(t) a,\quad \Psi_{n+1}(t)(a) = -\varepsilon  W \int_{0}^t S(t-\tau) \Big[ \sum_{n_1+n_2=n} \Psi_{n_1}(\tau)(a)\Psi_{n_2}(\tau)(a)\Big]d\tau.
\]

We set $\psi(t)(a) = S(-t) \Psi(t)(a)$ and $\psi_n(t)(a) = S(-t) \Psi_n(t)(a)$, we get
\[
\psi_0(t)(a) = a,\quad \psi_{n+1}(t)(a) = -\varepsilon  W \int_{0}^t S(-\tau) \Big[ \sum_{n_1+n_2=n} S(\tau)\Big(\psi_{n_1}(\tau)(a)\Big)\, S(\tau)\Big(\psi_{n_2}(\tau)(a)\Big)\Big]d\tau.
\]

In Fourier mode, this becomes for $\xi \in \frac1{L} \Z^*$,
\[
\hat \psi_0(t) (a)(\xi) = \hat a(\xi), \quad \hat\psi_{n+1}(t)(a) = \frac{-i\varepsilon}{\sqrt{2\pi L}} \omega(\xi) \sum_{n_1+n_2=n}\sum_{\xi_1+\xi_2 = \xi} \int_{0}^t  e^{i\Delta^\xi_{\xi_1,\xi_2}\tau} \hat\psi_{n_1}(\tau)(a)\hat\psi_{n_2}(\tau)(a)d\tau.
\]

Let $\mathcal A_n$ be the set of binary trees with $n$ nodes. We denote by $\bot$ the tree with $0$ node such that
\[
\mathcal A_0 = \{\bot\},\quad \mathcal A_1 = \{(\bot,\bot)\}
\]
etc... We recall that $\# \mathcal A_n = c_n \leq 4^n$ and that a tree with $n$ nodes has $n+1$ leaves.

\vspace{1cm}

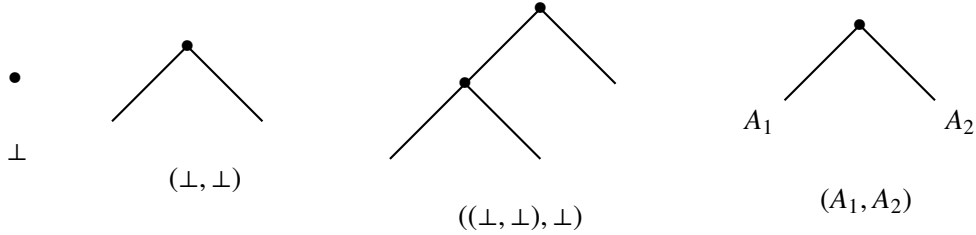
\begin{minipage}{0,5cm}
\begin{tikzpicture}
\draw (0,0) node {$\bullet$};
\end{tikzpicture}
\[
\bot
\]
\end{minipage}\hspace{1cm}\begin{minipage}{2,5cm}
\begin{tikzpicture}
\draw (0,0) node {$\bullet$};
\draw[thick] (0,0)--(-1,-1);
\draw[thick] (0,0)--(1,-1);
\end{tikzpicture}
\[
(\bot,\bot)
\]
\end{minipage} \hspace{1cm} \begin{minipage}{3,5cm}
\begin{tikzpicture}
\draw (0,0) node {$\bullet$};
\draw (-1,-1) node {$\bullet$};
\draw[thick] (0,0)--(-1,-1);
\draw[thick] (0,0)--(1,-1);
\draw[thick] (-1,-1)--(-2,-2);
\draw[thick] (-1,-1)--(0,-2);
\end{tikzpicture}
\[
((\bot,\bot),\bot)
\]
\end{minipage}\hspace{1cm} \begin{minipage}{3,5cm}
\begin{tikzpicture}
\draw (0,0) node {$\bullet$};
\draw (-1,-1) node[below left] {$A_1$};
\draw (1,-1) node[below right] {$A_2$};
\draw[thick] (0,0)--(-1,-1);
\draw[thick] (0,0)--(1,-1);
\end{tikzpicture}
\[
(A_1,A_2)
\]
\end{minipage}

\vspace{1cm}

Above, we have represented the tree with $0$ node $\bot$, the tree with one node $(\bot,\bot)$, one tree with two nodes $((\bot,\bot),\bot)$ and the generic way of writing binary trees by induction $(A_1,A_2)$. For instance $((\bot,\bot),\bot)) = (A_1,A_2)$ with $A_1 = (\bot,\bot) $ and $A_2 = \bot$. 


We describe $\psi_n$ thanks to $\mathcal A_n$.

\begin{proposition}\label{prop:psinwithTn1} Let $n\in \N$, $a\in H^s$ and $\xi \in \frac1{L}\Z$, we have 
\[
\hat \psi_n(t)(a) (\xi) = \frac{(-i\varepsilon)^n}{(2\pi L)^{n/2}} \sum_{A\in \mathcal A_n} \sum_{\xi_1+\hdots +\xi_{n+1} = \xi} F_A(t,\xi_1,\hdots,\xi_{n+1}) \prod_{l=1}^{n+1}\hat a(\xi_l)
\]
where $F_A$ is defined by induction as $F_\bot (t,\xi) =1$ and for all $A_1 \in \mathcal A_{n_1}$, $A_2 \in \mathcal A_{n_2}$, $n_1,n_2 \in \N$, 
\[
F_{(A_1,A_2)}(t,\xi_1,\hdots,\xi_{n_1+n_2 + 2}) = \int_{0}^t \omega(\bar \xi) e^{i\Delta^{\bar \xi}_{\bar \xi_1,\bar \xi_2}\tau} F_{A_1}(\tau, \xi_1,\hdots, \xi_{n_1+1})F_{A_2}(\tau, \xi_{n_1+2},\hdots , \xi_{n_1+n_2+2})d\tau
\]
where we used the notations $\bar \xi = \sum_{l=1}^{n_1+n_2+2}\xi_l$, $\bar \xi_1 = \sum_{l=1}^{n_1 + 1} \xi_l$ and $\bar \xi_2 = \bar \xi - \bar \xi_1$.
\end{proposition}

\begin{proof} By induction on $n$. For $n=0$, we have 
\[
\hat \psi_n(t)(a) (\xi) = \hat a (\xi).
\]

Assume that the formula is true for $m\leq n$ and let us prove it for $n+1$. We have 
\[
\hat\psi_{n+1}(t)(a) = \frac{-i\varepsilon}{\sqrt{2\pi L}} \omega(\xi) \sum_{n_1+n_2=n}\sum_{\xi_1+\xi_2 = \xi} \int_{0}^t  e^{i\Delta^\xi_{\xi_1,\xi_2}\tau} \hat\psi_{n_1}(\tau)(a)\hat\psi_{n_2}(\tau)(a)d\tau.
\]
We use the induction hypothesis and get
\begin{multline*}
\hat\psi_{n+1}(t)(a) = \\
\frac{-i\varepsilon}{\sqrt{2\pi L}} \omega(\xi) \sum_{n_1+n_2=n}\sum_{\xi_1+\xi_2 = \xi} \int_{0}^t  e^{i\Delta^\xi_{\xi_1,\xi_2}\tau} \Big(\frac{(-i\varepsilon)^{n_1}}{(2\pi L)^{n_1/2}} \sum_{A_1\in \mathcal A_{n_1}} \sum_{\eta_1+\hdots +\eta_{n_1+1} = \xi} F_{A_1}(\tau,\eta_1,\hdots,\eta_{n_1+1}) \prod_{l=1}^{n_1+1}\hat a(\eta_l)\Big)\\
\times \Big(\frac{(-i\varepsilon)^{n_2}}{(2\pi L)^{n_2/2}} \sum_{A_2\in \mathcal A_{n_2}} \sum_{\eta_{n_1+2}+\hdots +\eta_{n_1+n_2+2} = \xi} F_{A_2}(\tau,\eta_{n_1+2},\hdots,\xi_{n_1+ n_2+2}) \prod_{l=n_1+2}^{n_1+n_2+2}\hat a(\eta_l)\Big)d\tau.
\end{multline*}
We recognize
\[
\hat\psi_{n+1}(t)(a) = \frac{(-i\varepsilon)^{n+1}}{(2\pi L)^{(n+1)/2}}  \sum_{n_1+n_2=n}\sum_{\eta_1+\hdots +\eta_{n+2} = \xi} \sum_{A_1\in \mathcal A_{n_1}}  \sum_{A_2\in \mathcal A_{n_2}} F_{(A_1,A_2)}(t,\eta_1,\hdots,\eta_{n+2})\prod_{l=1}^{n+2}\hat a(\eta_l).
\]
It remains to use that
\[
\mathcal A_{n+1} = \{(A_1,A_2) \; |\; \forall j=1,2,\; A_j \in \mathcal A_{n_j}, n_1+n_2 = n\}
\]
to conclude.
\end{proof}

We label the nodes in a tree $A$.

\begin{definition} Let $A \in \cup \mathcal A_n$. We define $N(A)$ the set of node labels of $A$ as
\[
N(\bot) = \emptyset, \quad N((A_1,A_2) ) = \{0\} \sqcup \{1\}\times N(A_1) \sqcup \{2\}\times N(A_2).
\]
\end{definition}

\begin{remark} We interpret this label as follows: $m\in N(A)$ is the label of the root if $m=0$, it is the label of the $m'$ node in $A_1$ if $m = (1,m')$ and it is the label of the $m'$ node in $A_2$ if $m=(2,m')$. We give two examples below.

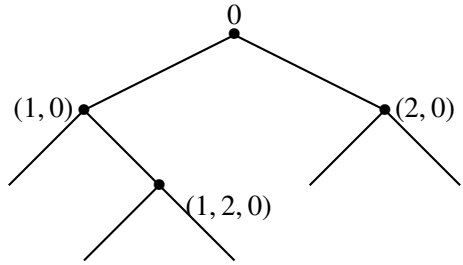
\begin{center}
\begin{tikzpicture}
\draw (0,0) node {$\bullet$};
\draw (0,0) node [above] {$0$};
\draw (-2,-1) node {$\bullet$};
\draw (-2, -1) node[left] {$(1,0)$};
\draw (-1,-2) node {$\bullet$};
\draw (-1,-2) node [below right] {$\;\;(1,2,0)$};
\draw (2,-1) node {$\bullet$};
\draw (2,-1) node [right] {$(2,0)$};
\draw[thick] (0,0)--(-2,-1);
\draw[thick] (-2,-1)--(-3,-2);
\draw[thick] (-2,-1)--(-1,-2);
\draw[thick] (-1,-2)--(-2,-3);
\draw[thick] (-1,-2)--(0,-3);
\draw[thick] (0,0)--(2,-1);
\draw[thick] (2,-1)--(1,-2);
\draw[thick] (2,-1)--(3,-2);
\end{tikzpicture}
\end{center}

Above, we used : 

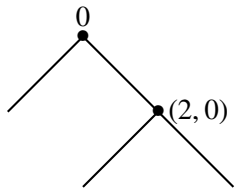
\begin{center}
\begin{tikzpicture}
\draw (0,0) node {$\bullet$};
\draw (0,0) node [above] {$0$};
\draw (1,-1) node {$\bullet$};
\draw (1,-1) node [right] {$(2,0)$};
\draw[thick] (0,0)--(-1,-1);
\draw[thick] (0,0)--(1,-1);
\draw[thick] (1,-1)--(0,-2);
\draw[thick] (1,-1)--(2,-2);
\end{tikzpicture}
\end{center}

\end{remark}

\begin{definition} Let $\leq_A$ be the (partial) order on $N(A)$ defined for $A = (A_1,A_2)$ as 
\[
m \leq_A j \quad \Leftrightarrow j=0 \textrm{ or } \exists k\in \{1,2\},\; m=(k,m'),\; j=(k,j') \textrm{ and } m'\leq_{A_k} j'.
\]

We also set for $t\geq 0$,
\[
I_A(t) = \{ (t_{m})_{m\in N(A)}\;|\;  0 \leq t_0\leq t \; \textrm{ and } \forall m\leq_A j,\; t_m \leq t_j\}.
\]
\end{definition}

\begin{remark} The order on $N(A)$ corresponds to parenthood. We explain why it is an order by induction on $n$. 

For $n=0$, we have $N(\bot) = \emptyset$. 

Assume that $\leq_A$ is an order for any $A$ in $\mathcal A_m$ with $m\leq n$, we prove that $\leq_A$ is an order for any $A \in \mathcal A_{n+1}$. Let $A = (A_1,A_2) \in \mathcal A_{n+1}$. We have that $\leq_{A_1}$ and $\leq_{A_2}$ are orders.

Reflexivity : We have that $0\leq_A 0$. Then if $m=(k,m')$ with $k=1,2$, by reflexivity of $\leq_{A_k}$, we have that $m'\leq_{A_k} m'$ and thus $m\leq_A m$.

Skew-symmetry : Let $m,j$ such that $m\leq_A j$ and $j\leq_A m$. If $j=0$ then $0 \leq_A m$ and thus $m=0=j$. Otherwise, there exists $k=1,2$ such that $j=(k,j')$ and $m=(k,m')$ and $m'\leq_{A_k}j'$ and $j'\leq_{A_k} m'$. By reflexivity of $\leq_{A_k}$, we deduce $j'=m'$ and thus $j=m$.

Transitivity : Let $m,j,l$ be such that $m\leq_A j$ and $j\leq_A l$. If $l=0$ then $l\geq_A m$. If $l\neq 0$ then there exists $k=1,2$ such that $l=(k,l')$, $j= (k,j')$ and $j'\leq_{A_k}l'$. Because $ m\leq_A j$, we deduce that $m=(k,m')$ with $m'\leq_{A_k} j'$. By transitivity of $\leq_{A_k}$, we have that $m'\leq_{A_k}l'$ and thus $m\leq_A l$.
 
\end{remark}

\begin{definition} Let $A \in \mathcal A_{n}$ for $n\geq 1$ and set $A=(A_1,A_2)$ with $A_j \in \mathcal A_{n_j}$. Let $\vec \xi = (\xi_1,\hdots, \xi_{n+1}) \in \frac1{L} \Z^{n+1}$ and set $\bar \xi = \sum_l \xi_l$, $\bar \xi_1 = \sum_{l\leq n_1+1} \xi_l$ and $\bar \xi_2 = \bar \xi - \bar \xi_l$. We define for $m\in N(A)$,
\[
\Delta_A(m, \vec \xi) =\left \lbrace{\begin{array}{cc} 
\Delta^{\bar \xi}_{\bar \xi_1,\bar \xi_2} & \textrm{ if } m=0\\
\Delta_{A_1}(m', \xi_1,\hdots ,\xi_{n_1+1}) & \textrm{ if } m=(1,m') \\
\Delta_{A_2}(m', \xi_{n_1+2},\hdots, \xi_{n+1}) & \textrm{ if } m=(2,m')
\end{array}} \right.
\]
and
\[
\omega_{ A}(m,\vec \xi) =  \left \lbrace{\begin{array}{cc} 
\omega(\bar \xi) & \textrm{ if } m=0\\
\omega_{A_1}(m', \xi_1,\hdots ,\xi_{n_1+1}) & \textrm{ if } m=(1,m') \\
\omega_{A_2}(m', \xi_{n_1+2},\hdots, \xi_{n+1}) & \textrm{ if } m=(2,m')
\end{array}} \right. .
\]
\end{definition}

\begin{remark} The quantities $\Delta_A(m,\vec \xi)$ and $\omega_A(m,\vec \xi)$ are the $\omega$ and $\Delta$ that appear when we do the interaction described by the node $m$.
\end{remark}

\begin{proposition} We have, with the above notations for all $A\in \mathcal A_n$, $n\geq 1$, $t\geq 0$,
\[
F_A(t,\vec \xi) = \int_{I_A(t)} \prod_{m \in N(A)} \Big( \omega_{A}(m,\vec \xi) e^{i\Delta_A(m,\vec \xi)t_m}dt_m\Big)
\]
\end{proposition}

\begin{proof} We proceed by induction on $n$. If $n=1$, then $A = (\bot,\bot)$. We have $N(A) = \{0\}$, $I_A(t) = [0,t]$, $\omega_{A}(0,\xi_1,\xi_2) = \omega (\xi_1+\xi_2) = \omega(\bar \xi)$ and finally, $\Delta_A(0,\xi_1,\xi_2) = \Delta^{\xi_1+\xi_2}_{\xi_1,\xi_2} = \Delta^{\bar \xi}_{\bar \xi_1,\bar \xi_2}$, and the above equality corresponds to
\[
F_A(t,\vec \xi) = \int_{0}^t   \omega(\bar \xi) e^{i\Delta^{\bar \xi}_{\xi_1,\xi_2}t_0}dt_0
\]
which is true.

Assume the equality is valid for $m\leq n$. We claim that it is valid for $n+1$. By definition, we have, setting $A=(A_1,A_2)$,
\[
F_{A}(t,\vec \xi) = \int_{0}^t \omega(\bar \xi) e^{i\Delta^{\bar \xi}_{\bar \xi_1,\bar \xi_2}\tau} F_{A_1}(\tau, \xi_1,\hdots, \xi_{n_1+1})F_{A_2}(\tau, \xi_{n_1+2},\hdots , \xi_{n_1+n_2+2})d\tau .
\]
We recognize
\[
F_{A}(t,\vec \xi) = \int_{0}^t \omega_{A}(0,\vec \xi) e^{i\Delta_A(0,\vec \xi)\tau} F_{A_1}(\tau, \xi_1,\hdots, \xi_{n_1+1})F_{A_2}(\tau, \xi_{n_1+2},\hdots , \xi_{n_1+n_2+2})d\tau .
\]
We use the induction hypothesis and get
\begin{multline*}
F_{A}(t,\vec \xi) = \int_{0}^t \omega_{A}(0,\vec \xi) e^{i\Delta_A(0,\vec \xi)\tau} \int_{I_{A_1}(\tau)} \prod_{m \in N(A_1)} \Big( \omega_{A_1}(m,\vec \xi_1) e^{i\Delta_{A_1}(m,\vec \xi_1)s_m}ds_m\Big)\\
\int_{I_{A_2}(\tau)} \prod_{m \in N(A_2)} \Big( \omega_{A_2}(m,\vec \xi_2) e^{i\Delta_{A_2}(m,\vec \xi_2)s'_m}ds'_m\Big)d\tau 
\end{multline*}
where $\vec \xi_1 = (\xi_1,\hdots,\xi_{n_1+1})$, $\vec \xi_2 = (\xi_{n_1+2},\hdots,\xi_{n+2})$.

Setting $t_{(1,m)} = s_m$, $t_{(2,m)} = s'_m$ and $t_0 = \tau$, we get
\begin{multline*}
F_{A}(t,\vec \xi) = 
\int_{0}^t dt_0  \omega_{A}(0,\vec \xi) e^{i\Delta_A(0,\vec \xi)t_0} \int_{I_{A_1}(t_0)} \prod_{m \in N(A_1)} \Big( \omega_{A}((1,m),\vec \xi) e^{i\Delta_{A}((1,m),\vec \xi)t_{(1,m)}}dt_{(1,m)}\Big)\\
\times \int_{I_{A_2}(t_0)} \prod_{m \in N(A_2)} \Big( \omega_{A}((2,m),\vec \xi)e^{i\Delta_{A}((2,m),\vec \xi)t_{(2,m)}}dt_{(2,m)}\Big)d\tau .
\end{multline*}
We recall that
\[
I_A(t) = \{ (t_m)_{m\in N(A)} \;| \; \forall m\leq_A j, \; t_m \leq t_j\}.
\]
And that $m\leq_A j$ iff $j=0$ or $m=(k,m')$ and $j=(k,j')$ and $m'\leq_{A_k} j'$ for $k=1$ or $k=2$. We deduce
\[
I_A(t) = \{(t_m)_{m\in N(A)} \; |\; t_0 \in [0,t]\, \wedge \, (t_{(1,m)})_{m\in N(A_1)} \in I_{A_1}(t_0) \, \wedge \, (t_{(2,m)})_{m\in N(A_2)} \in I_{A_2}(t_0)\}.
\]
Therefore, 
\[
F_{A}(t,\vec \xi) = \int_{I_A(t)} \prod_{m\in N(A)}dt_m  \prod_{m\in N(A)} (\omega_{A}(m,\vec \xi) e^{i\Delta_A(m,\vec \xi)t_m})  .
\]
\end{proof}

\begin{notation} For a given tree $A \in \mathcal A_n$, $n\geq 0$, we set $\mathfrak S_A$ the set of bijections $\sigma$ from $[|1,n|]$ to $N(A)$ such that
\[
\forall (m,j) \in N(A), \; m\leq_A j \Rightarrow \sigma^{-1}(m) \leq \sigma^{-1}(j).
\]
By convention, we have that $\mathfrak S_\bot = \{\emptyset\}$ where $\emptyset$ is the empty map $\emptyset : \emptyset \rightarrow \emptyset$.
\end{notation}

\begin{remark} The set $\mathfrak{S}_A$ is the set of all the total orders on $N(A)$ that respect the partial order $\leq_A$. Below, we give an example of all the elements of $\mathfrak S_A$ for $A=((\bot,(\bot,\bot)),(\bot,\bot))$.

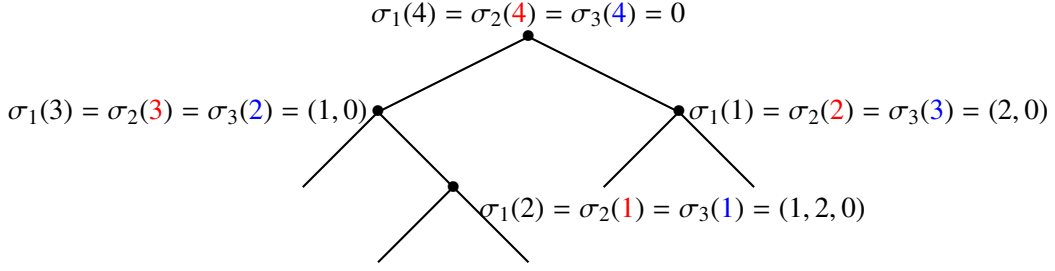
\begin{center}
\begin{tikzpicture}
\draw (0,0) node {$\bullet$};
\draw (0,0) node [above] {$\sigma_1(4)= \sigma_2(\textcolor{red}{4}) = \sigma_3(\textcolor{blue}{4}) = 0$};
\draw (-2,-1) node {$\bullet$};
\draw (-2, -1) node[left] {$\sigma_1(3)= \sigma_2(\textcolor{red}{3}) = \sigma_3(\textcolor{blue}{2})= (1,0)$};
\draw (-1,-2) node {$\bullet$};
\draw (-1,-2) node [below right] {$\;\;\sigma_1(2)= \sigma_2(\textcolor{red}{1}) = \sigma_3(\textcolor{blue}{1})=(1,2,0)$};
\draw (2,-1) node {$\bullet$};
\draw (2,-1) node [right] {$\sigma_1(1)= \sigma_2(\textcolor{red}{2}) = \sigma_3(\textcolor{blue}{3})=(2,0)$};
\draw[thick] (0,0)--(-2,-1);
\draw[thick] (-2,-1)--(-3,-2);
\draw[thick] (-2,-1)--(-1,-2);
\draw[thick] (-1,-2)--(-2,-3);
\draw[thick] (-1,-2)--(0,-3);
\draw[thick] (0,0)--(2,-1);
\draw[thick] (2,-1)--(1,-2);
\draw[thick] (2,-1)--(3,-2);
\end{tikzpicture}
\end{center}

A priori, if we do not require that $\sigma$ has to respect the parenthood order, then we can choose $4!$ different bijections from $[|1,4|]$ to $N(A)$. But $\sigma^{-1}(0)$ has to be bigger than any of the other which imposes $\sigma^{-1}(0) = 4$. Then, $(1,0)$ is bigger than $(1,2,0)$ but $(2,0)$ is not comparable to $(1,0)$ and $(1,2,0)$ which means that we can choose any value among $\{,2,3\}$ for $\sigma^{-1}(2,0)$ and that sets completely $\sigma$. Therefore, we have three elements in $\mathfrak S_A$.
\end{remark}

\begin{proposition} We have for all $A \in \mathcal A_n$, $n\geq 1$, $t\geq 0$,
\[
F_A(t,\vec \xi) = \sum_{\sigma \in \mathfrak S_A} F_A^\sigma (t,\vec \xi)
\]
with 
\[
F_A^\sigma(t,\vec \xi) =  \int_{0\leq t_1 \leq \hdots \leq t_n\leq t} \prod_{m=1}^n dt_m  \prod_{m=1}^n (\omega_{A}(\sigma(m),\vec \xi) e^{i\Delta_A(\sigma(m),\vec \xi)t_m}) .
\]
\end{proposition}

\begin{proof} We have 
\[
I_A(t) = \bigcup_{\sigma \in \mathfrak S_A} I_A^\sigma(t)
\]
where 
\[
I_A^\sigma (t) = \{ (t_m)_{m \in N(A)} \; |\; 0\leq t_{\sigma(1)} \leq \hdots \leq t_{\sigma(n)}\leq t \}.
\]
Indeed, the $(t_m)_{m\in N(A)}$ might be in any (total) order that respect the (partial order) of parenthood. 

We have that $I_A^{\sigma}(t) \cap I_A^{\sigma'}(t)$ is of null Lebesgue measure if $\sigma\neq \sigma'$, therefore,
\[
F_A(t,\vec \xi) = \sum_{\sigma \in \mathfrak S_A} F_A^\sigma (t,\vec \xi)
\]
with 
\[
F_A^\sigma(t,\vec \xi) = \int_{I^\sigma_A(t)} \prod_{m\in N(A)}dt_m  \prod_{m\in N(A)} (\omega_{A}(m,\vec \xi) e^{i\Delta_A(m,\vec \xi)t_m})  .
\]
We rename the variables $t_m = t_{\sigma(m)}$ to get the result.
\end{proof}

\subsection{Estimates on \texorpdfstring{$\Psi_n$}{Psin}}\label{subsec:estimates}

\begin{proposition}\label{prop:estFsigma} There exists $C$ such that for all $n\in \N$, for all $L \geq 1$,for all $R\geq 1$, there exists $C(L,R,n)$ such that for all $t\geq 0$, for all $A \in \mathcal A_n$, $\vec \xi \in ([-R,R] \cap \frac1{L} \Z)^{n+1}$, and for all $\sigma \in \mathfrak S_A$, we have 
\[
|F_A^\sigma(t,\vec \xi)| \leq C^n L^{2n}t^{(n-1)/2} C(L,R,n) 
\]
if $n$ is odd,
\begin{multline*}
\Big| F_A^\sigma(t,\vec \xi) -  \prod_{m=1}^{n/2} \frac{\omega_{A}(\sigma(2m),\vec \xi)\omega_{A}(\sigma(2m-1),\vec \xi)}{i\Delta_{A}(\sigma(2m-1),\vec \xi)}\delta\Big(\Delta_{A}(\sigma(2m-1),\vec \xi)+\Delta_{A}(\sigma(2m),\vec\xi)\Big) \frac{t^{n/2}}{(n/2)!} \Big| \\
\leq C^n L^{2n}t^{(n-2)/2}C(L,R,n)
\end{multline*}
\end{proposition}

To prove this proposition, we use the following lemma.

\begin{lemma} Let $L\geq 1$, $R\geq 0$, $n\geq 2$, there exists $C(R,L,n) \geq 1$ such that for all $A\in \mathcal A_n$, $m_1\neq m_2\in N(A)$, and let $\vec \xi \in \Big[ [-R,R]\cap \frac1{L}\Z\Big]^{n+1}$, we have either
\[
\Delta_A(m_1,\vec \xi) + \Delta_A(m_2,\vec \xi) = 0,
\]
or
\[
|\Delta_A(m_1,\vec \xi) + \Delta_A(m_2,\vec \xi)| \geq C(L,R,n)^{-1}.
\]
\end{lemma}

\begin{proof} It is sufficient to see that if $L,R,n$ are fixed, then the set
\[
\{ (A,m_1,m_2,\vec \xi) \;| \; A\in \mathcal A_n,\, m_1\neq m_2 \in N(A),\, \vec \xi \in \Big[ [-R,R]\cap \frac1{L}\Z\Big]^{n+1}\}
\]
is finite.

\end{proof}

\begin{proof}[Proof of Proposition \ref{prop:estFsigma}] 

To lighten the notations, we set $\omega_k = \omega_{A}(\sigma(k),\vec \xi)$, $\Delta_k = \Delta_A(\sigma(k),\vec \xi)$ and 
\[
F_k(t) = \int_{0\leq t_1\leq \hdots \leq t_k\leq t} \prod_{m=1}^k \omega_m e^{i\Delta_m t_m}.
\]
We have that
\[
F_A^\sigma(t,\vec \xi) = F_n(t).
\]

If $t\leq 1$, then 
\[
|F_k(t) |\leq \frac{t^k}{k!} 
\]
from which we deduce that
\[
|F_k(t)|\leq C^k t^{(k-1)/2}L^{2k} C(L,R,n) 
\]
if $k$ is odd.

If $k$ is even, we remark that 
\[
\Big|\frac{\omega_m}{\Delta_m} \Big| \leq 4L^2.
\] 
Since we have
\[
\Big|\prod_{m=1}^{k/2} \frac{\omega_{2m} \omega_{2m-1}\delta(\Delta_{2m-1} + \Delta_{2m})}{i\Delta_{2m-1}}\Big|\leq \frac{t^{k/2}}{(k/2)!} L^k,
\]
we deduce that
\[
\Big|F_k(t) - \frac{t^{k/2}}{(k/2)!} \prod_{m=1}^{k/2} \frac{\omega_{2m} \omega_{2m-1}\delta(\Delta_{2m-1} + \Delta_{2m})}{i\Delta_{2m-1}}\Big| \leq C^k t^{(k-2)/2}L^{2k} C(L,R,n) 
\]
if $k$ is even.

If $t\geq 1$, we prove by induction on $k$ that
\[
|F_k(t)| \leq C^k t^{(k-1)/2}L^{2k} C(L,R,n) 
\]
if $k$ is odd and
\[
\Big|F_k(t) -\frac{t^{k/2}}{(k/2)!}  \prod_{m=1}^{k/2} \frac{\omega_{2m} \omega_{2m-1}\delta(\Delta_{2m-1} + \Delta_{2m})}{i\Delta_{2m-1}}\Big| \leq C^k t^{(k-2)/2}L^{2k} C(L,R,n) 
\]
if $k$ is even.

If $k=0$, we have $F_k(t) = 1$ and 
\[
\prod_{m=1}^{k/2} \frac{\omega_{2m} \omega_{2m-1}\delta(\Delta_{2m-1} + \Delta_{2m})}{i\Delta_{2m-1}} = 1
\]
hence the result is true.

If $k=1$, then 
\[
|F_k(t)| = \Big|\omega_1 \frac{e^{i\Delta_1 t}-1}{\Delta_1} \Big|\leq 8L^2
\]
which makes the result true if $C\geq 8$.

If $k$ is even, we write
\[
G_k = \frac1{(k/2)!}\prod_{m=1}^{k/2} \frac{\omega_{2m} \omega_{2m-1}\delta(\Delta_{2m-1} + \Delta_{2m})}{i\Delta_{2m-1}}.
\]
We have 
\[
|G_k| \leq 4^{k/2} L^k.
\]

Let $k\geq 2$, we have 
\[
F_k(t) = \int_{0}^t dt_k \omega_k e^{i\Delta_k t_k} F_{k-1}(t_k).
\]
We integrate by parts and get
\[
F_k(t) = [\omega_k \frac{e^{i\Delta_k t_k}}{i\Delta_k}F_{k-1}(t_k)]_0^t - \int_{0}^t \omega_k \frac{e^{i\Delta_k t_k}}{i\Delta_k} F'_{k-1}(t_k) dt_k.
\]
This yields
\[
F_k(t) = \omega_k \frac{e^{i\Delta_k t}}{i\Delta_k}F_{k-1}(t)-\int_{0}^t  \omega_k \omega_{k-1}\frac{e^{i(\Delta_k + \Delta_{k-1}) t_k}}{i\Delta_k} F_{k-2}(t_k) dt_k.
\]

Assume that $k$ is even. Then $k-1$ is odd and we get
\[
\Big| \omega_k \frac{e^{i\Delta_k t}}{i\Delta_k}F_{k-1}(t) \Big| \leq 4 L^2  C^{k-1}t^{(k-2)/2}L^{2(k-1)} C(L,R,n) ,
\]
that is
\[
\Big| \omega_k \frac{e^{i\Delta_k t}}{i\Delta_k}F_{k-1}(t) \Big| \leq 4   C^{k-1} t^{(k-2)/2}L^{2k} C(L,R,n) ,
\]
We divide 
\[
\int_{0}^t  \omega_k \omega_{k-1}\frac{e^{i(\Delta_k + \Delta_{k-1}) t_k}}{i\Delta_k} F_{k-2}(t_k) dt_k = I+II
\]
with
\[
I = \int_{0}^t  \omega_k \omega_{k-1}\frac{e^{i(\Delta_k + \Delta_{k-1}) t_k}}{i\Delta_k} (F_{k-2}(t_k) - G_{k-2}t_k^{(k-2)/2}) dt_k
\]
and
\[
II = \int_{0}^t  \omega_k \omega_{k-1}\frac{e^{i(\Delta_k + \Delta_{k-1}) t_k}}{i\Delta_k} G_{k-2}t_k^{(k-2)/2} dt_k.
\]
We have that
\[
|F_{k-2}(t_k) - G_{k-2}t_k^{(k-2)/2}| \leq C^{k-2} t^{(k-4)/2}L^{2(k-2)} C(L,R,n) 
\]
from which we get
\[
I \leq 4 C^{k-2} t^{(k-2)/2}L^{2(k-1)}  C(L,R,n) 
.
\]
For $II$, if $\Delta_k + \Delta_{k-1} = 0$, then
\[
II = G_k t^{k/2}.
\]
Otherwise, if $k=2$, we have 
\[
II \leq 4L^2 \frac2{|\Delta_2 + \Delta_1|} \leq 4 L^2 C(L,R,n)
\]
and if $k>2$, we can integrate by parts and get
\[
II = [\omega_k \omega_{k-1}\frac{e^{i(\Delta_k + \Delta_{k-1}) t_k}}{-\Delta_k(\Delta_k + \Delta_{k-1})}G_{k-2}t_k^{(k-2)/2}]_0^t + \int_{0}^t \omega_k \omega_{k-1}\frac{e^{i(\Delta_k + \Delta_{k-1}) t_k}}{\Delta_k(\Delta_k + \Delta_{k-1})}G_{k-2}\frac{k-2}{2}t_k^{(k-4)/2}dt_k.
\]
We have
\[
|G_{k-2}| \leq 4^{(k-2)/2} L^{k-2}, \quad \Big|\omega_k \omega_{k-1}\frac{e^{i(\Delta_k + \Delta_{k-1}) t_k}}{\Delta_k(\Delta_k + \Delta_{k-1})} \Big| \leq 4 L^2 C(L,R,n)
\]
from which we deduce
\[
II \leq 2\cdot 4^{k/2} L^k C(L,R,n) t^{(k-2)/2}.
\]
Summing up, we get
\[
|F_k(t)| \leq 4 C^{k-2} t^{(k-2)/2}L^{2(k-2)} C(L,R,n)  + 2\cdot 4^{k/2} L^{k} C(L,R,n)
\]
and therefore
\[
|F_k(t)| \leq \Big( 4C^{k-2} + 2^{k+1} \Big) L^{2k}   C(L,R,n) .
\]
If $C^k \geq 4C^{k-2} + 2^{k+1}$, for instance if $C\geq 2\sqrt 3$, we get the result.

Assume now that $k\geq 3$ is odd. We still have
\[
F_k(t) = \omega_k \frac{e^{i\Delta_k t}}{i\Delta_k}F_{k-1}(t)-\int_{0}^t  \omega_k \omega_{k-1}\frac{e^{i(\Delta_k + \Delta_{k-1}) t_k}}{i\Delta_k} F_{k-2}(t_k) dt_k.
\]
We have that $k-1$ is even, and therefore,
\[
\Big| \omega_k \frac{e^{i\Delta_k t}}{i\Delta_k}F_{k-1}(t) \Big| \leq 4L^2 \Big( C^{k-1} t^{(k-3)/2}L^{2(k-1)} C(L,R,n)  + 2^{k-1}L^{(k-1)} t^{(k-1)/2}\Big).
\]
We deduce
\[
\Big| \omega_k \frac{e^{i\Delta_k t}}{i\Delta_k}F_{k-1}(t) \Big| \leq t^{(k-1)/2} \Big( 4C^{k-1} L^{2k)}  C(L,R,n) + 2^{k+1}L^{k+1} \Big).
\]
We also have
\[
|F_{k-2}(t_k) |\leq C^{k-2} t_k^{(k-3)/2}L^{2(k-2)}  C(L,R,n) 
\]
from which we deduce
\[
\Big| \int_{0}^t  \omega_k \omega_{k-1}\frac{e^{i(\Delta_k + \Delta_{k-1}) t_k}}{i\Delta_k} F_{k-2}(t_k) dt_k\Big| \leq 4L^2C^{k-2} t^{(k-1)/2}L^{2(k-2)} C(L,R,n)  .
\]
Therefore, we get
\[
|F_k(t)|\leq t^{(k-1)/2}L^{2k} C(L,R,n) \Big( 4C^{k-1} + 2^{k+1} +4C^{k-2}\Big)
\]
and we get the result if $\Big( 4C^{k-1} + 2^{k+1} +4C^{k-2}\Big)\leq C^k$ for instance if $C\geq 10$. 
\end{proof}

\begin{definition} Let $n \in 2\N$, $A\in \mathcal A_n$, $\sigma \in \mathfrak S_A$. Let $\vec \xi \in (\frac1{L}\Z^*)^{n+1}$ and $t\in \R_+$. We set
\[
\tilde F_A^\sigma(t,\vec \xi) = \prod_{m=1}^{n/2} \frac{\omega_{A}(\sigma(2m),\vec \xi)\omega_{A}(\sigma(2m-1),\vec \xi)}{i\Delta_{A}(\sigma(2m-1),\vec \xi)}\delta(\Delta_{A}(\sigma(2m-1),\vec \xi)+\Delta_{A}(\sigma(2m),\vec\xi)) \frac{t^{n/2}}{(n/2)!}.
\]
For $n\in 2\N$, $t\in \R$ and $a\in H^s$, we define accordingly $\tilde \psi_n(t)(a)$ in Fourier mode as
\[
\widehat{\tilde \psi}_n(t)(a)(\xi) = \frac{(-i\varepsilon)^n}{(2\pi L)^{n/2}}\sum_{A\in \mathcal A_n} \sum_{\sigma \in \mathfrak S_A} \sum_{\xi_1+\hdots+ \xi_{n+1} = \xi} \tilde F_A^\sigma(t,\vec \xi) \prod_{l=1}^{n+1}\hat a(\xi_l).
\]
For $n\in 2\N$, $t\in \R$ and $a\in H^s$, we set accordingly $\tilde\Psi_n(t)(a) = S(t) \tilde \psi_n(t)(a)$.
\end{definition}

\begin{proposition}\label{prop:sndestonPsin}
Let $s>\frac14$. Let $L,D,R\geq 1$. Let $a\in H^s$ such that $\|a\|_{H^s}\leq D\sqrt L$.  Set $a_R$ defined in Fourier mode as
\[
\hat a_R(\xi) = {\bf 1}_{|\xi|\leq R} \hat a(\xi).
\]
For all $n\in \N$, there exists $C(L,D,R,n)$ such that with the notations of Proposition \ref{prop:fstestonPsin}, we have that for all $T\leq T_0(L,D)$ and for all $\varepsilon \leq \epsilon_0(L,D)$,  all $t\in [-T\varepsilon^{-2}, T\varepsilon^{-2}]$, all $\frac14< s'<s$, 
\[
\|\Psi_n (t) (a) \|_{H^{s'}} \leq 8^{-n} c_n D\sqrt L R^{-(s-s')} 
+ \varepsilon C(L,D,R,n)
\]
if $n$ is odd, and
\[
\|\Psi_n (t) (a) - \tilde\Psi_n(t)(a_R) \|_{H^{s'}} \leq 8^{-n} c_n D\sqrt L R^{-(s-s')}
+ 
\varepsilon^2  C(L,D,R,n)
\]
if $n$ is even.
\end{proposition}

\begin{proof} Because of Proposition \ref{prop:fstestonPsin}, we have 
\[
\| \Psi_n(t)(a) - \Psi_n(t)(a_R)\|_{H^{s'}} \leq 8^{-n} c_n \|a-a_R\|_{H^{s'}}.
\]
We have $\|a-a_R\|_{H^{s'}}\leq R^{-(s-s')}\|a\|_{H^s} \leq R^{-(s-s')}D\sqrt L$, which yields
\[
\| \Psi_n(t)(a) - \Psi_n(t)(a_R)\|_{H^{s'}}\leq 8^{-n} c_n R^{-(s-s')}D\sqrt L.
\]

If $n$ is odd, we estimate
\[
\|\Psi_n(t)(a_R)\|_{H^{s'}} = \|\psi_n(t)(a_R)\|_{H^{s'}}.
\]
We have 
\[
\hat \psi_n(t)(a_R)(\xi)=
\frac{(-i\varepsilon)^n}{(2\pi L)^{n/2}}\sum_{A\in \mathcal A_n} \sum_{\sigma \in \mathfrak S_A} \sum_{\xi_1+\hdots+ \xi_{n+1} = \xi} F_A^\sigma(t,\vec \xi) \prod_{l=1}^{n+1}\hat a_R(\xi_l).
\]
Because of the support in Fourier mode of $a_R$, we have for some $C(L,R,n)$,
\[
|\hat \psi_n(t)(a_R)(\xi)|\leq t^{(n-1)/2}C(L,R,n)
\frac{(\varepsilon)^n}{(2\pi L)^{n/2}}\sum_{A\in \mathcal A_n} \sum_{\sigma \in \mathfrak S_A} \sum_{\xi_1+\hdots+ \xi_{n+1} = \xi}  \prod_{l=1}^{n+1}|\hat a_R(\xi_l)|.
\]
The cardinal of 
$
\{(A,\sigma) \; |\; A\in \mathcal A_n,\; \sigma \in \mathfrak S_A\}
$
is $n!$. We deduce that for $t\in [-T\varepsilon^{-2},T\varepsilon^{-2}]$,
\[
|\hat \psi_n(t)(a_R)(\xi)|\leq n! C(L,R,n) T^{(n-1)/2}
\frac\varepsilon{(2\pi L)^{n/2}} \sum_{\xi_1+\hdots+ \xi_{n+1} = \xi}  \prod_{l=1}^{n+1}|\hat a_R(\xi_l)|.
\]
Summing in $\xi$, we get
\[
\|\psi_n(t)(a_R)\|_{H^{s'}} \leq n! C(L,R,n) T^{(n-1)/2}
\frac\varepsilon{(2\pi L)^{n/2}} \|a_R\|_{H^{s'}} \Big( \sum_\xi \an{\xi}^{s'}|\hat a_R(\xi)|\Big)^n.
\]
We have 
\[
\frac1{\sqrt{2\pi L}}\sum_\xi \an{\xi}^{s'}|\hat a_R(\xi)| \leq \Big(\frac1{2\pi L}\sum_{|\xi|\leq R} \an{\xi}^{2(s'-s)}\Big)^{1/2} \|a\|_{H^s} \leq C_{s-s'} R^{\max(0,s'-s+\frac12)} \|a\|_{H^s}.
\]
We deduce 
\[
\|\psi_n(t)(a_R)\|_{H^{s'}} \leq n! C_{s-s'}^n R^{\max(0,s'-s+1/2) n} C(L,R,n) T^{(n-1)/2}
\varepsilon  L^{n/2+1} D^{n+1}  .
\]
Because $T\leq T_0(L,D)$, we get indeed
\[
\|\psi_n(t)(a_R)\|_{H^{s'}} \leq \varepsilon C(L,D,R,n) .
\]

We proceed in the same way to estimate $\Psi_n(t)(a) - \tilde \Psi_n(t)(a_R)$ when $n$ is even.
\end{proof}

\subsection{Trees describing \texorpdfstring{$\tilde \psi_{\varepsilon,L,n}$}{~psiEpsLn}}\label{subsec:fstcancel}

\begin{definition}\label{def:trinarytrees} Let $\mathcal{\tilde A}_n$ be defined for $n\in \N$ as
\[
\tildeA_n = \{(A,\sigma) \; |\; A\in \mathcal A_{2n}, \sigma \in \mathfrak S_A,\; \forall j=1,\hdots,\; \sigma(2j-1) \leq_A \sigma(2j)\}.
\]
\end{definition}

\begin{remark}
If $n=0$, we have
\[
\tildeA_0 = \{(\bot,\emptyset)\}.
\]
\end{remark}

\begin{proposition}\label{prop:fstcancel}We have for $n\in 2\N$,
\[
\widehat{\tilde \psi}_n(t)(a)(\xi) = \frac{(-i\varepsilon)^n}{(2\pi L)^{n/2}}\sum_{(A,\sigma)\in \tildeA_{n/2}} \sum_{\xi_1+\hdots+ \xi_{n+1} = \xi} \tilde F_A^\sigma(t,\vec \xi) \prod_{l=1}^{n+1}\hat a(\xi_l).
\]
\end{proposition}

\begin{proof} We recall that
\[
\widehat{\tilde \psi}_{2n}(t)(a)(\xi) = \frac{(-i\varepsilon)^{2n}}{(2\pi L)^{n}}\sum_{A\in \mathcal A_{2n}} \sum_{\sigma \in \mathfrak S_A} \sum_{\xi_1+\hdots+ \xi_{2n+1} = \xi} \tilde F_A^\sigma(t,\vec \xi) \prod_{l=1}^{2n+1}\hat a(\xi_l).
\]
with
\[
\tilde F_A^\sigma(t,\vec \xi) = \prod_{m=1}^{n} \frac{\omega_{A}(\sigma(2m),\vec \xi)\omega_{A}(\sigma(2m-1),\vec \xi)}{i\Delta_{A}(\sigma(2m-1),\vec \xi)}\delta(\Delta_{A}(\sigma(2m-1),\vec \xi)+\Delta_{A}(\sigma(2m),\vec\xi)) \frac{t^{n}}{n!}.
\]
Let $j\in [|1,n|]$. If $\sigma \in \mathfrak{S}_A$ is such that $\sigma(2j-1)$ is not comparable in $A$ with $\sigma(2j)$, then $\sigma' = \sigma \circ (2j-1,2j)$ where $(2j-1,2j)$ is the transposition between $2j-1$ and $2j$ still belongs to $\mathfrak{S}_A$. Indeed, assume that $l\leq_A m$. If both $\sigma^{-1}(l)$ and $\sigma^{-1}(m)$ do not belong to $\{2j-1,2j\}$, then 
\[
(\sigma')^{-1}(l) = \sigma^{-1}(l) \leq \sigma^{-1}(m) = (\sigma')^{-1}(m).
\]
If both $\sigma^{-1}(l)$ and $\sigma^{-1}(m)$ belong to $\{2j-1,2j\}$, then because $\sigma(2j)$ and $\sigma(2j-1)$ are not comparable, we get $l=m$ and thus $(\sigma')^{-1}(l)=(\sigma')^{-1}(m)$. 

If $\sigma^{-1}(l) \in \{2j-1,2j\}$ and $\sigma^{-1}(m) \notin \{2j-1,2j\}$. We have $(\sigma')^{-1}(l) \in \{2j-1,2j\}$ and $\sigma^{-1}(m) \geq \sigma^{-1}(l)$ and thus $\sigma^{-1}(m)>2j$. We deduce
\[
(\sigma')^{-1}(m) = \sigma^{-1}(m) >2j \geq (\sigma')^{-1}(l).
\]

If $\sigma^{-1}(m) \in \{2j-1,2j\}$ and $\sigma^{-1}(l) \notin \{2j-1,2j\}$. We have $(\sigma')^{-1}(m) \in \{2j-1,2j\}$ and $\sigma^{-1}(l) \leq \sigma^{-1}(m)$ and thus $\sigma^{-1}(l)<2j-1$. We deduce
\[
(\sigma')^{-1}(l) = \sigma^{-1}(l) <2j-1 \leq (\sigma')^{-1}(m).
\]

What is more,
\[
\tilde F_A^\sigma(t,\vec \xi) = -\tilde F_A^{\sigma'}(t,\vec \xi).
\]
Set 
\[
\tildeA_n^{(1)} = \{(A,\sigma) \; |\; A \in \mathcal A_{2n} ,\, \sigma \in \mathfrak S_A , \, \sigma(1)\nleq_A \sigma(2)\}
\]
and by induction on $j\in [|1,n|]$
\[
\tildeA_n^{(j)} = \{(A,\sigma) \; |\; A \in \mathcal A_{2n} ,\, \sigma \in \mathfrak S_A , \, \sigma(2j-1)\nleq_A \sigma(2j)\} \cap (\tildeA_n^{(j-1)})^c
\]
such that
\[
\tildeA_n  = (\tildeA_n^{(n)})^c.
\]

We deduce 
\begin{align*}
\widehat{\tilde \psi}_{2n}(t)(a)(\xi) = \frac{(-i\varepsilon)^{2n}}{(2\pi L)^{n}}\sum_{j=1}^n \sum_{(A,\sigma)\in \tildeA_n^{(j)}}  \sum_{\xi_1+\hdots+ \xi_{2n+1} = \xi} \tilde F_A^\sigma(t,\vec \xi) \prod_{l=1}^{2n+1}\hat a(\xi_l) \\
+ \frac{(-i\varepsilon)^{2n}}{(2\pi L)^{n}}\sum_{(A,\sigma)\in \tildeA_n}  \sum_{\xi_1+\hdots+ \xi_{2n+1} = \xi} \tilde F_A^\sigma(t,\vec \xi) \prod_{l=1}^{2n+1}\hat a(\xi_l).
\end{align*}
We have 
\begin{align*}
\sum_{(A,\sigma)\in \tildeA_n^{(j)}}   \tilde F_A^\sigma(t,\vec \xi) = \sum_{(A,\sigma)\in \tildeA_n^{(j)}}  \tilde F_A^{\sigma\circ (2j-1,2j)}(t,\vec \xi)\\
= -\sum_{(A,\sigma)\in \tildeA_n^{(j)}}   \tilde F_A^\sigma(t,\vec \xi)\\
=0
\end{align*}
We deduce the result.
\end{proof}

\begin{definition} Take $n\in 2\N$. Set now for $(A,\sigma)\in \tildeA_{n/2}$, and $\xi \in \frac1{L}\Z^*$, $a\in H^s$,
\begin{multline*}
(2\pi L)^{n/2}\hat G_A^\sigma(a)(\xi) = \\
\sum_{\xi_1+\hdots+ \xi_{n+1} = \xi} \prod_{m=1}^{n/2} \frac{\omega_{A}(\sigma(2m),\vec \xi)\omega_{A}(\sigma(2m-1),\vec \xi)}{i\Delta_{A}(\sigma(2m-1),\vec \xi)}\delta(\Delta_{A}(\sigma(2m-1),\vec \xi)+\Delta_{A}(\sigma(2m),\vec\xi)) \prod_{l=1}^{n+1}\hat a(\xi_l).
\end{multline*}
By definition, we have 
\[
\tilde \psi_n(t)(a) = (-i\varepsilon)^n  \frac{t^{n/2}}{(n/2)!} \sum_{(A,\sigma) \in \tildeA_{n/2}} G_A^\sigma (a).
\]
\end{definition}

\begin{remark} We remark that
\[
\tilde F_A(t,\vec \xi) = \prod_{m=1}^{n/2} \frac{\omega_{A}(\sigma(2m),\vec \xi)\omega_{A}(\sigma(2m-1),\vec \xi)}{i\Delta_{A}(\sigma(2m-1),\vec \xi)}\delta(\Delta_{A}(\sigma(2m-1),\vec \xi)+\Delta_{A}(\sigma(2m),\vec\xi)) \frac{t^{n/2}}{(n/2)!}.
\]
\end{remark}

\begin{proposition}\label{prop:GTsigma} Let $(A,\sigma) \in \tildeA_n$. Assume $n\geq 1$. We have either $\sigma(2n-1)= (2,0)$ or $\sigma(2n-1) = (1,0)$. If $\sigma(2n-1) = (2,0)$, then set
\[
A = (A_1,(A_2,A_3))
\]
and $(\sigma_1,\sigma_2,\sigma_3)$ the orders on the nodes of $(A_1,A_2,A_3)$ induced by $\sigma$, that is for $m \in N(A_k)$,
\begin{align*}
(1,\sigma_1(j)) = \sigma(\varphi_1(j))\\
(2,(1,\sigma_2(j))) = \sigma(\varphi_2(j))\\
(2,(2,\sigma_3(j))) = \sigma(\varphi_3(j))
\end{align*}
where $\varphi_k$ is the only increasing bijection from $[|1,2n_k|]$ to $\sigma^{-1}(\{1\}\times N(A_1)) \subseteq [|1,2n-2|]$ if $k=1$, $\sigma^{-1}(\{2\}\times \{1\} \times N(A_2))$ if $k=2$, $\sigma^{-1}(\{2\}\times \{2\} \times N(A_3))$ if $k=3$.

If $\sigma(2n-1) = (1,0)$, we set
\[
A = ((A_2,A_3),A_1)
\]
and $(\sigma_1,\sigma_2,\sigma_3)$ the orders on the nodes of $(A_1,A_2,A_3)$ induced by $\sigma$. In both cases, we have $(A_k,\sigma_k) \in \tildeA_{n_k}$ with $n_1+n_2+n_3 = n-1$.

We have 
\[
G_A^\sigma(a) = i K(G_{A_1}^{\sigma_1}(a), G_{A_2}^{\sigma_2}(a),G_{A_3}^{\sigma_3}(a))
\]
where $K$ is defined in Fourier mode by
\[
K(u,v,w) = \frac1{2\pi L}\sum_{\xi_1 + \xi_2+\xi_3 = \xi}\frac{\omega(\xi)\omega(\xi_2+\xi_3)}{\Delta^\xi_{\xi_1,\xi_2+\xi_3}}\delta(\Delta^{\xi}_{\xi_1,\xi_2+\xi_3} + \Delta^{\xi_2+\xi_3}_{\xi_2,\xi_3})\hat u(\xi_1) \hat v(\xi_2)\hat w(\xi_3).
\]
\end{proposition}

\begin{lemma}\label{lem:parity} Assume that $(A,\sigma) \in \tildeA_n$. For all $j \in [|1,n|]$, we have that there exists no $m\in [|1,2n|]$ such that
\[
\sigma(2j-1) <_A \sigma(m) <_A \sigma(2j)
\]
where $<_A$ is the strict order associated to $\leq_A$.

With the above notations, for $k=1,2,3$, $\varphi_k$ respects parity. 
\end{lemma}

\begin{proof}  For the first part of the lemma, we argue by contradiction. If there exists $m\in [|1,2n|]$ such that
\[
\sigma(2j-1)<_A \sigma(m) <_A\sigma(2j)
\]
them $m\neq 2j-1,2j$. We deduce that $m>2j$ or $m<2j-1$. In both cases, this violates the fact that $\sigma \in \mathfrak S_A$ and thus if $l,l'\in N(A)$ are such that $l\leq_A l'$, then $\sigma^{-1}(l)\leq \sigma^{-1}(l')$.

For the second part of the lemma, we treat the case $\sigma(2n-1) = (2,0)$, the case $\sigma(2n-1) = (1,0)$ being similar.

On the one hand, because $\varphi_1$ is increasing, we have by induction
\begin{align*}
\varphi_1(2n_1) = \max \{l\in [|1,2n-2|] \, |\, \sigma(l) \in \{1\}\times N(A_1)\},\\
\forall j<2n_1,\quad \varphi_1(j) = \max \{l\in [|1,\varphi_1(j+1)-1|]\,|\, \sigma(l) \in \{1\}\times N(A_1)\}.
\end{align*}

For the same reasons, we have 
\begin{align*}
\varphi_2(2n_2) = \max \{l\in [|1,2n-2|] \, |\, \sigma(l) \in \{2\}\times \{1\}  \times N(A_2)\},\\
\forall j<2n_2,\quad \varphi_2(j) = \max \{l\in [|1,\varphi_2(j+1)-1|]\,|\, \sigma(l) \in \{2\}\times\{1\}\times N(A_2)\},\\
\varphi_3(2n_3) = \max \{l\in [|1,2n-2|] \, |\, \sigma(l) \in \{2\}\times \{2\}  \times N(A_3)\},\\
\forall j<2n_3,\quad \varphi_3(j) = \max \{l\in [|1,\varphi_3(j+1)-1|]\,|\, \sigma(l) \in \{2\}\times\{2\}\times N(A_3)\}.
\end{align*}

One the other hand, we have that for all $j\in [|1,n-1|]$, since $\sigma(2j-1)\leq_A \sigma(2j)$, the nodes $\sigma(2j-1)$ and $\sigma(2j)$ belongs to the same subtree of $A$ (either $A_1$, $A_2$ or $A_3$).

Therefore, we prove by (decreasing) induction that for all $j\in [|1,n-1|]$, there exists $k\in \{1,2,3\}$ and $j_k \in [|1,n_k|]$ such that
\[
2j=\varphi_k(2j_k), \quad 2j-1 = \varphi_k(2j_k-1).
\]

For $j=n-1$, we have either $\sigma(2j) \in \{1\}\times N(A_1)$ or $\sigma(2j) \in \{2\}\times\{1\}\times N(A_2)$ or $\sigma(2j) \in  \{2\}\times\{2\}\times N(A_3)$. In the first (resp. second, resp. third) case this ensures that $\varphi_1(2n_1) = 2j$, (resp. $\varphi(2n_2) = 2j$, resp. $\varphi_3(2n_3) = 2j$). From the remark above we have that $\sigma(2j-1) \in \{1\}\times N(A_1)$ (resp. $\sigma(2j-1) \in \{2\}\times \{1\}\times N(A_2)$, resp.  $\sigma(2j-1) \in \{2\}\times \{2\}\times N(A_3)$). We deduce that $\varphi_1(2n_1-1) = 2j-1$ (resp. $\varphi_2(2n_2-1) = 2j-1$, resp. $\varphi_3(2n_3-1) = 2j-1$). 

Let $j\in [|1,n-1|]$ and assume that the induction hypothesis is true for all $l\in [|j+1, n-1|]$. Assume that $\sigma(2j) \in \{1\}\times N(A_1)$, then there exists $j'_1$ such that $\varphi_1(j'_1) = 2j$. Let $m= \varphi_1(j'_1+1)$. We have that $m>2j$. If $m$ is even, from the induction hypothesis, we have that $\varphi_1(j'_1) = m-1$ which contradicts the fact that $\varphi(j'_1) = 2j$. We deduce that $m$ is odd. But since $m>2j$, by the induction hypothesis, we get that $j'_1+1$ is odd and thus $j'_1$ is even. We write $j'_1 = 2j_1$. From the remark above, we get that $\sigma(2j-1) \in \{1\} \times N(A_1)$ and thus $\varphi_1(2j_1-1)  = 2j-1$.

With similar arguments, we get that if $\sigma(2j) \in \{2\}\times \{1\}\times N(A_2)$ (resp. $\sigma(2j) \in \{2\}\times \{2\}\times N(A_3)$) then there exists $j_2 \in [|1,n_2|]$ (resp. $j_3 \in [|1,n_3|]$) such that
\[
\varphi_2(2j_2) = 2j,\quad \varphi_2(2j_2-1) = 2j-1
\]
(resp.
\[
\varphi_3(2j_3) = 2j , \quad \varphi_3(2j_3-1) = 2j-1 .)
\]

Therefore, $\varphi_1,\varphi_2,\varphi_3$ respect parity.
\end{proof}

\begin{proof}[Proof of Proposition \ref{prop:GTsigma}] To prove that the discussion on cases is exhaustive, we use the first part of the above lemma. Assume by contradiction that $\sigma(2n-1) \neq (1,0), (2,0)$ then we have either $\sigma(2n-1) <_A (1,0)$ or $\sigma(2n-1)<_A (2,0)$. We get that
\[
\sigma(2n-1) <_A (\iota, 0) <_A \sigma(2n)
\]
with $\iota = 1$ or $2$. Hence there exists $m=\sigma^{-1}(\iota,0)$ such that
\[
\sigma(2n-1) <_A \sigma(m) <_A \sigma(2n)
\]
which is absurd.

We set $K_A^\sigma$ the function defined by induction on $A$ as 
\[
K_\bot^\emptyset = a
\]
and 
\[
K_A^\sigma = iK(K_{A_1}^{\sigma_1}, K_{A_2}^{\sigma_2},K_{A_3}^{\sigma_3})
\]
And we prove by induction that $K_A^\sigma = G_A^\sigma (a)$. For $A = \bot$, this is by definition.

Let $(A,\sigma) \in \tildeA_n$. We treat the case $\sigma(2n-1) = (2,0)$ and we set
\[
A = (A_1,(A_2,A_3))
\]
and $(\sigma_1,\sigma_2,\sigma_3)$ the orders on the nodes of $(A_1,A_2,A_3)$ induced by $\sigma$. Let $k\in \{1,2,3\}$ and $j_k \in [|1,n_k|]$. Because $\varphi_k$ respects parity, we have 
\[
\sigma(\varphi_k(2j_k-1))\leq_A \sigma(\varphi_k(2j_k-1)+1) = \sigma(\varphi_k(2j_k)).
\]
We deduce that 
\[
\sigma_k(2j_k-1) \leq_{A_k} \sigma_k(2j_k).
\]
Therefore, $(A_k,\sigma_k) \in \tildeA_k$.

We pass in Fourier mode and get
\[
\hat K_A^\sigma(\xi) =  \frac1{2\pi L}\sum_{\eta_1 + \eta_2+\eta_3 = \xi}\frac{\omega(\xi)\omega(\eta_2+\eta_3)}{-i\Delta^\xi_{\eta_1,\eta_2+\eta_3}}\delta(\Delta^{\xi}_{\eta_1,\eta_2+\eta_3} + \Delta^{\eta_2+\eta_3}_{\eta_2,\eta_3})\hat K_{A_1}^{\sigma_1}(\eta_1) \hat K_{A_2}^{\sigma_2}(\eta_2)\hat K_{A_3}^{\sigma_3}(\eta_3).
\]
By the induction hypothesis, we have for $k=1,2,3$, 
\small
\begin{multline*}
(2\pi L)^{n_k}\hat K_{A_k}^{\sigma_k}(\eta_k) = (2\pi L)^{n_k}\hat G_{A_k}^{\sigma_k}(a)(\eta_k) = \\
 \sum_{\sum \eta_{k,j} = \eta_k} \prod_{m=1}^{n/2} \frac{\omega_{A_k}(\sigma_k(2m),\vec \eta_k)\omega_{A_k}(\sigma_k(2m-1),\vec \eta_k)}{i\Delta_{A_k}(\sigma_k(2m-1),\vec \eta_k)}\delta(\Delta_{A_k}(\sigma_k(2m-1),\vec \eta_k)+\Delta_{A_k}(\sigma_k(2m),\vec\eta_k)) \prod_{l=1}^{2n_k+1}\hat a(\eta_{k,l})
\end{multline*}
\normalsize
where $\vec \eta_k = (\eta_{k,1},\hdots,\eta_{k,2n_k+1}) $. By definition of $\sigma_k$, we have 
\[
\omega_{A_k}(\sigma_k(m),\vec \eta_k) = \omega_{A}(\sigma(\varphi_k(m)),\vec \xi)
\]
where $\vec \xi = (\eta_{1,1},\hdots,\eta_{1,2n_1+1},\eta_{2,1},\hdots,\eta_{2,2n_2+1},\eta_{3,1},\hdots,\eta_{3,2n_3+1}) \in (\frac1{L}\Z^*)^{2(n_1+n_2+n_3) + 3} = (\frac1{L}\Z^*)^{2n+1}$. We also have that
\[
\omega(\xi) = \omega_{A}(\sigma(2n),\vec \xi), \quad \omega(\eta_1+\eta_2) = \omega_{A}(\sigma(2n-1),\vec \xi).
\]

Using the lemma, we have that for all $m\in [|1,n_k|]$, we have 
\[
\varphi_k(2m-1) = \varphi_k(2m) - 1 \in 2[|1,n-1|] - 1.
\]
We deduce that
\[
\Delta_{A_k}(\sigma_k(2m-1),\vec \eta_k) = \Delta_{A}(\sigma(\varphi_k(2m)-1),\vec \xi), \quad \Delta_{A_k}(\sigma_k(2m),\vec \eta_k) = \Delta_A(\sigma(\varphi_k(2m)),\vec \xi)
\]
with $\varphi_k(2m) \in 2[|1,(n-1)|]$. We also have that
\[
\Delta^{\xi}_{\eta_1,\eta_2+\eta_3} = \Delta_{A}(\sigma(2n),\vec \xi), \quad \Delta^{\eta_2+\eta_3}_{\eta_2,\eta_3}= \Delta_A(\sigma(2n-1),\vec \xi)
\]

What is more, we have that 
\[
\varphi_1(2[|1,n_1|]) \sqcup \varphi_2(2[|1,n_2|]) \sqcup \varphi_3(2[|1,n_3|]) = 2[|1,n-1|].
\]

And finally, we remark that renaming the variables as $\vec \xi = (\xi_1,\hdots,\xi_{2n+1})$ we get that the sets of 
\[
(\eta_{1,1},\hdots,\eta_{1,2n_1+1},\eta_{2,1},\hdots,\eta_{2,2n_2+1},\eta_{3,1},\hdots,\eta_{3,2n_3+1}) \in \frac1{L}(\Z^*)^{2n+1}
\]
such that there exist $(\eta_1,\eta_2,\eta_3)\in \frac1{L}\Z$ such that $ \eta_1+ \eta_2 + \eta_3 = \xi$ and such that for all $k=1,2,3$, we have $\eta_{k,1}+\hdots,\eta_{k,2n_k+1}= \eta_k $ is equal to
\[
\{ (\xi_1,\hdots,\xi_{2n+1})\,|\, \xi_1 + \hdots \xi_{2n+1} = \xi\}
\]
which allows to conclude noticing that if for some $k=1,2,3$, $\eta_k = 0$, then either $\omega_A ( (2,0))$ or $\omega_A(1,1,0)$ or $\omega_A(1,2,0)$ is null.

The case $\sigma(2n-1) = (1,0)$ is similar.
\end{proof}

\subsection{Exact computation of \texorpdfstring{$\tilde \Psi$}{~ Psi}}\label{subsec:combs}


\begin{proposition}\label{prop:rewriteK} Let $(A,\sigma) \in \tilde A_n$ and $a\in H^s$. We have that 
$
G_A^\sigma(a)
$
is real valued and for all $\xi$,
\[
\widehat{G_A^\sigma (a)}(\xi) = i^{n} U_A^{\sigma}(a,\xi) \hat a (\xi)
\]
where $U_A^\sigma (a,\xi)$ is real, from which we deduce that $U_A^\sigma (a,-\xi) = (-1)^{n} U_A^\sigma (a,\xi)$. From this we get that $U_A^\sigma (a,\xi)$ is obtained by induction (with the notations of Proposition \ref{prop:GTsigma}) by
\[
U_A^\sigma (a,\xi) = -\frac1{2\pi L}\sum_{\eta} F(\eta,\xi)  U_{A_1}^{\sigma_1}(a,\eta)\Big( U_{A_2}^{\sigma_2}(a,\eta) U_{A_3}^{\sigma_3}(a,\xi) + U_{A_3}^{\sigma_3}(a,\eta) U_{A_2}^{\sigma_2}(a,\xi)\Big) |a(\eta)|^2
\]
where
\[
F(\eta,\xi) = \frac{\omega(\xi)\omega (\xi-\eta)}{\Delta^{\xi}_{\eta,\xi-\eta}} .
\]
\end{proposition}

\begin{lemma} Assume $(\xi_1,\xi_2, \xi_3,\xi)\in \R^*$ are such that $\xi_1+\xi_2+\xi_3 = \xi$, $\xi_2+\xi_3 \neq 0$ and 
\[
\omega(\xi) - \omega(\xi_1) - \omega(\xi_2) -\omega(\xi_3) = \Delta^{\xi}_{\xi_1,\xi_2+\xi_3} + \Delta^{\xi_2+\xi_3}_{\xi_2,\xi_3} = 0.
\]
Then, we have either
\begin{itemize}
\item $\xi_2 = \xi$, $\xi_3=-\xi_1$,
\item $\xi_3 = \xi$, $\xi_3=-\xi_1$,
\item \begin{align*}\xi_2 = \frac{\xi-\xi_1}{2} \pm \frac12 \sqrt{(\xi-\xi_1)^2 -8 + 4 \frac{(1+\xi_1^2)(1+\xi^2)}{1-\xi\xi_1} -4\xi\xi_1}, \\
\xi_3 = \frac{\xi-\xi_1}{2} \mp \frac12 \sqrt{(\xi-\xi_1)^2 -8 + 4 \frac{(1+\xi_1^2)(1+\xi^2)}{1-\xi\xi_1} -4\xi\xi_1}
\end{align*}
assuming that 
\[(\xi-\xi_1)^2 -8 + 4 \frac{(1+\xi_1^2)(1+\xi^2)}{1-\xi\xi_1} -4\xi\xi_1\geq 0,
\]
and $\xi\xi_1 \neq 1$.
\end{itemize} 
\end{lemma}

\begin{proof} For the sake of the proof, we write $\eta = \xi-\xi_1 \neq 0$. We have $\omega(-x) = -\omega(x)$ therefore, we know that the first two items are solutions to the equation
\[
\omega(\xi) - \omega(\xi_1) - \omega(\xi_2) -\omega(\xi_3) = 0.
\]

The only solutions in $\R^*$ of $\omega(x) = \omega(y)$ are $x=y$ or $x=\frac1{y}$. 

If $\omega(\xi) = \omega(\xi_1)$, we have either $\xi = \xi_1$ or $\xi = \frac1{\xi_1}$. What is more, $\omega(\xi_2)+\omega(\xi_3) = 0$, which yields $\omega(\xi_2) = \omega(-\xi_3)$, that is either $\xi_2 = -\xi_3$ or $\xi_2 = \frac1{\xi_3}$. Since $\xi_2 + \xi_3 \neq 0$, we have $\xi_2 + \frac1{\xi_3} = 0$. This implies that
\[
\xi_2 - \frac1{\xi_2} = \eta
\]
which admits two solutions $\xi_2 = \xi$ or $\xi_2 = -\xi_1$. This yields one of the two first items.

We assume now that $\delta =  \omega(\xi) - \omega(\xi_1) \neq 0$. We look for the zeros of the function
\[
\begin{array}{c}
\R^* \rightarrow \R\\
X \mapsto \omega(X) + \omega(\eta - X) -\delta
\end{array} .
\]
Multiplying the function by $(1+X^2)(1+(\eta - X)^2) = 1 + \eta^2 -2\eta X +X^2(2+\eta^2) -2\eta X^3 +X^4$, we get the polynomial
\[
X(1+(\eta-X)^2) + (\eta-X)(1+X^2) -\delta (1 + \eta^2 -2\eta X +X^2(2+\eta^2) -2\eta X^3 +X^4).
\]
We divide by $-\delta$ and get the polynomial 
\[
P = X^4-2\eta X^3 +X^2(2+\eta^2-\delta^{-1}\eta) +X(\delta^{-1}\eta^2-2\eta ) +(1+\eta^2) -\delta^{-1} \eta.
\]
Because $P$ has $\xi$ and $-\xi_1$ as roots we get that
\[
P = (X-\xi) (X+\xi_1) (X^2 + aX+b) = (X^2-\eta X -\xi \xi_1)(X^2+aX+b)
\]
with $a = -\eta$ and $b -a\eta -\xi\xi_1=  2+\eta^2-\delta^{-1}\eta$ that is $b= 2 - \frac{\eta}{\delta} +\xi\xi_1$. 

The polynomial $Q = X^2 - \eta X +b$ admits roots in $\R$ if
\[
\eta^2 - 4b \geq 0
\]
that is
\[
\eta^2 -8 + 4 \frac{\eta}{\delta} -4\xi\xi_1 \geq 0.
\]
We notice that
\[
 \frac{\eta}{\delta} = \frac{(1+\xi_1^2)(1+\xi^2)}{1-\xi\xi_1}.
\]
Hence the condition writes
\[
\eta^2 -8 + 4 \frac{(1+\xi_1^2)(1+\xi^2)}{1-\xi\xi_1} -4\xi\xi_1 \geq 0.
\]
and in this case the roots write
\[
X = \frac{\eta}{2} \pm \frac12 \sqrt{\eta^2 -8 + 4 \frac{(1+\xi_1^2)(1+\xi^2)}{1-\xi\xi_1} -4\xi\xi_1}.
\]

\end{proof}

\begin{lemma}\label{lem:notalgebraic} If $L^2$ is not algebraic of degree less than $2$ then the system
\begin{equation*}
\left \lbrace{\begin{array}{cc}
\xi_1+\xi_2+\xi_3 = \xi\\
\xi_2+\xi_3 \neq 0 \\
\omega(\xi) - \omega(\xi_1) - \omega(\xi_2) -\omega(\xi_3) = \Delta^{\xi}_{\xi_1,\xi_2+\xi_3} + \Delta^{\xi_2+\xi_3}_{\xi_2,\xi_3} = 0
\end{array}} \right.
\end{equation*}
admits as solutions in $\frac1{L} \Z^*$,
\begin{itemize}
\item $\xi_2 = \xi$, $\xi_3=-\xi_1$,
\item $\xi_3 = \xi$, $\xi_3=-\xi_1$.
\end{itemize}
\end{lemma}

\begin{proof} We know that the only other possible solutions are
\begin{align*}
\xi_2 = \frac{\xi-\xi_1}{2} \pm \frac12 \sqrt{(\xi-\xi_1)^2 -8 + 4 \frac{(1+\xi_1^2)(1+\xi^2)}{1-\xi\xi_1} -4\xi\xi_1}, \\
 \xi_3 = \frac{\xi-\xi_1}{2} \mp \frac12 \sqrt{(\xi-\xi_1)^2 -8 + 4 \frac{(1+\xi_1^2)(1+\xi^2)}{1-\xi\xi_1} -4\xi\xi_1}
\end{align*}
assuming that $(\xi-\xi_1)^2 -8 + 4 \frac{(1+\xi_1^2)(1+\xi^2)}{1-\xi\xi_1} -4\xi\xi_1\geq 0$ and $\xi\xi_1 \neq 1$,

If there are such solutions then
\[
\Big(\xi_2 - \frac{\xi-\xi_1}{2}\Big)^2 = \frac14 \Big((\xi-\xi_1)^2 -8 + 4 \frac{(1+\xi_1^2)(1+\xi^2)}{1-\xi\xi_1} -4\xi\xi_1\Big)
\]
which is equivalent to
\[
(1-\xi\xi_1)\xi_2(\xi_1 + \xi_2 - \xi) = -2 + \xi\xi_1 + (1+\xi_1^2) (1+\xi^2) -\xi^2\xi_1^2.
\]
We write
\[
\xi  = \frac{k}{L},\quad \xi_1 = \frac{k_1}{L}, \quad \xi_2 = \frac{k_2}{L}
\]
with $k,k_1,k_2 \in \Z^*$. Multiplying the above identity by $L^4$, we get
\[
(L^2 - kk_1) (k_2 +k_1 - k) = -2L^4 + kk_1 L^2 +(L^2+k_1^2)(L^2+k^2) - k^2k_1^2
\]
which is equivalent to
\[
L^4 + L^2 (k_2(k_1+k_2 - k) - kk_1 -k_1^2 -k^2) - kk_1k_2 (k_1+k_2 - k) = 0
\]
which implies that $L^2$ is algebraic of degree less than $2$. 

\end{proof}

\begin{proof}[Proof of Proposition \ref{prop:rewriteK}] 
By definition, we have 
\[
\hat K(u,v,w) (\xi) = \frac1{2\pi L}\sum_{\xi_1 + \xi_2+\xi_3 = \xi}\frac{\omega(\xi)\omega(\xi_2+\xi_3)}{\Delta^\xi_{\xi_1,\xi_2+\xi_3}}\delta(\Delta^{\xi}_{\xi_1,\xi_2+\xi_3} + \Delta^{\xi_2+\xi_3}_{\xi_2,\xi_3})\hat u(\xi_1) \hat v(\xi_2)\hat w(\xi_3).
\]
Assume that $u,v,w$ are real-valued, we have 
\[
\hat K(u,v,w) (-\xi) = \frac1{2\pi L}\sum_{\xi_1 + \xi_2+\xi_3 = -\xi}\frac{\omega(-\xi)\omega(\xi_2+\xi_3)}{\Delta^{-\xi}_{\xi_1,\xi_2+\xi_3}}\delta(\Delta^{-\xi}_{\xi_1,\xi_2+\xi_3} + \Delta^{\xi_2+\xi_3}_{\xi_2,\xi_3})\hat u(\xi_1) \hat v(\xi_2)\hat w(\xi_3).
\]
We do the change of variables $\eta_j = -\xi_j$ and get
\[
\hat K(u,v,w) (-\xi) = \frac1{2\pi L}\sum_{\eta_1 + \eta_2+\eta_3 = \xi}\frac{\omega(-\xi)\omega(-\eta_2-\eta_3)}{\Delta^{-\xi}_{-\eta_1,-\eta_2-\eta_3}}\delta(\Delta^{-\xi}_{-\eta_1,-\eta_2-\eta_3} + \Delta^{-\eta_2-\eta_3}_{-\eta_2,-\eta_3})\hat u(-\eta_1) \hat v(-\eta_2)\hat w(-\eta_3).
\]
We use that $\omega$ is odd, which implies $\Delta^{-x}_{-y,-z} =\omega(-x) - \omega(-y) - \omega(-z) = -\Delta^{x}_{y,z}$ and get 
\[
\hat K(u,v,w) (-\xi) = \frac1{2\pi L}\sum_{\eta_1 + \eta_2+\eta_3 = \xi}\frac{\omega(\xi)\omega(\eta_2+\eta_3)}{-\Delta^{\xi}_{\eta_1,\eta_2+\eta_3}}\delta(-\Delta^{\xi}_{\eta_1,\eta_2+\eta_3} - \Delta^{\eta_2+\eta_3}_{\eta_2,\eta_3})\hat u(-\eta_1) \hat v(-\eta_2)\hat w(-\eta_3).
\] 
We now use that $u,v,w$ are real-valued and get
\[
\overline{\hat K(u,v,w) (-\xi)} = -\frac1{2\pi L}\sum_{\eta_1 + \eta_2+\eta_3 = \xi}\frac{\omega(\xi)\omega(\eta_2+\eta_3)}{\Delta^{\xi}_{\eta_1,\eta_2+\eta_3}}\delta(\Delta^{\xi}_{\eta_1,\eta_2+\eta_3} + \Delta^{\eta_2+\eta_3}_{\eta_2,\eta_3})\hat u(\eta_1) \hat v(\eta_2)\hat w(\eta_3)
\] 
and thus
\[
\overline{\hat K(u,v,w) (-\xi)} = -\hat K(u,v,w)(\xi)
\]
which implies that $K(u,v,w)$ is imaginary and thus by induction and using Proposition \ref{prop:GTsigma}, $G_A^\sigma(a)$ is real-valued.

Assume that $u,v,w$ write in Fourier mode as 
\[
\hat u(\xi) = U_u(\xi) a (\xi),\quad \hat v(\xi) = U_v(\xi) a (\xi),\quad \hat w(\xi) = U_w(\xi) a (\xi)
\]
with $U_u(\xi),U_v(\xi),U_w(\xi)$ real. 

Again, we have 
\[
\hat K(u,v,w) (\xi) = \frac1{2\pi L}\sum_{\xi_1 + \xi_2+\xi_3 = \xi}\frac{\omega(\xi)\omega(\xi_2+\xi_3)}{\Delta^\xi_{\xi_1,\xi_2+\xi_3}}\delta(\Delta^{\xi}_{\xi_1,\xi_2+\xi_3} + \Delta^{\xi_2+\xi_3}_{\xi_2,\xi_3})\hat u(\xi_1) \hat v(\xi_2)\hat w(\xi_3).
\]
We apply Lemma \ref{lem:notalgebraic} and get
\[
\hat K(u,v,w) (\xi) = \frac1{2\pi L}\sum_{\eta}\frac{\omega(\xi)\omega(\xi-\eta)}{\Delta^\xi_{\eta,\xi-\eta}}\hat u(\eta)( \hat v(-\eta)\hat w(\xi) +  \hat w(-\eta)\hat v(\xi)) .
\]
We use the hypothesis on $u,v,w$ and get
\[
\hat K(u,v,w) (\xi) = \frac1{2\pi L}\sum_{\eta}\frac{\omega(\xi)\omega(\xi-\eta)}{\Delta^\xi_{\eta,\xi-\eta}}|a(\eta)|^2 U_u(\eta)( U_v(-\eta)U_w(\xi) +  U_w(-\eta)U_v(\xi)) \hat a(\xi) .
\]
Therefore, we have 
\[
\hat K(u,v,w) (\xi) = U_{K(u,v,w)}(\xi) a(\xi)
\]
with
\[
U_{K(u,v,w)}(\xi) = \frac1{2\pi L}\sum_{\eta}\frac{\omega(\xi)\omega(\xi-\eta)}{\Delta^\xi_{\eta,\xi-\eta}}|a(\eta)|^2 U_u(\eta)( U_v(-\eta)U_w(\xi) +  U_w(-\eta)U_v(\xi)).
\]
Because if $A=(A_1,(A_2,A_3))$ or $A=((A_2,A_3),A_1)$ we have $n = n_1+n_2+n_3 +1$ (where $(A_j,\sigma_j) \in \tildeA_{n_j}$), the powers on the imaginary $i$ match and we get by induction that
\[
\widehat{G_A^{\sigma}(a)}(\xi) = i^{n} U_A^\sigma(a,\xi) a(\xi)
\]
with $U_A^\sigma(a,\xi)$ real and satisfying the induction definition
\[
U_A^\sigma(a,\xi) = \frac1{2\pi L}\sum_{\eta}\frac{\omega(\xi)\omega(\xi-\eta)}{\Delta^\xi_{\eta,\xi-\eta}}|a(\eta)|^2 U_{A_1}^{\sigma_1}(\eta)( U_{A_2}^{\sigma_2}(-\eta)U_{A_3}^{\sigma_3}(\xi) +  U_{A_3}^{\sigma_3}(-\eta)U_{A_2}^{\sigma_2}(\xi))
\]
with initial condition $U_{\bot}^{\emptyset} (a,\xi) = 1$.

Because $G_A^\sigma(a)$ and $a$ are real-valued, we get
\[
U_A^\sigma(a,-\xi) = (-1)^{n} U_A^\sigma(a,\xi) .
\]

\end{proof}

\begin{proposition}\label{prop:rewriteK2} Let $n\in \N$ and $\xi \in \frac1{L}\Z^*$. We have that
\[
\sum_{(A,\sigma)\in \tildeA_n} U_A^\sigma(a,\xi) = U(a,\xi)^{n}
\]
with
\[
U(a,\xi) = \frac2{\pi L} \sum_\eta |a(\eta)|^2 F(\xi,\eta) = \frac2{\pi L} \sum_\eta |a(\eta)|^2 \tilde F (\xi,\eta)
\]
with
\[
\tilde F(\xi,\eta) = \frac{\xi (1+\eta^2)}{(3 + \frac{\eta^2 + \xi^2 + (\xi-\eta)^2}{2}) ( 3 + \frac{\eta^2 + \xi^2 + (\xi+\eta)^2}{2})}.
\]
What is more,
\[
\sum_\xi \an{\xi}|U(a,\xi)|\lesssim \frac{\|a\|_{L^2}^2}{L}.
\]
\end{proposition}

\begin{remark}\label{rm:combs} The standard reason for this is that the contributions of most trees cancel each other out. The only ones that do not cancel out are the combs.

The condition $\Delta_A(\sigma(2j-1),\vec \xi) + \Delta_A(\sigma(2j),\vec \xi) = 0$ imposes a coupling on the leaves. In particular, for the trees in $\tildeA_1$, the coupling takes the following form.

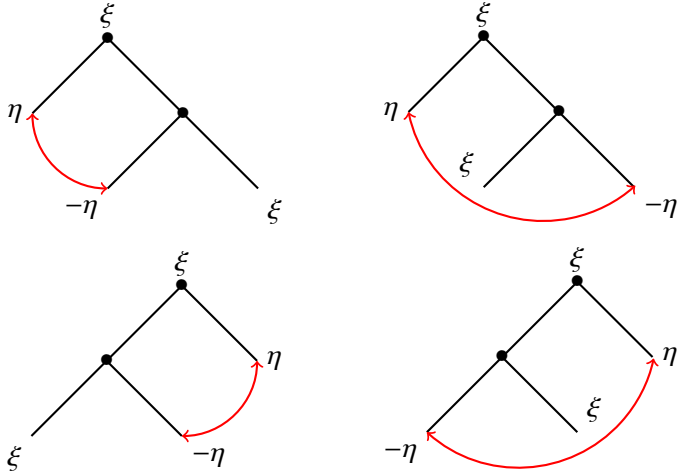
\begin{minipage}{4cm}
\begin{tikzpicture}
\draw (0,0) node {$\bullet$};
\draw (1,-1) node {$\bullet$};
\draw[thick] (0,0)--(-1,-1);
\draw[thick] (0,0)--(1,-1);
\draw[thick] (1,-1)--(0,-2);
\draw[thick] (1,-1)--(2,-2);
\draw (0,0) node [above] {$\xi$};
\draw (-1,-1) node [left] {$\eta$};
\draw (0,-2) node [below left] {$-\eta$};
\draw (2,-2) node [below right] {$\xi$};
\draw [thick, red,<->] (-1,-1) arc (0:90:-1);
\end{tikzpicture}
\end{minipage}\hspace{1cm}\begin{minipage}{4,5cm}
\begin{tikzpicture}
\draw (0,0) node {$\bullet$};
\draw (1,-1) node {$\bullet$};
\draw[thick] (0,0)--(-1,-1);
\draw[thick] (0,0)--(1,-1);
\draw[thick] (1,-1)--(0,-2);
\draw[thick] (1,-1)--(2,-2);
\draw (0,0) node [above] {$\xi$};
\draw (-1,-1) node [left] {$\eta$};
\draw (0,-2) node [above left] {$\xi$} ;
\draw (2,-2) node [below right] {$-\eta$};
\draw [thick, red,<->] (-1,-1) arc (12:132:-1.83);
\end{tikzpicture}
\end{minipage}

\begin{minipage}{4cm}
\begin{tikzpicture}
\draw (0,0) node {$\bullet$};
\draw (-1,-1) node {$\bullet$};
\draw[thick] (0,0)--(-1,-1);
\draw[thick] (0,0)--(1,-1);
\draw[thick] (-1,-1)--(0,-2);
\draw[thick] (-1,-1)--(-2,-2);
\draw (0,0) node [above] {$\xi$};
\draw (1,-1) node [right] {$\eta$};
\draw (0,-2) node [below right] {$-\eta$} ;
\draw (-2,-2) node [below left] {$\xi$};
\draw [thick, red,<->] (0,-2) arc (90:180:-1);
\end{tikzpicture}
\end{minipage}\hspace{1cm}\begin{minipage}{4,5cm}
\begin{tikzpicture}
\draw (0,0) node {$\bullet$};
\draw (-1,-1) node {$\bullet$};
\draw[thick] (0,0)--(-1,-1);
\draw[thick] (0,0)--(1,-1);
\draw[thick] (-1,-1)--(0,-2);
\draw[thick] (-1,-1)--(-2,-2);
\draw (0,0) node [above] {$\xi$};
\draw (1,-1) node [right] {$\eta$};
\draw (0,-2) node [above right] {$\xi$};
\draw (-2,-2) node [below left] {$-\eta$};
\draw [thick, red,<->] (-2,-2) arc (48:168:-1.83);
\end{tikzpicture}
\end{minipage}

Those are the building blocks of trees in $\tildeA_n$. They all contribute in the same way to $\tilde \psi_n$. Therefore, we will discuss only the blocks of the first form, obtaining the other ones by symmetries left-right. By symmetry left-right, we mean the one obtained by exchanging the tree to the left with the tree to the right at any node as
\[
(A_1,A_2) \quad \leftrightarrow \quad (A_2,A_1).
\] 
Up to symmetries, all the trees in $\tildeA_2$ are of the forms described below.

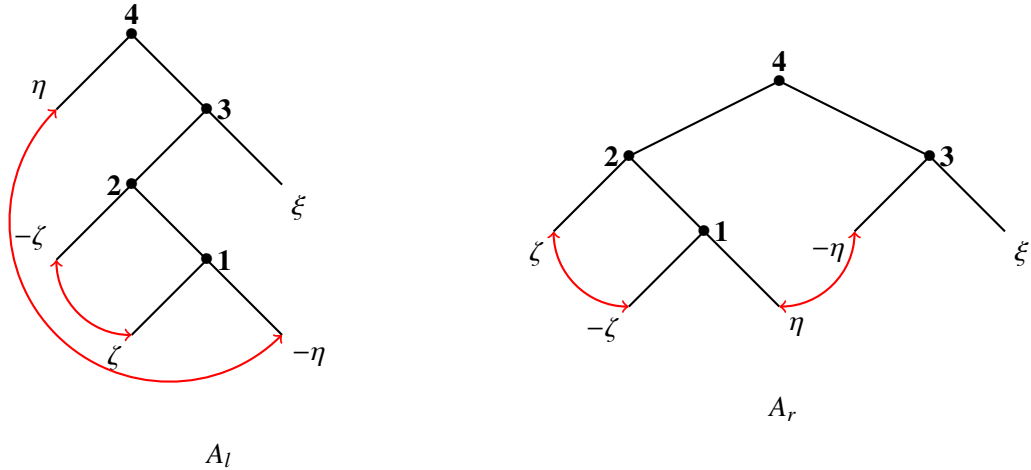
\begin{minipage}{6cm}
\begin{tikzpicture}
\draw (0,0) node {$\bullet$};
\draw (1,-1) node {$\bullet$};
\draw (0,-2) node {$\bullet$};
\draw (1,-3) node {$\bullet$};
\draw[thick] (0,0)--(-1,-1);
\draw[thick] (0,0)--(1,-1);
\draw[thick] (1,-1)--(0,-2);
\draw[thick] (1,-1)--(2,-2);
\draw[thick] (0,-2)--(-1,-3);
\draw[thick] (0,-2)--(1,-3);
\draw[thick] (1,-3)--(0,-4);
\draw[thick] (1,-3)--(2,-4);
\draw (0,-4) node [below left] {$\zeta$};
\draw (-1,-3) node [above left] {$-\zeta$};
\draw (-1,-1) node [above left] {$\eta$};
\draw (2,-4) node [below right] {$-\eta$};
\draw (2,-2) node [below right] {$\xi$};
\draw (0,0) node [above] {$\bf 4$};
\draw (1,-1) node [right]{$\bf 3$};
\draw (0,-2) node [left] {$\bf 2$};
\draw (1,-3) node [right ]{$\bf 1$};
\draw [thick, red,<->] (-1,-3) arc (0:90:-1);
\draw [thick, red,<->] (-1,-1) arc (-45:135:-2.12);
\end{tikzpicture}
\[
A_l
\]
\end{minipage}\hspace{1cm}\begin{minipage}{7cm}
\begin{tikzpicture}
\draw (0,0) node {$\bullet$};
\draw (2,-1) node {$\bullet$};
\draw (-2,-1) node {$\bullet$};
\draw (-1,-2) node {$\bullet$};
\draw[thick] (0,0)--(-2,-1);
\draw[thick] (0,0)--(2,-1);
\draw[thick] (2,-1)--(1,-2);
\draw[thick] (2,-1)--(3,-2);
\draw[thick] (-2,-1)--(-3,-2);
\draw[thick] (-2,-1)--(-1,-2);
\draw[thick] (-1,-2)--(-2,-3);
\draw[thick] (-1,-2)--(0,-3);
\draw (0,0) node [above]{$\bf 4$};
\draw (2,-1) node [right] {$\bf 3$};
\draw (-2,-1) node [left] {$\bf 2$};
\draw (-1,-2) node [right] {$\bf 1$};
\draw (-3,-2) node [below left] {$\zeta$};
\draw (-2,-3) node [below left] {$-\zeta$};
\draw (0,-3) node [below right] {$\eta$};
\draw (1,-2) node [below left] {$-\eta$};
\draw (3,-2) node [below right] {$\xi$};
\draw [thick, red,<->] (-3,-2) arc (0:90:-1);
\draw [thick, red,<->] (0,-3) arc (-90:0:1);
\end{tikzpicture} 
\[
A_r
\]
\end{minipage}

\begin{center}
\begin{tikzpicture}
\draw (0,0) node {$\bullet$};
\draw (1,-1) node {$\bullet$};
\draw (2,-2) node {$\bullet$};
\draw (3,-3) node {$\bullet$};
\draw[thick] (0,0)--(-1,-1);
\draw[thick] (0,0)--(1,-1);
\draw[thick] (1,-1)--(0,-2);
\draw[thick] (1,-1)--(2,-2);
\draw[thick] (2,-2)--(1,-3);
\draw[thick] (2,-2)--(3,-3);
\draw[thick] (3,-3)--(2,-4);
\draw[thick] (3,-3)--(4,-4);
\draw (0,0) node [right] {$\bf 4$};
\draw (1,-1) node [right] {$\bf 3$};
\draw (2,-2) node [right] {$\bf 2$};
\draw (3,-3) node [right] {$\bf 1$};
\draw (-1,-1) node [below left] {$\eta$};
\draw (0,-2) node [below left] {$-\eta$};
\draw (1,-3) node [below left] {$\zeta$};
\draw (2,-4) node [below left] {$-\zeta$};
\draw (4,-4) node [below right] {$\xi$};
\draw [thick, red,<->] (-1,-1) arc (0:90:-1);
\draw [thick, red,<->] (1,-3) arc (0:90:-1);
\end{tikzpicture} 
\[
A_c
\]
\end{center}

We remark that with the constraint of the order of the nodes inherent to $\tildeA_2$, there is only one $\sigma$ possible which is the order of the nodes written in bold figures.

The two trees of the first line, $A_l$ and $A_r$, have contributions in $\tilde \psi_2$ that compensate each other. Indeed we have 
\begin{align*}
\omega_{A_l}(\sigma_l(4)) = \omega(\xi) = \omega_{A_r}(\sigma_r(4)) ,\\
\omega_{A_l}(\sigma_l(3)) = \omega(\xi-\eta) = \omega_{A_r}(\sigma_r(3)),\\
\omega_{A_l}(\sigma_l(2)) = \omega(-\eta) = -\omega(\eta) = -\omega_{A_r}(\sigma_r(2)), \\
\omega_{A_l}(\sigma_l(1)) = \omega(\zeta-\eta) = -\omega(\eta-\zeta) = -\omega_{A_r}(\sigma_r(1))
\end{align*}
and
\begin{align*}
\Delta_{A_l}(\sigma_l(4)) = \omega(\xi) - \omega(\eta)-\omega(\xi-\eta) = \Delta_{A_r}(\sigma_r(4)),\\
\Delta_{A_l}(\sigma_l(2)) = \omega(-\eta) - \omega(-\zeta) - \omega(\zeta-\eta) = -(\omega(\eta) - \omega(\zeta) - \omega(\eta - \zeta) ) = - \Delta_{A_r}(\sigma_r(2)).
\end{align*}
The tree of the second line $A_c$ is what is called a comb and is the only one to contribute to $\tilde \psi_2$. We also note that only one leaf in each tree is left uncoupled.

More generally, in a tree of $\tildeA_n$ with the constraints of the $\Delta$, all leaves are coupled but one. To build a tree in $\tildeA_{n+1}$ from the tree in $\tildeA_n$, there are two ways. Either we add a building block to the uncoupled leaf or we take two coupled leaves and we add a building block in the following ways.

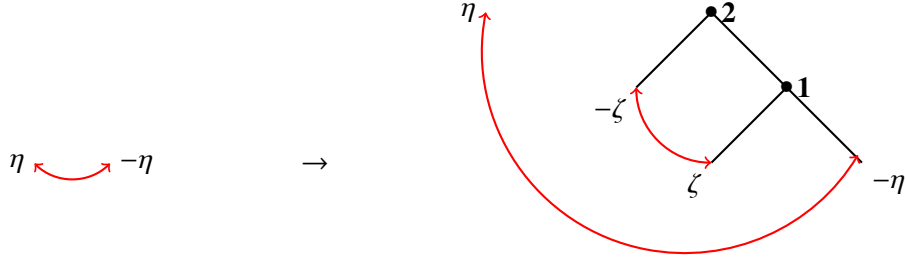
\begin{minipage}{3cm}
\begin{tikzpicture}
\draw (0,0) node [left] {$\eta$};
\draw (1,0) node [right] {$-\eta$};
\draw [thick,red,<->] (0,0) arc (-135:-45:0.71);
\end{tikzpicture}
\end{minipage}\hspace{1cm}\begin{minipage}{1cm}
$\rightarrow$
\end{minipage}\hspace{1cm}\begin{minipage}{7cm}
\begin{tikzpicture} 
\draw (0,0) node {$\bullet$};
\draw (1,-1) node {$\bullet$};
\draw (0,0) node [right] {$\bf 2$};
\draw (1,-1) node [right] {$\bf 1$};
\draw [thick] (0,0)--(-1,-1);
\draw [thick] (0,0)--(1,-1);
\draw [thick] (1,-1)--(0,-2);
\draw [thick] (1,-1)--(2,-2);
\draw (-1,-1) node [below left] {$\zeta$};
\draw (0,-2) node [below left] {$-\zeta$};
\draw (2,-2) node [below right] {$\eta$};
\draw (4,0) node [right] {$-\eta$};
\draw [thick,red,<->] (-1,-1) arc (0:90:-1);
\draw [thick,red,<->] (2,-2) arc (-90:0:2);
\end{tikzpicture}
\[
A_u
\]
\end{minipage}

\vspace{1cm}

\begin{minipage}{3cm}
\begin{tikzpicture}
\draw (0,0) node [left] {$\eta$};
\draw (1,0) node [right] {$-\eta$};
\draw [thick,red,<->] (0,0) arc (-135:-45:0.71);
\end{tikzpicture}
\end{minipage}\hspace{1cm}\begin{minipage}{1cm}
$\rightarrow$
\end{minipage}\hspace{1cm}\begin{minipage}{7cm}
\begin{tikzpicture} 
\draw (0,0) node {$\bullet$};
\draw (1,-1) node {$\bullet$};
\draw (0,0) node [right] {$\bf 2$};
\draw (1,-1) node [right] {$\bf 1$};
\draw [thick] (0,0)--(-1,-1);
\draw [thick] (0,0)--(1,-1);
\draw [thick] (1,-1)--(0,-2);
\draw [thick] (1,-1)--(2,-2);
\draw (-1,-1) node [below left] {$-\zeta$};
\draw (0,-2) node [below left] {$\zeta$};
\draw (2,-2) node [below right] {$-\eta$};
\draw (-3,0) node [left] {$\eta$};
\draw [thick,red,<->] (-1,-1) arc (0:90:-1);
\draw [thick,red,<->] (-3,0) arc (169:329:2.69);
\end{tikzpicture}
\[
A_d
\]
\end{minipage}

The two trees $A_u$ and $A_d$, have contributions in $\tilde \psi_{2n+2}$ that compensate each other. Indeed we have 
\begin{align*}
\omega_{A_d}(\sigma_d(2)) = \omega(-\eta) = -\omega(\eta) = -\omega_{A_u}(\sigma_u(2)), \\
\omega_{A_d}(\sigma_d(1)) = \omega(\zeta-\eta) = -\omega(\eta-\zeta) = -\omega_{A_u}(\sigma_u(1))
\end{align*}
and
\[
\Delta_{A_d}(\sigma_d(2)) = \omega(-\eta) - \omega(-\zeta) - \omega(\zeta-\eta) = -(\omega(\eta) - \omega(\zeta) - \omega(\eta - \zeta) ) = - \Delta_{A_u}(\sigma_u(2)).
\]

The only way to form trees whose contributions do not compensate each other is to add a building block to the uncoupled leaf, meaning that up to symmetries, the only tree contributing to $\tilde \psi_n$ is :
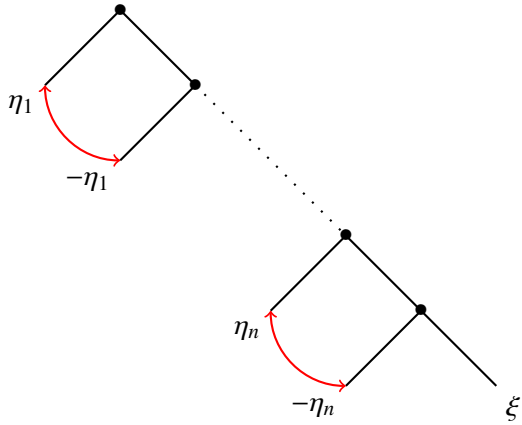
\begin{center}
\begin{tikzpicture}
\draw (0,0) node {$\bullet$};
\draw (1,-1) node {$\bullet$};
\draw (3,-3) node {$\bullet$};
\draw (4,-4) node {$\bullet$};
\draw [thick] (0,0)--(-1,-1);
\draw [thick] (0,0)--(1,-1);
\draw [thick] (1,-1)--(0,-2);
\draw [thick,loosely dotted] (1,-1)--(3,-3);
\draw [thick] (3,-3)--(2,-4);
\draw [thick] (3,-3)--(4,-4);
\draw [thick] (4,-4)--(3,-5);
\draw [thick] (4,-4)--(5,-5);
\draw (-1,-1) node [below left] {$\eta_1$};
\draw (0,-2) node [below left] {$-\eta_1$};
\draw (2,-4) node [below left] {$\eta_n$};
\draw (3,-5) node [below left] {$-\eta_n$};
\draw (5,-5) node [below right] {$\xi$};
\draw [thick,red,<->] (-1,-1) arc (0:90:-1);
\draw [thick,red,<->] (2,-4) arc (0:90:-1);
\end{tikzpicture}
\end{center}

The proof, however, is a lot more simple than this because we observe this compensation at a global level. But the local phenomenon that is encompassed is the one we have just described.
\end{remark}

\begin{proof} As we have seen previously if $(A,\sigma)\in \tildeA_{n+1}$ then there exists a unique triplet 
\[
((A_1,\sigma_1),(A_2,\sigma_2),(A_3,\sigma_3)) \in \tildeA_{n_1}\times \tildeA_{n_2} \times \tildeA_{n_3}
\]
such that $n_1+n_2+n_3 = n$ and either $A=(A_1,(A_2,A_3))$ or $A=((A_2,A_3),A_1)$ and $\sigma$ is an ordering on the nodes compatible with $\sigma_1$, $\sigma_2$ and $\sigma_3$ satisfying $\sigma(2j-1)\leq_A \sigma(2j)$. 

Conversely, if $n_1+n_2+n_3 = n$ and $(A_j,\sigma_j) \in \tildeA_{n_j}$ we can build several trees  $(A,\sigma) \in \tildeA_{n+2}$. Indeed, we can build either $A=(A_1,(A_2,A_3))$ or $A=((A_2,A_3),A_1)$ and $\sigma$ should be compatible with $\sigma_1,\sigma_2,\sigma_3$ and such that $\sigma(2j-1)\leq_A \sigma(2j)$. We note this property $\sigma >^A (\sigma_1,\sigma_2,\sigma_3)$. 

We deduce 
\small
\begin{align*}
\tildeA_{n+1} = \{ ((A_1,(A_2,A_3)), \sigma)\; |\; \exists (n_j)_{j=1,2,3}, n_1+n_2+n_3=n, \exists (\sigma_j)_{j=1,2,3},   (A_j,\sigma_j)\in \tildeA_{n_j}, \sigma >^A (\sigma_1,\sigma_2,\sigma_3)\}\\
\sqcup \{(((A_2,A_3),A_1),\sigma)\; |\; \exists (n_j)_{j=1,2,3}, n_1+n_2+n_3=n, \exists (\sigma_j)_{j=1,2,3}, (A_j,\sigma_j)\in \tildeA_{n_j}, \sigma >^A (\sigma_1,\sigma_2,\sigma_3)\}.
\end{align*}
\normalsize
Therefore, writing 
\[
U_n(a,\xi) = \sum_{(A,\sigma)\in \tildeA_n} U_A^\sigma(a,\xi)
\]
we get
\begin{align*}
U_{n+1}(a,\xi) = \sum_{n_1+n_2+n_3 = n} \sum_{(A_j,\sigma_j)\in \tildeA_{n_j}} \sum_{\sigma >^{(A_1,(A_2,A_3))} (\sigma_1,\sigma_2,\sigma_3)} U_{(A_1,(A_2,A_3))}^\sigma \\
+ \sum_{n_1+n_2+n_3 = n} \sum_{(A_j,\sigma_j)\in \tildeA_{n_j}} \sum_{\sigma >^{((A_2,A_3),A_1)} (\sigma_1,\sigma_2,\sigma_3)} U_{((A_2,A_3),A_1)}^\sigma.
\end{align*}
If $\sigma >^A (\sigma_1,\sigma_2,\sigma_3)$ for $A=(A_1,(A_2,A_3))$ and $\sigma' >^{A'} (\sigma_1,\sigma_2,\sigma_3) $ for $A'=((A_2,A_3),A_1)$ we have 
\begin{eqnarray*}
U_A^\sigma (a,\xi)&=& \frac1{2\pi L} \sum_\eta F(\xi,\eta) |a(\eta)|^2 U_{A_1}^{\sigma_1}(a,\eta)\Big( U_{A_2}^{\sigma_2}(a,-\eta) U_{A_3}^{\sigma_3}(a,\xi) + U_{A_2}^{\sigma_2}(a,\xi) U_{A_3}^{\sigma_3}(a,-\eta)\Big)\\ &=& U_{A'}^{\sigma'}(a,\xi).
\end{eqnarray*}
We compute the cardinal of $\{\sigma >^A(\sigma_1,\sigma_2,\sigma_3)\}$. We treat the case $A=(A_1,(A_2,A_3))$, the other one being similar. Take $\sigma >^A(\sigma_1,\sigma_2,\sigma_3)$. We have $\sigma(2n+2) = 0$ and $\sigma(2n+1) = (2,0)$. Now, to describe $\sigma (1),\hdots,\sigma(2n)$, because of Lemma \ref{lem:parity}, it is equivalent to choose $(\varphi_1,\varphi_2,\varphi_3)$ three increasing functions such that for $k\in \{1,2,3\}$,
\[
\varphi_k : [|1,n_k |]\rightarrow [|1,n|]
\]
and such that 
\[
\varphi_1([|1,n_1|]) \sqcup \varphi_2([|1,n_2|]) \sqcup \varphi_3([|1,n_3|])= [|1,n|].
\]
Then, for $j\in [|1,n|]$, there exists a unique $k\in \{1,2,3\}$ and a unique $j_k \in [|1,n_k|]$ such that $j= \varphi_k(j_k)$ and we set 
\[
\sigma(2j) = \left \lbrace{\begin{array}{cc}
(1,\sigma_1(2j_1)) & \textrm{ if } k=1,\\
(2,(1,\sigma_2(2j_2))) & \textrm{ if } k=2,\\
(2,(2,(\sigma_3(2j_3))) & \textrm{ if } k=3,
\end{array}}\right.  \quad 
\sigma(2j-1) = \left \lbrace{\begin{array}{cc}
(1,\sigma_1(2j_1-1)) & \textrm{ if } k=1,\\
(2,(1,\sigma_2(2j_2-1))) & \textrm{ if } k=2,\\
(2,(2,(\sigma_3(2j_3-1))) & \textrm{ if } k=3.
\end{array}}\right.
\]

To build such maps, because they are increasing, it is sufficient to decide for each element of $[|1,n|]$ if it will correspond to an element or $[|1,n_1|]$, $[|1,n_2|]$ or $[|1,n_3|]$, that is, we choose $n_1$ elements among $n$ then $n_2$ among $n-n_1$ and $n_3$ elements are left.

We deduce
\[
\# \{\sigma >^A (\sigma_1,\sigma_2,\sigma_3) \} = \frac{n!}{n_1! n_2! n_3!}.
\]
Therefore, we get
\begin{multline*}
U_{n+1}(a,\xi) = 
\frac1{\pi L}\sum_{n_1+n_2+n_3 = n} \sum_{(A_j,\sigma_j)\in \tildeA_{n_j}} \frac{n!}{n_1! n_2! n_3!} \\
\sum_\eta |a(\eta)|^2 F(\xi,\eta) U_{A_1}^{\sigma_1}(a,\eta)\Big( U_{A_2}^{\sigma_2}(a,-\eta) U_{A_3}^{\sigma_3}(a,\xi) + U_{A_2}^{\sigma_2}(a,\xi) U_{A_3}^{\sigma_3}(a,-\eta)\Big) .
\end{multline*}
The summand is symmetric in $A_2$ and $A_3$, we deduce
\begin{multline*}
U_{n+1}(a,\xi) = \frac2{\pi L}\sum_{n_1+n_2+n_3 = n} \sum_{(A_j,\sigma_j)\in \tildeA_{n_j}} \frac{n!}{n_1! n_2! n_3!} \\
\sum_\eta |a(\eta)|^2 F(\xi,\eta) U_{A_1}^{\sigma_1}(a,\eta)U_{A_2}^{\sigma_2}(a,-\eta) U_{A_3}^{\sigma_3}(a,\xi) .
\end{multline*}
The sets $\tildeA_j$ are finite, therefore we can exchange the sums on $\frac1{L}\Z^*$ and $\tildeA_{n_j}$.
\[
U_{n+1}(a,\xi) = \frac2{\pi L}\sum_{n_1+n_2+n_3 = n}  \frac{n!}{n_1! n_2! n_3!}\sum_\eta |a(\eta)|^2 F(\xi,\eta) U_{n_1}(a,\eta)U_{n_2}(a,-\eta) U_{n_3}(a,\xi) .
\]
We prove by induction on $n$ that
\[
U_n(a,\xi) = U(a,\xi)^{n}.
\]
For $n=0$, we have indeed $U_0(a,\xi) = U_{\bot}^{\emptyset}(a,\xi)=1$.

By the induction formula, we have 
\[
U_{n+1}(a,\xi) = \frac2{\pi L}\sum_\eta |a(\eta)|^2 F(\xi,\eta) \sum_{n_1+n_2+n_3 = n}  \frac{n!}{n_1! n_2! n_3!}    U^{n_1}(a,\eta)U^{n_2}(a,-\eta) U^{n_3}(a,\xi) .
\]
We deduce that
\[
U_{n+1}(a,\xi) = \frac2{\pi L}\sum_\eta |a(\eta)|^2 F(\xi,\eta) (U(a,\eta) +U(a,-\eta) + U(a,\xi))^{n} .
\]
We now remark that $U(a,-\eta) =-U(a,\eta)$. Indeed, by the odd character of $\omega$, we have 
\[
F(\xi,\eta) = -F(-\xi,-\eta)
\]
and thus
\[
U(a,-\eta) = \frac2{\pi L} \sum_\zeta F(-\eta,\zeta) |a(\zeta)|^2 = -\frac2{\pi L} \sum_\zeta F(\eta,-\zeta) |a(\zeta)|^2
\]
and we do the change of variables, $\zeta \leftarrow -\zeta$ to conclude. 

We deduce
\[
U_{n+1}(a,\xi) = \frac2{\pi L}\sum_{\eta}|a(\eta)|^2 F(\xi,\eta)  U(a,\xi)^{n} = U(a,\xi)^{n+1} .
\]

Now we have that
\[
U(a,\xi) = \frac12\Big( U(a,\xi) - U(a,-\xi)\Big) = \frac1{\pi L} \sum_\eta |a(\eta)|^2 \Big( F(\xi,\eta) - F(-\xi,\eta) \Big).
\]
We have 
\[
F(\xi,\eta) = \frac1{\omega(\eta) (3 + \frac{\eta^2 + \xi^2 + (\xi-\eta)^2}{2})}.
\]
We deduce
\[
U(a,\xi) = \frac1{\pi L} \sum_\eta |a(\eta)|^2 \frac1{\omega(\eta)}\Big(\frac1{3 + \frac{\eta^2 + \xi^2 + (\xi-\eta)^2}{2}}- \frac1{3 + \frac{\eta^2 + \xi^2 + (\xi+\eta)^2}{2}} \Big)
\]
and thus
\[
U(a,\xi) = \frac2{\pi L} \sum_\eta |a(\eta)|^2 \frac{(1+\eta^2)\xi}{(3 + \frac{\eta^2 + \xi^2 + (\xi-\eta)^2}{2})(3 + \frac{\eta^2 + \xi^2 + (\xi+\eta)^2}{2})}
\]
where we recognise 
\[
\tilde F(\xi,\eta) = \frac{(1+\eta^2)\xi}{(3 + \frac{\eta^2 + \xi^2 + (\xi-\eta)^2}{2})(3 + \frac{\eta^2 + \xi^2 + (\xi+\eta)^2}{2})}.
\]

We deduce
\[
\an{\xi}|U(a,\xi)| \leq \frac8{\pi L} \sum_\eta |a(\eta)|^2 = \frac{8\|a\|_{L^2}^2}{\pi L}.
\]
\end{proof}


\section{Estimate on the remainder}\label{sec:EstRemInvMea}

\subsection{Comparison between \texorpdfstring{$\psi$}{psi} and \texorpdfstring{$\tilde \psi$}{~ psi}}\label{subsec:PropagEst}

From the previous computation, by definition of $\tilde \psi_n$ we have
\[
\widehat{\tilde \psi}_n (t)(a)(\xi) = \frac{(-i\varepsilon^2 t U(a,\xi))^{n/2}}{(n/2)!}  \hat a(\xi)
\]
and thus
\[
\Big| \widehat{\tilde \psi}_n (t)(a)(\xi)\Big| \leq \Big( \frac{8\varepsilon^2 t \|a\|_{L^2}^2}{\pi L \an{\xi}}\Big)^{n/2}\frac{|\hat a(\xi)|}{(n/2)!}.
\]
We deduce that at all times the series $\sum_n \tilde \psi_n$ converges in $H^s$ and we set
\[
\tilde \psi(t)(a) = \sum_n \tilde \psi_n(t)(a),
\]
that is
\[
\widehat{\tilde \psi}(t)(a)(\xi) = e^{-it\varepsilon^2 U(a,\xi)} \hat a(\xi).
\]
Note that 
\[
U(\tilde \psi(t)(a),\xi) = U(a,\xi)
\]
and that $\tilde \psi$ conserves any $H^s$ norm. We set
\[
\tilde \Psi_{\varepsilon,L} (t) = S(t) \psi_{\varepsilon,L}(t).
\]

\begin{proposition} There exists $C$ such that for all $D$, all $L$, all $T \leq T_1(D) = \frac1{CD^2}$ for all $t\in [-T\varepsilon^{-2},T\varepsilon^{-2}]$, for all $s\geq 0$, and all $a$ such that $\|a\|_{H^s}\leq D\sqrt L$, we have 
\[
\|\tilde \psi_n(t)(a) \|_{H^s} \leq \frac1{(n/2)!} D\sqrt L 
\]
and for all $b$ such that $\|b\|_{H^s}\leq D\sqrt L$,
\[
\|\tilde \psi(t)(a) - \tilde \psi(t) (b)\|_{H^s} \leq 2 \|a-b\|_{H^s}.
\]
\end{proposition}

\begin{proof} 
\[
\|\tilde \psi_n(t)(a)\|_{H^s} \leq \Big[\sup_\xi |U(a,\xi)| \varepsilon^{2}t\Big]^{n/2} \frac{\|a\|_{H^s}}{(n/2)!}.
\]
We have 
\[
\sup_\xi |U(a,\xi)| \varepsilon^{2}t \leq C \|a\|_{L^2}^2 L^{-1} T\leq CD^2 T
\]
which is less than $1$ if the constant defining $T_1$ is well-chosen.

What is more, we have 
\[
\|\tilde \psi(t)(a)-\tilde \psi(t) (b)\|_{H^s} \leq \Big[\sup_\xi |U(a,\xi) -U(b,\xi)| \varepsilon^{2}t\Big] \|a\|_{H^s}+ \|a-b\|_{H^s}.
\]
We have 
\[
U(a,\xi) - U(b,\xi) = \frac2{\pi L} \sum_\eta \tilde F(\xi,\eta) (|a(\eta)|^2 - |b(\eta)|^2)
\]
with 
\[
\tilde F(\xi,\eta) =\frac{(1+\eta^2)\xi}{(3 + \frac{\eta^2 + \xi^2 + (\xi-\eta)^2}{2})(3 + \frac{\eta^2 + \xi^2 + (\xi+\eta)^2}{2})}
\]
which is uniformly bounded in $\xi$ and $\eta$, we deduce by Cauchy-Schwarz,
\[
|U(a,\xi) - U(b,\xi)|\leq T C \|a-b\|_{L^2} (\|a\|_{L^2} + \|b\|_{L^2}) L^{-1} \leq CT D L^{-1/2} \|a-b\|_{L^2}.
\]
This concludes the proof.
\end{proof}

\begin{proposition}\label{prop:constant} Let $s> \frac14$ and $\frac14 \leq s'<s$. For all $L,D$, all 
\[
T\leq T_2(L,D) = \min (T_0(L,D),T_1(L,D)),
\]
all $N\in \N$ and $R\geq 1$, there exists $C(L,D,R,N)$ such that for all $\varepsilon\leq \epsilon_0(L,D)$, for all $t\in [-T\varepsilon^{-2},T\varepsilon^{-2}]$ and finally, for all $a$ such that $\|a\|_{H^s}\leq D\sqrt L$, we have 
\[
\|\psi(t)(a) - \tilde\psi(t)(a)\|_{H^{s'}} \leq D\sqrt L \Big( 2^{-N} + \sum_{n>N/2}\frac1{n!} + 3R^{-(s-s')}\Big) + \varepsilon C(L,D,R,N).
\]
\end{proposition}

\begin{proof}
We have 
\begin{eqnarray*}
\|\psi(t)(a)-\tilde\psi(t)(a)\|_{H^{s'}} &\leq & \|\psi(t)(a) - \sum_{n\leq N} \psi_n(t)(a)\|_{H^{s'}} \\
 & & + \sum_{n\leq N} \|\psi_n(t)(a) - \tilde \psi_n(t)(a_R)\|_{H^s} \\
  & & + \|\tilde \psi(t)(a_R) - \sum_{n\leq N} \tilde \psi_n(t)(a_R)\|_{H^{s'}} \\
  & & + \|\tilde \psi(t)(a_R) - \tilde\psi(t)(a)\|_{H^{s'}}.
\end{eqnarray*}

By Proposition \ref{prop:fstestonPsin}, we have 
\[
\|\psi(t)(a) - \sum_{n\leq N} \psi_n(t)(a)\|_{H^{s'}} \leq 2^{-N} D\sqrt L.
\]
By Proposition \ref{prop:sndestonPsin}, we have 
\[
 \|\psi_n(t)(a) - \tilde \psi_n(t)(a_R)\|_{H^s} \leq 2^{-n} R^{-(s-s')} D\sqrt L + \varepsilon C(L,D,R,N).
\]
By the previous proposition, we have 
\[
\|\tilde \psi(t)(a_R) - \sum_{n\leq N} \tilde \psi_n(t)(a_R)\|_{H^{s'}} \leq \sum_{n>N, n \textrm{ even}} \frac1{(n/2)!}
\]
and
\[
\|\tilde \psi(t)(a_R) - \tilde\psi(t)(a)\|_{H^{s'}}\leq 2\|a-a_R\|_{H^{s'}} \leq 2R^{-(s-s')}D\sqrt L.
\]
We sum and get the result.
\end{proof}

\begin{proposition}\label{prop:psiandtildepsiPropagated} Let $t\in \R_+$, let $L,D\geq 1$ (with $L$ not algebraic of degree less than $2$), set
\[
M(t,L,D) = \lfloor \frac{t}{T_2(L,D)}\rfloor.
\]
Assume that for all $m = 0,\hdots,M$, we have 
\[
\|\psi(mT_2(L,D)\varepsilon^{-2})(a)\|_{H^s} \leq D\sqrt L
\]
which allows to define the flow up to $t\varepsilon^{-2}$. We have that for all $N,R$, there exists $C(L,D,N,R)$ such that for all $\varepsilon \leq \epsilon_0(L,D)$,
\[
\|\psi(t\varepsilon^{-2})(a) - \tilde \psi(t\varepsilon^{-2})(a)\|_{H^{s'}} \leq 2^{t/T_2} \Big[ D\sqrt L (2^{-N} + \sum_{n\geq N/2}\frac1{n!} + 3R^{-(s-s')}) + \varepsilon C(L,D,R,N)\Big].
\]
\end{proposition}

\begin{proof}
Set 
\[
X = \Big[ D\sqrt L (2^{-N} + \sum_{n\geq N/2}\frac1{n!} + 3R^{-(s-s')}) + \varepsilon C(L,D,R,N)\Big]
\]
with the constant of Proposition \ref{prop:constant}, and for $m=0,\hdots, M$,
\[
X_m = \sup_{\tau \in [mT_2\varepsilon^{-2},(m+1)T_2\varepsilon^{-2}]} \|\psi(\tau)(a) - \tilde\psi(\tau)(a)\|_{H^{s'}}.
\]

Let $\tau \in [mT_2\varepsilon^{-2},(m+1)T_2\varepsilon^{-2}]$, we have
\begin{multline*}
 \|\psi(\tau)(a) - \tilde\psi(\tau)(a)\|_{H^{s'}} \leq \|\psi(\tau - mT_2\varepsilon^{-2})(\psi(mT_2\varepsilon^{-2})(a)) - \tilde\psi(\tau - mT_2\varepsilon^{-2})(\psi(mT_2\varepsilon^{-2})(a)) \|_{H^{s'}} \\
 +  \|\tilde\psi(\tau - mT_2\varepsilon^{-2})(\psi(mT_2\varepsilon^{-2})(a)) - \tilde\psi(\tau - mT_2\varepsilon^{-2})(\tilde\psi(mT_2\varepsilon^{-2})(a)) \|_{H^{s'}}.
\end{multline*}
By the previous proposition, we have 
\[
\|\psi(\tau - mT_2\varepsilon^{-2})(\psi(mT_2\varepsilon^{-2})(a)) - \tilde\psi(\tau - mT_2\varepsilon^{-2})(\psi(mT_2\varepsilon^{-2})(a)) \|_{H^{s'}}\leq X.
\]
By continuity of $\tilde \psi$, we have 
\begin{multline*}
\|\tilde\psi(\tau - mT_2\varepsilon^{-2})(\psi(mT_2\varepsilon^{-2})(a)) - \tilde\psi(\tau - mT_2\varepsilon^{-2})(\tilde\psi(mT_2\varepsilon^{-2})(a)) \|_{H^{s'}}\\
\leq 2 \|\psi(mT_2\varepsilon^{-2})(a) - \tilde\psi(mT_2\varepsilon^{-2})(a) \|_{H^{s'}}.
\end{multline*}
We deduce 
\[
X_{m} \leq X+2X_{m-1} .
\]
Because $X_0 \leq X$, we deduce by induction 
\[
X_m \leq X2^{M} \leq X2^{t/T_2}
\]
which concludes the proof.
\end{proof}

\subsection{The set we work in}\label{subsec:invar}

We recall that $\mu_L$ is the law of the initial datum. When this law is invariant, we write it $\nu_L$.

\begin{proposition} Let $t\in \R_+$. Let $L,D \geq 1$. Let $ M(t,L,D) = \lfloor \frac{t}{T_2(L,D)}\rfloor$. Let $\frac14<s<\frac12$ in the case of the invariant measure $\varphi(\xi) = \an{\xi}^{-1}$. Let $s=1$ otherwise. Set
\[
B(t,L,D) = \{ a \in H^s(L \T) \; |\; \forall m=0,\hdots, M, \; \|\psi_{L,\varepsilon}(m T_2(L,D)\varepsilon^{-2})(a)\|_{H^s} \leq D\sqrt L\}.
\]
There exists $C,c>0$ such that for all $L,D,t$, we have 
\[
\mu_L ( B(t,L,D)^c) \leq M C e^{-c D^2}.
\]
\end{proposition}

\begin{proof} We have 
\[
B(t,L,D) = \bigcap_{m=0}^M B_m(L,D)
\]
with
\[
B_m(L,D) = \{ a \in H^s(L \T) \; |\; \|\psi_{L,\varepsilon}(m T_2(L,D)\varepsilon^{-2})(a)\|_{H^s} \leq D\sqrt L\}.
\]
Therefore, we have 
\[
B(t,L,D)^c = \bigcup_{m=0}^M A_m(L,D)^c.
\]
We deduce 
\[
\mu_L(B(t,L,D)^c) \leq  \sum_{m=0}^M \mu_L(A_m(L,D)^c).
\]
Since the linear flow preserves the Sobolev norms, we get
\[
\mu_L(A_m(L,D)^c) = \mu_L(\{\Psi_{L,\varepsilon}(m T_2(L,D)\varepsilon^{-2})(a)\|_{H^s} \leq D\sqrt L\}).
\]
Since $\nu_L$ and the $H^1$ norm are invariant under $\Psi_{L,\varepsilon}$, we deduce
\[
\mu_L(B(t,L,D)^c) \leq  M \mu_L(B_0(L,D)^c).
\]
We estimate $\mu_L(B_0(L,D)^c)$. By definition, we have 
\[
B_0(L,D)^c = \{ a \in H^s \; |\; \|a\|_{H^s}> D\sqrt L\}.
\]
By Fernique's theorem, there exist universal constants $C,c$ such that
\[
\mu_L(A_0(L,D)^c) \leq Ce^{-cD^2 L /r(L)} 
\]
where
\[
r(L) = \E_{\mu_L}(\|a\|_{H^s}^2) = \E( \|\varphi_L\|_{H^s}^2).
\]
We have 
\[
r(L) = \sum_{n\in \Z^*} \an{n/L}^{2s}|\varphi(\frac{n}{L})|^2 \E(|g_n|^2) = \sum_{n\in \Z^*}  \an{n/L}^{2s}|\varphi(\frac{n}{L})|^2 
\]
and thus
\[
r(L) \leq C(\varphi) L  .
\]
We deduce the result.
\end{proof}

\section{Proof of the result}\label{sec:proof}

In the sequel, we identify $B(t,L,D)$ and $\varphi_L^{-1}(B(t,L,D))$. 

\begin{lemma}\label{lem:tandminust} For all $t\in \R$, we have 
\begin{eqnarray}\label{minust1}
\E\Big(\overline{\hat S(t\varepsilon^{-2} (\varphi_L)(\xi)} \hat \Psi_{\varepsilon,L} (t\varepsilon^{-2})(\varphi_L)(\xi)\Big)&=& \overline{\E\Big(\overline{\hat S(-t\varepsilon^{-2} (\varphi_L)(\xi)} \hat \Psi_{\varepsilon,L} (-t\varepsilon^{-2})(\varphi_L)(\xi)\Big)}\\ \label{minust2}
&=& \E\Big( \overline{\hat \varphi_L (\xi)} \hat \psi_{\varepsilon,L}(t\varepsilon^{-2}) (\varphi_L)(\xi)\Big).
\end{eqnarray}
\end{lemma}

\begin{proof} Let $\textrm{Opp}$ be the operator acting on $L^2(L\T)$ such that
\[
\hat{\textrm{Opp}} (u)(\xi) = \hat u(-\xi).
\]
We have $\textrm{Opp}^2 = \textrm{Opp}$, $\textrm{Opp}\, W = -W\, \textrm{Opp}$ and $\textrm{Opp}(u^2) = (\textrm{Opp}(u))^2$. 

Therefore, we deduce that
\[
\Psi_{\varepsilon,L}(-t\varepsilon^{-2}) (a) = \textrm{Opp} \Big( \Psi_{\varepsilon,L}(t\varepsilon^{-2}) (\textrm{Opp}(a))\Big),
\]
and that
\[
S(-t\varepsilon^{-2} )(a) = \textrm{Opp} \Big( S(t\varepsilon^{-2})(\textrm{Opp}(a))\Big).
\]

What is more, we have 
\[
\textrm{Opp}(\varphi_L)(x) = \sum_{n\in \Z^*} g_{-n} \varphi(\frac{n}{L}) \frac{e^{inx/L}}{\sqrt{2\pi L}}
\]
and since the family $(g_{-n})_{n\in\Z^*}$ has the same law as $(g_n)_{n\in \Z^*}$ we deduce that $\varphi_L$ and $\textrm{Opp}(\varphi_L)$ have the same law.

Because $S(t\varepsilon^{-2})(\varphi_L)$ and $\Psi_{\varepsilon,L}(t\varepsilon^{-2}) (\varphi_L)$ are real-valued, we have
\[
\E\Big(\overline{\hat S(t\varepsilon^{-2} (\varphi_L)(\xi)} \hat \Psi_{\varepsilon,L} (t\varepsilon^{-2})(\varphi_L)(\xi)\Big)= \E\Big(\hat S(t\varepsilon^{-2} (\varphi_L)(-\xi)\overline{ \hat \Psi_{\varepsilon,L} (t\varepsilon^{-2})(\varphi_L)(-\xi)}\Big).
\]
We recognize the application of $\textrm{Opp}$ as in
\[
\E\Big(\overline{\hat S(t\varepsilon^{-2} (\varphi_L)(\xi)} \hat \Psi_{\varepsilon,L} (t\varepsilon^{-2})(\varphi_L)(\xi)\Big)= \E\Big(\hat{\textrm{Opp}}( S(t\varepsilon^{-2} (\varphi_L))(\xi)\overline{ \hat{\textrm{Opp}} (\Psi_{\varepsilon,L} (t\varepsilon^{-2})(\varphi_L))(\xi)}\Big).
\]
We commute the flows with $\textrm{Opp}$ and get
\[
\E\Big(\overline{\hat S(t\varepsilon^{-2} (\varphi_L)(\xi)} \hat \Psi_{\varepsilon,L} (t\varepsilon^{-2})(\varphi_L)(\xi)\Big)= \E\Big( \hat S(-t\varepsilon^{-2}) (\textrm{Opp}(\varphi_L))(\xi)\overline{ \hat \Psi_{\varepsilon,L} (-t\varepsilon^{-2})(\textrm{Opp} (\varphi_L))(\xi)}\Big).
\]
We use that the law of $\varphi_L$ is invariant under the application of $\textrm{Opp}$ and get
\[
\E\Big(\overline{\hat S(t\varepsilon^{-2} (\varphi_L)(\xi)} \hat \Psi_{\varepsilon,L} (t\varepsilon^{-2})(\varphi_L)(\xi)\Big)= \E\Big( \hat S(-t\varepsilon^{-2}) (\varphi_L)(\xi)\overline{ \hat \Psi_{\varepsilon,L} (-t\varepsilon^{-2})(\varphi_L)(\xi)}\Big)
\]
which is \eqref{minust1}.

We have 
\[
\overline{\hat S(t) u (\xi)} \hat S(t) v (\xi) = e^{-it \omega(\xi)} \overline{\hat u(\xi)} e^{it\omega(\xi)} \hat v(\xi) = \overline{\hat u(\xi)}\hat v(\xi).
\]
Since $\Psi_{\varepsilon,L}(t) = S(t) \psi_{\varepsilon,L}(t)$, we deduce yields \eqref{minust2}.
\end{proof}

\begin{proof}[Proof of Theorem \ref{th:main}] From Lemma \ref{lem:tandminust}, it is sufficient to study
\[
\E( \overline{\hat \varphi_L(\xi)}\hat \Psi_{\varepsilon,L}(t\varepsilon^{-2})(\varphi_L)(\xi))
\]
for positive times.

Set $\frac14 < s' <s < \frac12$ in the invariant case and $\frac14 < s'<s =1$ otherwise and
\[
I(\xi) = \E( \overline{\hat \varphi_L(\xi)}\hat \Psi_{\varepsilon,L}(t\varepsilon^{-2})(\varphi_L)(\xi)) - 
e^{it\Big( \frac1{\sqrt 3} \frac1{3+\xi^2} \frac1{\sqrt{1+\xi^2/4}}\Big) }\frac1{1+\xi^2}.
\]
We have 
\[
I(\xi) = \mathcal A(\xi) + \mathcal B(\xi) + \mathcal C(\xi) + \mathcal D(\xi) + \mathcal E(\xi)
\]
with
\begin{eqnarray} \label{Afin}
\mathcal A(\xi) &=& \E( \overline{\hat \varphi_L(\xi)}\hat \psi_{\varepsilon,L}(t\varepsilon^{-2})(\varphi_L)(\xi)) - 
\E({\bf 1}_{B(t,L,D)} \overline{\hat \varphi_L(\xi)}\hat \psi_{\varepsilon,L}(t\varepsilon^{-2})(\varphi_L)(\xi)),\\ \label{Bfin}
\mathcal B(\xi) &=& \E({\bf 1}_{B(t,L,D)} \overline{\hat \varphi_L(\xi)}\hat \psi_{\varepsilon,L}(t\varepsilon^{-2})(\varphi_L)(\xi))-\E({\bf 1}_{B(t,L,D)} \overline{\hat \varphi_L(\xi)}\widehat{\tilde \psi}_{\varepsilon,L}(t\varepsilon^{-2})(\varphi_L)(\xi)),\\ \label{Cfin}
\mathcal C(\xi) &=& \E({\bf 1}_{B(t,L,D)} \overline{\hat \varphi_L(\xi)}\widehat{\tilde \psi}_{\varepsilon,L}(t\varepsilon^{-2})(\varphi_L)(\xi)) - \E( \overline{\hat \varphi_L(\xi)}\widehat{\tilde \psi}_{\varepsilon,L}(t\varepsilon^{-2})(\varphi_L))(\xi),\\ \label{Dfin}
\mathcal D(\xi) &=& \E( \overline{\hat \varphi_L(\xi)}\widehat{\tilde \psi}_{\varepsilon,L}(t\varepsilon^{-2})(\varphi_L)(\xi)) - e^{it\Big( \frac2{\pi L} \sum_\eta \tilde F(\xi,\eta) |\varphi(\eta)|^2\Big) }|\varphi(\xi)|^2 ,\\ \label{Efin}
\mathcal E(\xi) &=& e^{it\Big( \frac2{\pi L} \sum_\eta \tilde F(\xi,\eta) |\varphi(\eta)|^2\Big) }|\varphi(\xi)|^2 - e^{it \xi \Phi(\xi) }|\varphi(\xi)|^2.
\end{eqnarray}

By definition of $\mathcal A$, we have  
\[
\mathcal A(\xi) = \E({\bf 1}_{B(t,L,D)^c}\overline{\hat \varphi_L(\xi)}\hat \psi_{\varepsilon,L}(t\varepsilon^{-2})(\varphi_L)(\xi))
\]
from which we deduce
\[
|\mathcal A(\xi )| \leq \|{\bf 1}_{B(t,L,D)^c}\|_{L^3(\Omega)} \|\hat \varphi_L(\xi)\|_{L^3(\Omega)}\|\hat \psi_{\varepsilon,L}(t\varepsilon^{-2})(\varphi_L)(\xi)\|_{L^3(\Omega)}.
\]
Because $S(t\varepsilon^{-2}) $ is a Fourier multiplier by a phase we have 
\[
\|\hat \psi_{\varepsilon,L}(t\varepsilon^{-2})(\varphi_L)(\xi)\|_{L^3(\Omega)} =  \|\hat \Psi_{\varepsilon,L}(t\varepsilon^{-2})(\varphi_L)(\xi)\|_{L^3(\Omega)}.
\]
Since either the law of $\varphi_L$ is invariant under the flow of BBM or the $H^1$ norm is defined and invariant, we have either 
\[
\|\hat \Psi_{\varepsilon,L}(t\varepsilon^{-2})(\varphi_L)(\xi)\|_{L^3(\Omega)} = \|\hat \varphi_L(\xi)\|_{L^3(\Omega)}
\]
or
\[
\|\hat \Psi_{\varepsilon,L}(t\varepsilon^{-2})(\varphi_L)(\xi)\|_{L^3(\Omega)} \leq \| \varphi_L\|_{L^3(\Omega, H^1)}
\]
and because $\varphi_L$ is a Gaussian, we have either
\[
\|\hat \varphi_L(\xi)\|_{L^3(\Omega)} \leq \sqrt 3 \|\hat \varphi_L(\xi)\|_{L^2(\Omega)} = \sqrt{\frac3{1+\xi^2}}
\]
or
\[
\|\varphi_L\|_{L^3(\Omega,H^1)} \leq C(\varphi) D\sqrt L.
\]

We deduce
\[
|\mathcal A(\xi)| \leq C(\varphi)D\sqrt L \|{\bf 1}_{B(t,L,D)^c}\|_{L^3} |\varphi(\xi)|.
\]
We have 
\[
\|{\bf 1}_{B(t,L,D)^c}\|_{L^3} = \mu_L^{1/3}(B(t,L,D)^{c}) \leq M(t,L,D)^{1/3} e^{-cD^2}.
\]
(We change $c$ to $c/3$.)

Therefore, we have 
\[
|\mathcal A(\xi)| \leq C(\varphi) D\sqrt L M(t,L,D)^{1/3} e^{-cD^2}|\varphi(\xi)|.
\]

By definition of $\mathcal B$, \eqref{Bfin}, we have 
\[
\mathcal B(\xi) = \E\Big[{\bf 1}_{B(t,L,D)} \overline{\hat \varphi_L(\xi)}\Big(\hat \psi_{\varepsilon,L}(t\varepsilon^{-2})(\varphi_L)(\xi)- \widehat{\tilde \psi}_{\varepsilon,L}(t\varepsilon^{-2})(\varphi_L)(\xi)\Big) \Big].
\]
We have 
\[
|\mathcal B(\xi)|
\leq  \E({\bf 1}_{B(t,L,D)} |\hat \varphi_L(\xi)|\,\| \psi_{\varepsilon,L}(t\varepsilon^{-2})(\varphi_L)- \tilde \psi_{\varepsilon,L}(t\varepsilon^{-2})(\varphi_L)\|_{H^{s'}}).
\]

We deduce by Proposition \ref{prop:psiandtildepsiPropagated}
\[
|\mathcal B(\xi)|
 \leq 
2^{t/T_2} \Big[ D\sqrt L (2^{-N} + \sum_{n\geq N/2}\frac1{n!} + 3R^{-(s-s')}) + \varepsilon C(L,D,R,N)\Big] \E(|\varphi_L(\xi)|) 
\]
that is
\[
|\mathcal B(\xi)|
 \leq 
2^{t/T_2} \Big[ D\sqrt L (2^{-N} + \sum_{n\geq N/2}\frac1{n!} + 3R^{-(s-s')}) + \varepsilon C(L,D,R,N)\Big] |\varphi(\xi)|.
\]

By definition of $\mathcal C$, \eqref{Cfin}, we have
\[
\mathcal C(\xi)
= - \E({\bf 1}_{B(t,L,D)^c} \overline{\hat \varphi_L(\xi)}\widehat{\tilde \psi}_{\varepsilon,L}(t\varepsilon^{-2})(\varphi_L)(\xi))
.
\]
We have 
\[
|\mathcal C(\xi)|\leq \|{\bf 1}_{B(t,L,D)^c}\|_{L^3(\Omega)}\|\hat \varphi_L(\xi)\|_{L^3(\Omega)} \|\widehat{\tilde \psi}_{\varepsilon,L}(t\varepsilon^{-2})(\varphi_L)(\xi))\|_{L^3(\Omega)}.
\]
Because 
\[
\widehat{\tilde \psi}_{\varepsilon,L}(t\varepsilon^{-2})(\varphi_L)(\xi))
\]
only differs from $\hat \varphi_L(\xi)$ by a phase we get
\[
|\mathcal C(\xi)| \leq 3|\varphi(\xi)|^2 M(t,L,D)^{1/3}e^{-cD^2}.
\]

We estimate $\mathcal D$, we have 
\[
\E( \overline{\hat \varphi_L(\xi)}\widehat{\tilde \Psi}_{\varepsilon,L}(t\varepsilon^{-2})(\varphi_L)(\xi)) = \E(e^{i U(\varphi_L,\xi) t} |\hat \varphi_L(\xi)|^2),
\]
we deduce by expanding the exponential
\[
\E( \overline{\hat \varphi_L(\xi)}\widehat{\tilde \Psi}_{\varepsilon,L}(t\varepsilon^{-2})(\varphi_L)(\xi)) = \sum_n \E(\frac{(i U(\varphi_L,\xi)t)^n}{n!} |\hat \varphi_L(\xi)|^2).
\]

Let 
\[
\mathcal D_n := \E((U(\varphi_L,\xi))^n |\hat \varphi_L(\xi)|^2).
\]
We recall that
\[
U(a,\xi) = \frac2{\pi L}\sum_\eta \tilde F(\xi,\eta) |\hat a(\eta)|^2
\]
with 
\[
\tilde F(\xi,\eta) = \frac{(1+\eta^2)\xi}{(3 + \frac{\eta^2 + \xi^2 + (\xi-\eta)^2}{2})(3 + \frac{\eta^2 + \xi^2 + (\xi+\eta)^2}{2})}.
\]
Note that
\[
\tilde F(\xi,-\eta) = \tilde F(\xi,\eta).
\]
We deduce
\[
 U(\varphi_L,\xi)^n = \frac{4^n}{(\pi L)^n} \sum_{\eta_1,\hdots,\eta_n>0}\prod_{l=1}^n \tilde F(\xi,\eta_l) |\hat \varphi_L(\eta_l)|^2. 
\]
Set $\eta_0 = |\xi|$ and for any partition $P$ of $[|0,n|]$ written 
\[
[|0,n|] = \bigsqcup_{j=0}^J P_j
\]
set
\[
B_P = \{(\eta_1,\hdots,\eta_n)\in \frac1{L}(\N^*)^n\; |\; \forall l,l'\in P_j,\; \eta_l = \eta_{l'},\; \wedge \; \forall j\neq j', \, l\in P_j,\, l'\in P_{j'},\, \eta_l \neq \eta_{l'}\}.
\]
We set accordingly
\[
U_P(\varphi_L,\xi) = \frac{4^n}{(\pi L)^n} \sum_{B_P}\prod_{l=1}^n \tilde F(\xi,\eta_l) |\hat \varphi_L(\eta_l)|^2.
\]
We have, since $\E(|g_n|^{2p} ) = p!$, 
\[
\E(U_P(\varphi_L,\xi)|\hat \varphi_L(\xi)|^2) = \frac{4^n}{(\pi L)^n} \sum_{B_P}\prod_{l=1}^n \tilde F(\xi,\eta_l) |\varphi(\eta_l)|^2 |\varphi(\xi)|^2 \prod_{j=0}^J (\# P_j)!.
\]
We assume that $P_0$ contains $0$. Because the sum is symmetric in the $\eta_j$, we have 
\begin{multline*}
\E(U_P(\varphi_L,\xi)|\hat \varphi_L(\xi)|^2) \\
= \frac{4^n}{(\pi L)^n} \sum_{(\eta_j)_{j=1,\hdots,J}\in \N_{L,J,\neq}}\prod_{j=1}^{J}  \tilde F(\xi,\eta_j)^{\# P_j}|\varphi(\eta_j)|^{\# P_j} \tilde F(\xi,\xi)^{\# P_0 - 1}|\varphi(\eta_j)|^{2\# P_0} \prod_{j=0}^J (\# P_j)!
\end{multline*}
with 
\[
\N_{L,J,\neq} = \{(\eta_j)_{j=1,\hdots,J} \in \frac1{L}(\N^*)^J \; |\; \forall i\neq j, \, \eta_i\neq \eta_j\}.
\]

The numbers of partitions of $[|0,n|]$ such that 
\[
[|0,n|] = \bigsqcup_{j=0}^J P_j
\]
with $0 \in P_0$ and $\# P_j = n_j>0$ is
\[
\frac{n!}{(n_0-1)! n_1! ... n_J!} \frac1{J!}.
\]
Indeed, we take $n_0-1$ elements among $n$, $n_1$ among $n-n_0+1$ etc, which explains
\[
\frac{n!}{(n_0-1)! n_1! ... n_J!}
\]
and the union is symmetric in $1,\hdots, J$ which explains the 
\[
\frac1{J!}.
\]
We deduce that 
\begin{multline*}
\mathcal D_n
= \frac{4^n}{(\pi L)^n}\sum_{J=0}^n \frac{n!}{J!} \sum_{n_0+\hdots + n_J = n+1, n_j>0}n_0 \sum_{(\eta_j)_{j=1,\hdots,J}\in \N_{L,J,\neq}}\prod_{j=1}^{ J}  \tilde F(\xi,\eta_j)^{n_j}|\varphi(\eta_j)|^{2n_j} \tilde F(\xi,\xi)^{n_0 - 1}|\varphi(\xi)|^{2n_0} .
\end{multline*}

Set 
\[
\mathcal D_{J,n}= \frac{4^n}{(\pi L)^n} \frac{n!}{J!} \sum_{n_0+\hdots + n_J = n+1, n_j>0}n_0 \sum_{(\eta_j)_{j=1,\hdots,J}\in \Z_{L,J,\neq}}\prod_{j=1}^{ J}  \tilde F(\xi,\eta_j)^{n_j}|\varphi(\eta_j)|^{2n_j} \tilde F(\xi,\xi)^{n_0 - 1}|\varphi(\eta_j)|^{2n_0} 
\]
such that
\[
\mathcal D_n  = \sum_{J\leq n} \mathcal D_{J,n}.
\]

By definition, we have 
\[
\mathcal D_{n,n}= \frac{4^n}{(\pi L)^n}  \sum_{(\eta_j)_{j=1,\hdots,n}\in \N_{L,n,\neq}}\prod_{j=1}^{ n} \tilde F(\xi,\eta_j)|\varphi(\eta_j)|^2 |\varphi(\xi)|^2.
\]
We deduce that 
\[
\mathcal D_{n,n} = \Big(\frac{4}{\pi L}  \sum_{\eta} \tilde F(\xi,\eta)|\varphi(\eta)|^2\Big)^n |\varphi(\xi)|^2 + \widetilde{\mathcal D}_{n}
\]
with
\[
\widetilde{\mathcal D}_{n} =  \frac{4^n}{(\pi L)^n}  \sum_{(\eta_j)_{j=1,\hdots,n}\in \N_{L,n,\neq}^c}\prod_{j=1}^{ n} \tilde F(\xi,\eta_j)|\varphi(\eta_j)|^2 |\varphi(\xi)|^2.
\]

We deduce
\[
\mathcal D(\xi) = \sum_{n}\frac{(it)^n}{n!} \Big( \mathcal D_{n} - \Big(\frac{4}{\pi L}  \sum_{\eta} \tilde F(\xi,\eta)|\varphi(\eta)|^2\Big)^n |\varphi(\xi)|^2\Big)
\]
and thus
\[
\mathcal D(\xi) = \sum_n \frac{(it)^n}{n!} \Big( \sum_{J< n} \mathcal D_{J,n} + \widetilde{\mathcal D}_n\Big).
\]
It remains to estimate
\[
\sum_J \sum_{n>J} \frac{t^n}{n!} |\mathcal D_{J,n}|
\]
and
\[
\sum_n \frac{t^n}{n!} \widetilde{\mathcal D}_n|.
\]

We have 
\[
\mathcal D_{J,n}= \frac{4^n}{(\pi L)^n} \frac{n!}{J!} \sum_{n_0+\hdots + n_J = n+1, n_j>0}n_0 \sum_{(\eta_j)_{j=1,\hdots,J}}\prod_{j=1}^{ J}  \tilde F(\xi,\eta_j)^{n_j}|\varphi(\eta_j)|^{2n_j} \tilde F(\xi,\xi)^{n_0 - 1}|\varphi(\xi)|^{2n_0} .
\]
We deduce
\begin{multline*}
\sum_{n>J} |\mathcal D_{J,n}|\frac{t^n}{n!} 
\leq  \\
\frac1{J!} \sum_{n_0+\hdots+n_J> J} n_0 \sum_{(\eta_j)_{j=1,\hdots,J}}\prod_{j=1}^{ J} \frac{4^{n_j}}{(\pi L)^{n_j}} t^{n_j}|\tilde F(\xi,\eta_j)|^{n_j}|\varphi(\eta_j)|^{2n_j} |\tilde F(\xi,\xi)|^{n_0-1 }t^{n_0-1}|\varphi(\xi)|^{2n_0} \frac{4^{n_0-1}}{(\pi L)^{n_0-1}}
\end{multline*}
and thus by the change of variable $k_j=n_j-1$ we get
\begin{multline*}
J!\sum_{n>J} |\mathcal D_{J,n}|\frac{t^n}{n!} \leq\\
 \sum_{k_0+\hdots+k_J\geq 1,(\eta_j)_{j=1,\hdots,J}} (k_0 +1)\prod_{j=1}^{ J} \frac{4^{k_j+1}}{(\pi L)^{k_j+1}} (t|\tilde F(\xi,\eta_j)|)^{k_j+1}|\varphi(\eta_j)|^{2(k_j+1)} (t |\tilde F(\xi,\xi)|)^{k_0}|\varphi(\xi)|^{2(k_0+1)}\frac{4^{k_0}}{(\pi L)^{k_0}}.
\end{multline*}
We factorize by $t^J$ and get
\begin{multline*}
J! t^{-J}\sum_{n>J}| \mathcal D_{J,n}|\frac{t^n}{n!} \leq \\
 \sum_{k_0+\hdots+k_J\geq 1} (k_0 +1)\sum_{(\eta_j)_{j=1,\hdots,J}}\prod_{j=1}^{ J} \frac{4^{k_j+1}}{(\pi L)^{k_j+1}} t^{k_j}|\tilde F(\xi,\eta_j)|^{k_j+1}|\varphi(\eta_j)|^{2(k_j+1)}  t^{k_0}|\tilde F(\xi,\xi)|^{k_0}|\varphi(\xi)|^{2(k_0+1)}\frac{4^{k_0}}{(\pi L)^{k_0}}.
\end{multline*}
We use that $\tilde F, \varphi$ are bounded and that
\[
\frac4{\pi L} \sum_{\eta} |\varphi(\eta)|^2 \leq C(\varphi)
\]
to get
\[
\sum_{n>J} |\mathcal D_{J,n}|\frac{t^n}{n!} \leq \frac{(Ct)^j}{J!} \sum_{k_0+\hdots+k_J\geq 1} (k_0 +1)\prod_{j=0}^{ J} \frac{(4Ct)^{k_j}}{(\pi L)^{k_j}} |\varphi(\xi)|^2.
\]

Because $k_0+\hdots+k_J \geq 1$, at least one of the $k_j$ is bigger than $1$, we get
\begin{multline*}
\sum_{n>J} |\mathcal D_{J,n}|\frac{t^n}{n!} 
\leq \frac{(Ct)^j}{J!} \frac{4Ct}{(\pi L)}\sum_{k_0} (k_0 +2)\frac{(4Ct)^{k_0}}{(\pi L)^{k_0}}
\Big(
\sum_k \frac{(4Ct)^{k}}{(\pi L)^{k}}\Big)^J |\varphi(\xi)|^2 \\
+ J \frac{t^j}{J!} \frac{4Ct}{(\pi L)} \sum_{k_0} (k_0 +1)\frac{(4Ct)^{k_0}}{(\pi L)^{k_0}}\Big(
\sum_k \frac{(4Ct)^{k}}{(\pi L)^{k}}\Big)^J |\varphi(\xi)|^2 .
\end{multline*}
Therefore, if $L\geq \frac{8C}{\pi}t $, we get
\[
\sum_{n>J} |\mathcal D_{J,n}|\frac{t^n}{n!} \leq \frac{t^j}{J!}(J+1) \frac{4C t}{(\pi L)}\sum_{k_0} (k_0 +2)2^{-k_0}
2^J |\varphi(\xi)|^2 .
\]
We deduce that
\[
\sum_J \sum_{n>J} \frac{t^n}{n!} \mathcal D_{J,n} \leq C \frac{t(1+t)}{L}e^{2t} |\varphi(\xi)|^2.
\]

We have 
\[
|\widetilde{\mathcal D}_n|
\leq \frac{n(n-1)}{2}\frac2{\pi L} \Big( \frac2{\pi L} \sum_\eta |\varphi(\eta)|^2 \Big)^{n-2} \Big( \frac2{\pi L}\sum_\eta |\varphi(\eta)|^2\Big) |\varphi(\xi)|^2.
\]
We deduce 
\[
|\widetilde{\mathcal D}_n|\leq \frac{C^{n}}{L} |\varphi(\xi)|^2.
\]
Therefore,
\[
\sum_{n} |\widetilde{\mathcal D}_n| \frac{t^n}{n!} \leq  \frac{e^{Ct}}{L}|\varphi(\xi)|^2,
\]
that is
\[
|\mathcal D(\xi)|\leq   C \frac{t(1+t)}{L}e^{2t} |\varphi(\xi)|^2 + \frac{e^{Ct}}{L}|\varphi(\xi)|^2.
\]
Note that if $L \leq \frac{8C}{\pi} t$ then we use that
\[
|\mathcal D(\xi)| \leq 2 |\varphi(\xi)|^2 \leq \frac{16C}{\pi} \frac{t}{L} |\varphi(\xi)|^2 \leq C \frac{t(1+t)}{L}e^{2t}|\varphi(\xi)|^2 + \frac{e^{Ct}}{L} |\varphi(\xi)|^2.
\]

Because for all $\eta,\xi$, we have 
\[
\partial_\eta (\tilde F |\varphi(\eta)|^2) \leq C (|\varphi(\eta)|^2 + \frac{|\varphi(\eta)|}{\an{\eta}}),
\]
We deduce 
\[
\Big| \frac2{\pi L} \sum_\eta \tilde F(\xi,\eta) |\varphi(\eta)|^2 - \frac1{\pi}\int_{\R} \tilde F(\xi,\eta)|\varphi(\eta)|^2 d\eta\Big| \leq \frac{C}{L} .
\]

We get, by definition of $\mathcal E$, \eqref{Efin},
\[
|\mathcal E(\xi)| \leq \frac{Ct}{L} |\varphi(\xi)|^2.
\]

Summing up, we get
\begin{multline*}
|I(\xi) |
\leq  C \frac{t(1+t)}{L}e^{2t} |\varphi(\xi)|^2 + \frac{e^{Ct}}{L} |\varphi(\xi)|^2\\
+ C(\varphi)D\sqrt L M(t,L,D)^{1/3}e^{-cD^2} \\
+ 2^{t/T_2} \Big[ D\sqrt L (2^{-N} + \sum_{n\geq N/2}\frac1{n!} + 3R^{-(s-s')}) + \varepsilon C(L,D,R,N)\Big]  |\varphi(\xi)|.
\end{multline*}

Optimizing $N,D$ and $R$ in $\varepsilon,t$, we get that for $t\in \R_+$, $L\geq 1$, there exists $f_{t,L} : \R_+^* \rightarrow \R_+$ such that $f_{t,L}$ goes to $0$ at $0$ and 
\[
|I(\xi)|  \leq \frac{e^{C(1+t)}}{L}  |\varphi(\xi)|^2 + f_{L,t}(\varepsilon)  |\varphi(\xi)|.
\]

Noticing that
\[
\frac2{\pi }\int d\eta \tilde F(\xi,\eta) \frac1{1+\eta^2}\Big) = \frac1{\sqrt 3} \frac1{3+\xi^2} \frac1{\sqrt{1+\xi^2/4}},
\]
we get the result.

\end{proof}

\bibliographystyle{plain}
\bibliography{WTbib}

\end{document}